\documentclass[11pt]{article}

\usepackage{amsfonts}
\usepackage{amscd}
\usepackage{amssymb}
\usepackage{amsthm}
\usepackage{amsmath} 
\usepackage{hyperref}
\usepackage{pictexwd}

 \theoremstyle{plain}
\newtheorem{thm}{Theorem}[section]
\newtheorem{lemma}[thm]{Lemma}
\newtheorem{prop}[thm]{Proposition}

\newtheorem{remark}[thm]{Remark}
\theoremstyle{definition}

\numberwithin{equation}{section}

\def\cB{\mathcal{B}}

\def\TL{{T\!L}}

\def\CC{\mathbb{C}}

\def\ZZ{\mathbb{Z}}

\def\fg{\mathfrak{g}}

\def\fsl{\mathfrak{sl}}  
\def\fgl{\mathfrak{gl}}

\def\End{\mathrm{End}} 
 
\def\Hom{\mathrm{Hom}} 
 
\def\id{\mathrm{id}}

\makeatletter
\renewcommand{\@makefnmark}{\mbox{\textsuperscript{}}}
\makeatother

\title{Commuting families in\\
Hecke and Temperley-Lieb Algebras}
\author{
Tom Halverson \\ Department of Mathematics \\ Macalester College \\
Saint Paul, MN 55105 \\ halverson@macalester.edu
\and
Manuela Mazzocco \\ Department of Mathematics\\
University of Wisconsin\\ Madison, WI 53706 \\
\and 
Arun Ram \\
Department of Mathematics\\ University of Wisconsin\\
Madison, WI 53706 \\ ram@math.wisc.edu \\
and \\
Department of Mathematics and Statistics \\
University of Melbourne \\
Parkville VIC 3010 Australia}
\date{}

\begin{document}

\maketitle

\begin{abstract}
We define analogs of the Jucys-Murphy elements for the affine Temperley-Lieb algebra
and give their explicit expansion in terms of the basis of planar Brauer diagrams.  
These Jucys-Murphy elements are a family of commuting elements in the affine Temperley-Lieb algebra, and
we compute their eigenvalues on the generic irreducible representations.  We show that they 
come from  Jucys-Murphy elements in  the affine Hecke algebra of type A, which
in turn come from the Casimir element of the quantum group $U_h\fgl_n$.
We also give the explicit specializations of these results to the finite Temperley-Lieb algebra.
\end{abstract}


\begin{section}{Introduction}%%%%%%%%%%%%%%%%%%%%%%%%%%%%%%

The ``Jucys-Murphy elements'' are a family of commuting elements in the group algebra
of the symmetric group.  In characteristic 0, these elements have enough distinct eigenvalues
to give a full analysis of the representation theory of the
symmetric group \cite{OV}.  Even in positive characteristic these elements are powerful tools
\cite{K}.  Similar elements are used in the Hecke algebras of type A and, in a strong sense,
it is these elements that control the beautiful connections between the modular representation
theory of Hecke algebras of type A and the Fock space representations of the affine quantum group
(see \cite{Ar} and \cite{Gr}).

Since the Temperley-Lieb algebra is a quotient of the Hecke algebra of type A it inherits a commuting family of elements from the Hecke algebra.   In order to use these elements for modular representation theory it is important to have good control of the expansion in terms of the standard 
basis of planar Brauer diagrams.  In this paper we study this question, in the more general setting
of the affine Temperley-Lieb algebras.
Specifically, we analyze a convenient choice of a
commuting family of elements in the affine Temperley-Lieb algebra.  Our main result,
Theorem \ref{ExpansionInDiagrams},
is an explicit expansion of these elements in the standard basis.  The fact that,
in the Templerley-Lieb algebra, these
elements have integral coefficients is made explicit in Remark \ref{integrality}.  The import of this
result is that this commuting family can be used to attack questions in modular representation 
theory.

In Section 3 we review the Schur-Weyl duality setup of Orellana and Ram \cite{OR} which 
(following the ideas in \cite{Re}) explains how commuting families in centralizer algebras 
arise naturally from Casimir elements.   We explain, in detail, the cases that lead to commuting 
families in the affine Hecke algebras of type A and the affine Temperley-Lieb algebra.  
One new consequence of our analysis is an explanation of the ``special'' relation that
is used in one of the Temperley-Lieb algebras of Graham and Lehrer \cite{GL4}.  In our context,
this relation appears naturally from the Schur-Weyl duality 
(see Proposition \ref{SchurWeylAffineTL}).
Using the knowledge of eigenvalues of Casimir elements we compute the eigenvalues of the 
commuting families in the affine Hecke algebra and in the affine Temperley-Lieb algebra
in the generic irreducible representations (analogues of the Specht, or Weyl, modules).

The research of M.\ Mazzocco was partially supported by a Mark Mensink Honors 
Research Award at the University of Wisconsin, Madison.  She was an undergraduate
researcher participating in research and teaching initiatives partly supported by the National
Science Foundation under grant DMS-0353038.  The research of A.\ Ram was also
partially supported by this award.  A significant portion of this research was done during a residency 
of A.\ Ram at the  Max-Planck-Institut f\"ur Mathematik (MPI) in Bonn.  He thanks the MPI for support,
hospitality, and a wonderful working environment.  The research of T.\ Halverson was partially supported  by the National Science Foundation under grant DMS-0100975. 

\end{section}

\begin{section}{Affine braid groups, Hecke and Temperley-Lieb algebras}%%%%%%%%%%%%%%%%%%%

\subsection{The Affine Braid Group $\tilde {\cal B}_k$ }

The \emph{affine braid group} is the group
$\tilde {\cal B}_k$ of affine braids with $k$ strands (braids with a flagpole).
The group $\tilde {\cal B}_k$ is presented
by generators $T_1,T_2,\ldots,T_{k-1}$ and $X^{\varepsilon_1}$,
\begin{equation}\label{AffineBraidGenerators}
T_i = 
\beginpicture
\setcoordinatesystem units <.5cm,.5cm>         % sets scale
\setplotarea x from -5 to 3.5, y from -2 to 2    % sets plot size up
\put{${}^i$} at 0 1.2      %
\put{${}^{i+1}$} at 1 1.2      %
\put{$\bullet$} at -3 .75      %
\put{$\bullet$} at -2 .75      %
\put{$\bullet$} at -1 .75      %
\put{$\bullet$} at  0 .75      %   Top dots
\put{$\bullet$} at  1 .75      %
\put{$\bullet$} at  2 .75      %
\put{$\bullet$} at  3 .75      %   
\put{$\bullet$} at -3 -.75          %
\put{$\bullet$} at -2 -.75          %
\put{$\bullet$} at -1 -.75          %
\put{$\bullet$} at  0 -.75          %  Bottom dots
\put{$\bullet$} at  1 -.75          %
\put{$\bullet$} at  2 -.75          %
\put{$\bullet$} at  3 -.75          %
%Flagpole
\plot -4.5 1.25 -4.5 -1.25 /
\plot -4.25 1.25 -4.25 -1.25 /
\ellipticalarc axes ratio 1:1 360 degrees from -4.5 1.25 center 
at -4.375 1.25
\put{$*$} at -4.375 1.25  
\ellipticalarc axes ratio 1:1 180 degrees from -4.5 -1.25 center 
at -4.375 -1.25 
% Vertical edges
\plot -3 .75  -3 -.75 /
\plot -2 .75  -2 -.75 /
\plot -1 .75  -1 -.75 /
%\plot  0 .75   0 -.75 /
%\plot  1 .75   1 -.75 /
\plot  2 .75   2 -.75 /
\plot  3 .75   3 -.75 /
\setquadratic
% single crossing
\plot  0 -.75  .05 -.45  .4 -0.1 /
\plot  .6 0.1  .95 0.45  1 .75 /
\plot 0 .75  .05 .45  .5 0  .95 -0.45  1 -.75 /
\endpicture
\qquad\hbox{and}\qquad
X^{\varepsilon_1} = 
~~\beginpicture
\setcoordinatesystem units <.5cm,.5cm>         % sets scale
\setplotarea x from -5 to 3.5, y from -2 to 2    % sets plot size up
\put{$\bullet$} at -3 0.75      %
\put{$\bullet$} at -2 0.75      %
\put{$\bullet$} at -1 0.75      %
\put{$\bullet$} at  0 0.75      %   Top dots
\put{$\bullet$} at  1 0.75      %
\put{$\bullet$} at  2 0.75      %
\put{$\bullet$} at  3 0.75      %   
\put{$\bullet$} at -3 -0.75          %
\put{$\bullet$} at -2 -0.75          %
\put{$\bullet$} at -1 -0.75          %
\put{$\bullet$} at  0 -0.75          %  Bottom dots
\put{$\bullet$} at  1 -0.75          %
\put{$\bullet$} at  2 -0.75          %
\put{$\bullet$} at  3 -0.75          %
%Flagpole
\plot -4.5 1.25 -4.5 -0.13 /
\plot -4.5 -0.37   -4.5 -1.25 /
\plot -4.25 1.25 -4.25  -0.13 /
\plot -4.25 -0.37 -4.25 -1.25 /
\ellipticalarc axes ratio 1:1 360 degrees from -4.5 1.25 center 
at -4.375 1.25
\put{$*$} at -4.375 1.25  
\ellipticalarc axes ratio 1:1 180 degrees from -4.5 -1.25 center 
at -4.375 -1.25 
% Vertical edges
%\plot -3 0.75  -3 -0.75 /
\plot -2 0.75  -2 -0.75 /
\plot -1 0.75  -1 -0.75 /
\plot  0 0.75   0 -0.75 /
\plot  1 0.75   1 -0.75 /
\plot  2 0.75   2 -0.75 /
\plot  3 0.75   3 -0.75 /
\setlinear
%curve around pole
\plot -3.3 0.25  -4.1 0.25 /
\ellipticalarc axes ratio 2:1 180 degrees from -4.65 0.25  center 
at -4.65 0 
\plot -4.65 -0.25  -3.3 -0.25 /
\setquadratic
%\plot  -3.8 0.25  -3.7 0.2 -3.5 0  -3.05 -0.45  -3 -0.75 /
\plot  -3.3 0.25  -3.05 .45  -3 0.75 /
\plot  -3.3 -0.25  -3.05 -0.45  -3 -0.75 /
\endpicture
\end{equation},
with relations
\begin{equation} \label{AffineBraidRelations}
\begin{array}{ll} 
X^{\varepsilon_1}T_1X^{\varepsilon_1}T_1
=T_1X^{\varepsilon_1}T_1X^{\varepsilon_1} \\
X^{\varepsilon_1}T_i = T_iX^{\varepsilon_1},
&\hbox{for $i>1$,} \\
T_{i}T_{j}=T_{j}T_{i}, &\hbox{if $\left| i-j\right| >1$}, \\
T_{i}T_{i + 1}T_{i}=T_{i+1}T_{i}T_{i+1}, &\hbox{if $1\leq i\leq k-2$,} 
\end{array}
\end{equation}

For $1\le i\le k$ define
\begin{equation}\label{BraidMurphy}
X^{\varepsilon_i}=T_{i-1}T_{i-2}\cdots T_2T_1 
X^{\varepsilon_1}T_1T_2\cdots T_{i-1} = 
~~\beginpicture
\setcoordinatesystem units <.5cm,.5cm>         % sets scale
\setplotarea x from -5 to 3.5, y from -2 to 2    % sets plot size up
\put{${}^i$} at 1 1.2 
\put{$\bullet$} at -3 0.75      %
\put{$\bullet$} at -2 0.75      %
\put{$\bullet$} at -1 0.75      %
\put{$\bullet$} at  0 0.75      %   Top dots
\put{$\bullet$} at  1 0.75      %
\put{$\bullet$} at  2 0.75      %
\put{$\bullet$} at  3 0.75      %   
\put{$\bullet$} at -3 -0.75          %
\put{$\bullet$} at -2 -0.75          %
\put{$\bullet$} at -1 -0.75          %
\put{$\bullet$} at  0 -0.75          %  Bottom dots
\put{$\bullet$} at  1 -0.75          %
\put{$\bullet$} at  2 -0.75          %
\put{$\bullet$} at  3 -0.75          %
%Flagpole
\plot -4.5 1.25 -4.5 -0.13 /
\plot -4.5 -0.37   -4.5 -1.25 /
\plot -4.25 1.25 -4.25  -0.13 /
\plot -4.25 -0.37 -4.25 -1.25 /
\ellipticalarc axes ratio 1:1 360 degrees from -4.5 1.25 center 
at -4.375 1.25
\put{$*$} at -4.375 1.25  
\ellipticalarc axes ratio 1:1 180 degrees from -4.5 -1.25 center 
at -4.375 -1.25 
% Vertical edges
\plot -3 0.75  -3 -0.1 /
\plot -2 0.75  -2 -0.1 /
\plot -1 0.75  -1 -0.1 /
\plot  0 0.75   0 -0.1 /
\plot -3 -.35  -3 -0.75 /
\plot -2 -.35   -2 -0.75 /
\plot -1 -.35   -1 -0.75 /
\plot  0 -.35    0 -0.75 /
\plot  2 0.75   2 -0.75 /
\plot  3 0.75   3 -0.75 /
\setlinear
%curve around pole
\plot -3.2 0.25  -4.1 0.25 /
\plot -2.8 0.25  -2.2 0.25 /
\plot -1.8 0.25  -1.2 0.25 /
\plot -.8 0.25  -.2 0.25 /
\plot  .2 0.25  .5 0.25 /
\plot -3.3 -.25  .5 -.25 /
\ellipticalarc axes ratio 2:1 180 degrees from -4.65 0.25  center 
at -4.65 0 
\plot -4.65 -0.25  -3.3 -0.25 /
\setquadratic
%\plot  -3.8 0.25  -3.7 0.2 -3.5 0  -3.05 -0.45  -3 -0.75 /
\plot  .5 0.25  .9 .45  1 0.75 /
\plot  .5  -0.25  .9 -0.45 1 -0.75 /
\endpicture.
\end{equation}
By drawing pictures of the corresponding affine braids it is easy to check that 
\begin{equation}\label{XiCommuteInAffine}
X^{\varepsilon_i}X^{\varepsilon_j}=X^{\varepsilon_j}X^{\varepsilon_i},
\quad\hbox{for $1\le i,j\le k$,}
\end{equation} 
so that the elements $X^{\varepsilon_1}, \ldots, X^{\varepsilon_k}$ are a commuting family
for $\tilde{\cal B}_k$.  Thus
$X=\langle X^{\varepsilon_i}\ |\ 1\le i\le k\rangle$
is an abelian subgroup of $\tilde {\cal B}_k$.  The free abelian group generated by
$\varepsilon_1,\ldots, \varepsilon_k$ is $\ZZ^k$ and 
\begin{equation}
X=\{ X^\lambda \ |\ \lambda\in \ZZ^k\}
\quad\hbox{where}\quad
X^\lambda=(X^{\varepsilon_1})^{\lambda_1}
(X^{\varepsilon_2})^{\lambda_2} \cdots 
(X^{\varepsilon_k})^{\lambda_k},
\end{equation}
for $\lambda = \lambda_1\varepsilon_1+\cdots+\lambda_k\varepsilon_k$ in $\ZZ^k$.

\begin{remark}\label{AlternateBk}
An alternate presentation of $\tilde\cB_k$ can be given using the generators
$T_0,T_1,\ldots, T_{k-1}$ and $\tau$ where
$$\tau = X^{-\varepsilon_1}T_1^{-1}\cdots T_{k-1}^{-1} = {\beginpicture
\setcoordinatesystem units <0.35cm,0.3cm>         % sets scale
\setplotarea x from -.5 to 6, y from -2.5 to 3.5    % sets plot size up
\linethickness=0.5pt                          % sets line thickness
\put{$\bullet$} at 1 -1 \put{$\bullet$} at 1 2
\put{$\bullet$} at 2 -1 \put{$\bullet$} at 2 2
\put{$\bullet$} at 3 -1 \put{$\bullet$} at 3 2
\put{$\bullet$} at 4 -1 \put{$\bullet$} at 4 2
\put{$\bullet$} at 5 -1 \put{$\bullet$} at 5 2
%\put{$\scriptstyle{i}$} at 2.9 3.5 
\put{$\bullet$} at 6 -1 \put{$\bullet$} at 6 2
% Pole
\plot 0 3.5 0 -2.5 /
\plot .25 3.5 .25 -2.5 /
\ellipticalarc axes ratio 1:1 360 degrees from 0 3.5 center 
at .124 3.6
\put{$*$} at  .124 3.6  
\ellipticalarc axes ratio 1:1 220 degrees from 0 -2.5 center 
at .124  -2.6 
% edges
\plot 1 -1 2 2  /
\plot 2 -1 3 2  /
\plot 3 -1 4 2  /
\plot 4 -1 5 2  /
\plot 5 -1 6 2 /
\setquadratic
\plot 1 2 -.5 1.5 -.2 .5  /
\plot .5 .2 .6 .1 1.1 .1 /
\plot 1.6 .1 1.8 .1 2.0 .1 /
\plot 2.6 .1 2.8 .1 3.2 .1 /
\plot 3.5 .1 3.7 .1 4.1 0 /
\plot 4.5 0 4.6 0 5.0 -.1 /
\plot 5.5 -.3 5.6 -.4 6 -1 /
\endpicture}
\qquad\hbox{and}\qquad
T_0=\tau^{-1}T_1\tau = 
{\beginpicture
\setcoordinatesystem units <0.35cm,0.3cm>         % sets scale
\setplotarea x from -.5 to 6, y from -2.5 to 3.5    % sets plot size up
\linethickness=0.5pt                          % sets line thickness
\put{$\bullet$} at 1 -1 \put{$\bullet$} at 1 2
\put{$\bullet$} at 2 -1 \put{$\bullet$} at 2 2
\put{$\bullet$} at 3 -1 \put{$\bullet$} at 3 2
\put{$\bullet$} at 4 -1 \put{$\bullet$} at 4 2
\put{$\bullet$} at 5 -1 \put{$\bullet$} at 5 2
%\put{$\scriptstyle{i}$} at 2.9 3.5 
\put{$\bullet$} at 6 -1 \put{$\bullet$} at 6 2
% Pole
\plot 0 3.5 0 -2.5 /
\plot .25 3.5 .25 -2.5 /
\ellipticalarc axes ratio 1:1 360 degrees from 0 3.5 center 
at .124 3.6
\put{$*$} at  .124 3.6  
\ellipticalarc axes ratio 1:1 220 degrees from 0 -2.5 center 
at .124  -2.6 
% edges
\plot 2 -1 2 2  /
\plot 3 -1 3 2  /
\plot 4 -1 4 2  /
\plot 5 -1 5 2  /
\setquadratic
\plot 1 2 -.5 1.6 -.2 1.1 /

\plot .5 .85   .825 .7  1.15 .45      / 
\plot  1.15 .45  1.5 .2      1.8 .1    / 

\plot .5 .1     .825  .2    1.0 .3  /
\plot 1.3 .55  1.5   .7     1.8 .8 /

%\plot .5 .1 1 .2 1.8 .1 /

\plot 2.2 .9 2.4 .9 2.8 1.0 /
\plot 3.2 1.0 3.4 1.0 3.8 1.1 /
\plot 4.2 1.1 4.4 1.1 4.8 1.2 /
\plot 5.2 1.3 5.4 1.4 6 2 /
%\plot .5 .9  4 1.0 6 2 /
%\plot .5 .1  4 0.0 6 -1 /
\plot 1 -1 -.5 -.5 -.2 0 /

\plot 2.2 .1 2.4 .1 2.8 .1 /
\plot 3.2 .1 3.4 .1 3.8 0 /
\plot 4.2 0 4.4 -.07 4.8 -.2  /
\plot 5.2 -.3 5.4 -.4 6 -1 /
\endpicture}.
$$
\end{remark}

\begin{remark}\label{SLPGL}  
The affine braid group $\tilde\cB_k$ is the affine braid group of type $GL_k$.
The affine braid groups of type $SL_k$ and $PGL_k$ are the subgroup
$$\tilde \cB_Q = \langle T_0,T_1,\ldots, T_{k-1}\rangle
\qquad\hbox{and the quotient}\qquad
\tilde \cB_P = \frac{\tilde\cB_k}{\langle \tau^k\rangle},\qquad\hbox{respectively.}
$$
Then $\tau^k = X^{-\varepsilon_1}X^{-\varepsilon_2}\cdots X^{-\varepsilon_k}$ is
a central element of $\tilde\cB_k$, 
$\tau T_i \tau^{-1} = T_{i+1}$ (where the indices are taken mod $n$), and 
$\tau X^{\varepsilon_i}\tau^{-1} = X^{\varepsilon_{i+1}}$
and
$$Z(\tilde\cB_k) = \langle \tau^k\rangle,\qquad
\tilde\cB_k = \langle \tau \rangle \ltimes \tilde\cB_Q,\qquad\qquad
\tilde \cB_P= \langle\bar \tau\rangle\ltimes \tilde\cB_Q.$$
In $\tilde\cB_k$ we have $\langle \tau \rangle \cong \ZZ$, and  $\bar \tau\in\tilde \cB_P$ is defined to be the image of  $\tau$ under the homomorphism $\ZZ\to \ZZ/k\ZZ$ so that $\langle \bar\tau \rangle \cong  \ZZ/k\ZZ$.
\end{remark}

\subsection{The Temperley-Lieb algebra $\TL_k(n)$}

A Temperley-Lieb diagram on $k$ dots  is a graph with  $k$ dots in the top row, $k$ dots in
the bottom row, and $k$ edges pairing the dots such that the graph is planar (without edge 
crossings).  
%If $T_k$ is the set of Temperley-Lieb diagrams on $k$ dots, then \cite[Lemma 2.8.4]{GHJ}
%$$\Card(T_k) = \frac{1}{k+1}   \binom{2k}{k}, \quad\hbox{the Catalan number,}
%$$
For example,
$$
d_1 = {\beginpicture
\setcoordinatesystem units <0.5cm,0.2cm> % sets scale
\setplotarea x from 1 to 7, y from 0 to 2    % sets plot size up
\linethickness=0.5pt
\put{$\bullet$} at 1 -1 \put{$\bullet$} at 1 2
\put{$\bullet$} at 2 -1 \put{$\bullet$} at 2 2
\put{$\bullet$} at 3 -1 \put{$\bullet$} at 3 2
\put{$\bullet$} at 4 -1 \put{$\bullet$} at 4 2
\put{$\bullet$} at 5 -1 \put{$\bullet$} at 5 2
\put{$\bullet$} at 6 -1 \put{$\bullet$} at 6 2
\put{$\bullet$} at 7 -1 \put{$\bullet$} at 7 2
\setquadratic
\plot 1 2  1.5 1 2 2 /
\plot 4 2 5.5 .5 7 2 /
\plot 5 2 5.5 1.25 6 2 /
\plot 2 -1 3.5 .5 5 -1 /
\plot 3 -1 3.5 -.25 4 -1 /
\plot 6 -1 6.5 0 7 -1 /
\setlinear
\plot 3 2 1 -1  /
\endpicture}
\qquad\hbox{and}\qquad
d_2 = {\beginpicture
\setcoordinatesystem units <0.5cm,0.2cm> % sets scale
\setplotarea x from 1 to 7, y from -1 to 2    % sets plot size up
\linethickness=0.5pt
\put{$\bullet$} at 1 2 \put{$\bullet$} at 1 -1
\put{$\bullet$} at 2 2 \put{$\bullet$} at 2 -1
\put{$\bullet$} at 3 2 \put{$\bullet$} at 3 -1
\put{$\bullet$} at 4 2 \put{$\bullet$} at 4 -1
\put{$\bullet$} at 5 2 \put{$\bullet$} at 5 -1
\put{$\bullet$} at 6 2 \put{$\bullet$} at 6 -1
\put{$\bullet$} at 7 2 \put{$\bullet$} at 7 -1
\setquadratic
\plot 3 2 3.5 1 4 2 /
\plot 5 2 5.5 1 6 2 /
\plot 5 -1 5.5 0 6 -1 /
\plot 1 -1  1.5 0 2 -1 /
\setlinear
\plot 1 2 3 -1 /
\plot 2 2 4 -1 /
\plot 7 2 7 -1 /
\endpicture}.
$$
are Temperley-Lieb diagrams on 7 dots.
The composition $d_1 \circ d_2$  of two diagrams $d_1, d_2 \in T_k$ is the diagram obtained by placing $d_1$ above $d_2$ and identifying the bottom vertices of $d_1$ with the top dots of $d_2$ removing any connected components that live entirely in the middle row.  If $T_k$ is the set of Temperley-Lieb
diagrams on $k$ dots then the Temperley-Lieb algebra 
$\TL_k(n)$ is the associative algebra with basis $T_k$,
$$
\TL_k(n) = \hbox{span}\{d \in T_k\} \qquad \hbox{with multiplication defined by}\qquad
d_1 d_2 = n^\ell (d_1 \circ d_2),
$$
where $\ell$ is the number of blocks removed from the middle row when constructing the composition $d_1 \circ d_2$ and $n$ is a fixed element of the base ring.  For example, using the diagrams $d_1$ and $d_2$ above, we have
$$
d_1 d_2 = 
\begin{array}{c}
 {\beginpicture
\setcoordinatesystem units <0.5cm,0.2cm> % sets scale
\setplotarea x from 1 to 7, y from 0 to 2    % sets plot size up
\linethickness=0.5pt
\put{$\bullet$} at 1 -1 \put{$\bullet$} at 1 2
\put{$\bullet$} at 2 -1 \put{$\bullet$} at 2 2
\put{$\bullet$} at 3 -1 \put{$\bullet$} at 3 2
\put{$\bullet$} at 4 -1 \put{$\bullet$} at 4 2
\put{$\bullet$} at 5 -1 \put{$\bullet$} at 5 2
\put{$\bullet$} at 6 -1 \put{$\bullet$} at 6 2
\put{$\bullet$} at 7 -1 \put{$\bullet$} at 7 2
\setquadratic
\plot 1 2  1.5 1 2 2 /
\plot 4 2 5.5 .5 7 2 /
\plot 5 2 5.5 1.25 6 2 /
\plot 2 -1 3.5 .5 5 -1 /
\plot 3 -1 3.5 -.25 4 -1 /
\plot 6 -1 6.5 0 7 -1 /
\setlinear
\plot 3 2 1 -1  /
\endpicture} \\
{\beginpicture
\setcoordinatesystem units <0.5cm,0.2cm> % sets scale
\setplotarea x from 1 to 7, y from -1 to 2    % sets plot size up
\linethickness=0.5pt
\put{$\bullet$} at 1 2 \put{$\bullet$} at 1 -1
\put{$\bullet$} at 2 2 \put{$\bullet$} at 2 -1
\put{$\bullet$} at 3 2 \put{$\bullet$} at 3 -1
\put{$\bullet$} at 4 2 \put{$\bullet$} at 4 -1
\put{$\bullet$} at 5 2 \put{$\bullet$} at 5 -1
\put{$\bullet$} at 6 2 \put{$\bullet$} at 6 -1
\put{$\bullet$} at 7 2 \put{$\bullet$} at 7 -1
\setquadratic
\plot 3 2 3.5 1 4 2 /
\plot 5 2 5.5 1 6 2 /
\plot 5 -1 5.5 0 6 -1 /
\plot 1 -1  1.5 0 2 -1 /
\setlinear
\plot 1 2 3 -1 /
\plot 2 2 4 -1 /
\plot 7 2 7 -1 /
\endpicture}
\end{array}
= n \ 
 {\beginpicture
\setcoordinatesystem units <0.5cm,0.2cm> % sets scale
\setplotarea x from 1 to 7, y from 0 to 2    % sets plot size up
\linethickness=0.5pt
\put{$\bullet$} at 1 -1 \put{$\bullet$} at 1 2
\put{$\bullet$} at 2 -1 \put{$\bullet$} at 2 2
\put{$\bullet$} at 3 -1 \put{$\bullet$} at 3 2
\put{$\bullet$} at 4 -1 \put{$\bullet$} at 4 2
\put{$\bullet$} at 5 -1 \put{$\bullet$} at 5 2
\put{$\bullet$} at 6 -1 \put{$\bullet$} at 6 2
\put{$\bullet$} at 7 -1 \put{$\bullet$} at 7 2
\setquadratic
\plot 1 2  1.5 1 2 2 /
\plot 4 2 5.5 .5 7 2 /
\plot 5 2 5.5 1.25 6 2 /
\plot 1 -1 1.5 -.25 2 -1 /
\plot 5 -1  5.5 -.5 6 -1 /
\plot 4 -1 5.5 0 7 -1 /
\setlinear
\plot 3 2 3 -1  /
\endpicture}.
$$
The algebra $\TL_k(n)$ is presented by generators
\begin{equation}\label{TLgenerator}
e_i = {\beginpicture
\setcoordinatesystem units <0.4cm,0.18cm>         % sets scale
\setplotarea x from 1 to 10, y from 0 to 3    % sets plot size up
\linethickness=0.5pt                          % sets line thickness
\put{$\bullet$} at 1 -1 \put{$\bullet$} at 1 2
\put{$\bullet$} at 2 -1 \put{$\bullet$} at 2 2
\put{$\cdots$} at 3 -1 \put{$\cdots$} at 3 2
\put{$\bullet$} at 4 -1 \put{$\bullet$} at 4 2
\put{$\bullet$} at 5 -1 \put{$\bullet$} at 5 2
\put{$\bullet$} at 6 -1 \put{$\bullet$} at 6 2
\put{$\scriptstyle{i}$} at 4.9 3.5 
\put{$\scriptstyle{i+1}$} at 6.1 3.5 
\put{$\bullet$} at 7 -1 \put{$\bullet$} at 7 2
\put{$\cdots$} at 8 -1 \put{$\cdots$} at 8 2
\put{$\bullet$} at 9 -1 \put{$\bullet$} at 9 2
\put{$\bullet$} at 10 -1 \put{$\bullet$} at 10 2
\plot 1 -1 1 2  /
\plot 2 -1 2 2  /
\plot 4 -1 4 2  /
\plot 7 -1 7 2  /
\plot  9 -1 9 2  /
\plot 10 -1 10 2 /
\setquadratic
\plot 5 -1 5.5 0 6 -1 /
\plot 5 2 5.5 1.0 6 2 /
\endpicture}, \qquad 1\le i\le k-1,\qquad\hbox{}
\end{equation}
\begin{equation}\label{TLpresentation}
\begin{array}{ll}
\hbox{and relations}\qquad e_{i}^{2}=ne_{i},
\qquad e_{i}e_{i\pm 1}e_{i}=e_{i}, \qquad \hbox{and} 
\qquad
e_{i}e_{j}=e_{j}e_{i},\ \ \hbox{if $\left| i-j\right| >1$} \\
\end{array}
\end{equation}
(see \cite[Lemma 2.8.4]{GHJ}).

\begin{remark}  In the definition of the Temperley-Lieb algebra, and for other algebras defined
in this paper, the base ring could be any one of several useful rings (e.g. $\CC$, $\CC(q)$, 
$\CC[[h]]$, $\ZZ[q,q^{-1}]$, $\ZZ[n]$ or localizations of these at special primes).  The most 
useful approach is to view the results of
computations as valid over any ring $R$ with $n,q,h\in R$ such that the formulas make sense.
\end{remark}

\subsection{The Surjection ${\tilde H}_k(q) \mapsto \TL_k(n)$}

The \emph{affine Hecke algebra} $\tilde H_k$ is the quotient of the group algebra of the affine braid group $\CC \tilde \cB_k$ by the relations
\begin{equation} \label{AffineHeckeAlgebra}
T_{i}^{2}=(q-q^{-1})T_{i}+1,
\qquad\qquad
\hbox{so that}\qquad
\CC\tilde\cB_k \longrightarrow \tilde H_k
\end{equation}
is a surjective homomorphism ($q$ is a fixed element of the base ring).
The affine Hecke algebra $\tilde H_k$ is the affine Hecke algebra of type $GL_k$.
The \emph{affine Hecke algebras of types $SL_k$ and $PGL_k$} are, respectively, 
the quotients ${\tilde H}_Q$ and ${\tilde H}_P$ of the 
group algebras of $\tilde \cB_Q$ and $\tilde \cB_P$ (see Remark \ref{SLPGL}) by the relations 
\eqref{AffineHeckeAlgebra}.

The \emph{Iwahori-Hecke algebra} is the subalgebra $H_k$ of $\tilde H_k$ generated
by $T_1,\ldots, T_{k-1}$.
In the Iwahori-Hecke algebra  $H_k$,  define
\begin{equation}  \label{TiFromEi}
e_i = q -  T_i, \qquad\hbox{for $i=1,2,\ldots, k-1$.}
\end{equation}
Direct calculations show that $e_i^2 = (q+q^{-1})e_i$ and that
$e_1e_2e_1 = e_1$ and $e_2e_1e_2 = e_2$ if and only if
\begin{equation}\label{TLKernel}
q^{3}- q^{2}T_1 - q^{2}T_2+q^{}T_1T_2+q^{}T_2T_1- T_1T_2T_1 = 0.
\end{equation}
Thus, setting $n = [2] = q+q^{-1}$, there are surjective algebra homomorphisms given by
\begin{equation}\label{AffineToTL}
\begin{matrix}
\psi:& \tilde H_k(q) &\longrightarrow &H_k(q) &\longrightarrow &\TL_k(n) \\
& X^{\varepsilon_1} &\longmapsto &1&\longmapsto &1  \\
& T_i &\longmapsto &T_i &\longmapsto &q-e_i.
 \end{matrix}
\end{equation}
The kernel of $\psi$ is generated by the element on the left hand side of equation
\eqref{TLKernel}.  In the notation of Theorem  \ref{Thm:AffineMurphy}, the representations of $H_k$ correspond 
to the case  when  $\mu= \emptyset$. Writing $\tilde H_k^{\lambda/\emptyset}$ as $\tilde H_k^\lambda$,
the element from \eqref{TLKernel} acts as 0  on the irreducible Iwahori-Hecke algebra modules  
${\tilde H}_3^{\beginpicture 
\setcoordinatesystem units <0.1cm,0.1cm> % sets scale
\setplotarea x from 0 to 3, y from 0 to 1   % sets plot size up
\linethickness=0.5pt                        % sets line thickness
\putrectangle corners at 0 0 and 1 1
\putrectangle corners at 1 0 and 2 1
\putrectangle corners at 2 0 and 3 1
\endpicture}$
and
${\tilde H}_3^{\beginpicture 
\setcoordinatesystem units <0.1cm,0.1cm> % sets scale
\setplotarea x from 0 to 2, y from -1 to 1   % sets plot size up
\linethickness=0.5pt                        % sets line thickness
\putrectangle corners at 0 0 and 1 1
\putrectangle corners at 1 0 and 2 1
\putrectangle corners at 0 -1 and 1 0
\endpicture}$, 
and (up to a scalar multiple) it is a projection onto ${\tilde H}_3^{\beginpicture 
\setcoordinatesystem units <0.1cm,0.1cm> % sets scale
\setplotarea x from 0 to 1, y from -1 to 1   % sets plot size up
\linethickness=0.5pt                        % sets line thickness
\putrectangle corners at 0 0 and 1 1
\putrectangle corners at 0 1 and 1 2
\putrectangle corners at 0 -1 and 1 0
\endpicture}$.  

\begin{remark}   There is an alternative surjective homomorphism that instead sends $T_i \mapsto e_i - q^{-1}$.  This alternative surjection has kernel generated by
$$q^{-3}+q^{-2}T_1+q^{-2}T_2+q^{-1}T_1T_2+q^{-1}T_2T_1+T_2T_1T_2.$$
This element is 0 on ${\tilde H}_3^{\beginpicture 
\setcoordinatesystem units <0.1cm,0.1cm> % sets scale
\setplotarea x from 0 to 1, y from -1 to 1   % sets plot size up
\linethickness=0.5pt                        % sets line thickness
\putrectangle corners at 0 0 and 1 1
\putrectangle corners at 0 1 and 1 2
\putrectangle corners at 0 -1 and 1 0
\endpicture}$
and
${\tilde H}_3^{\beginpicture 
\setcoordinatesystem units <0.1cm,0.1cm> % sets scale
\setplotarea x from 0 to 2, y from -1 to 1   % sets plot size up
\linethickness=0.5pt                        % sets line thickness
\putrectangle corners at 0 0 and 1 1
\putrectangle corners at 1 0 and 2 1
\putrectangle corners at 0 -1 and 1 0
\endpicture}$, 
and (up to a scalar multiple) it is a projection onto
${\tilde H}_3^{\beginpicture 
\setcoordinatesystem units <0.1cm,0.1cm> % sets scale
\setplotarea x from 0 to 3, y from 0 to 1   % sets plot size up
\linethickness=0.5pt                        % sets line thickness
\putrectangle corners at 0 0 and 1 1
\putrectangle corners at 1 0 and 2 1
\putrectangle corners at 2 0 and 3 1
\endpicture}$.
\end{remark}

\begin{remark} A priori, there are two different kinds of integrality for the Temperley-Lieb algebra:
coefficients in $\ZZ[n]$ or coefficients in $\ZZ[q,q^{-1}]$ (in terms of the basis of Temperley-Lieb 
diagrams).  The relation between these is as follows.
If
$$[2] = q+q^{-1} = n
\qquad\hbox{then}\qquad q = \frac{1}{2}(n+\sqrt{n^2-4}),
\quad q^{-1} = \frac{1}{2}(n-\sqrt{(n^2-4}),$$
since $q^2-nq+1=0$.  Then
$$[k] = \frac{q^k-q^{-k}}{q-q^{-1}}
= \frac{1}{2^{k-1}} \sum_{m=1}^{(k+1)/2} \binom{k}{2m-1} n^{k-2m+1}(n^2-4)^{m-1}$$
so that $[k]$ is a polynomial in $n$.
The polynomials
$$n^k = (q+q^{-1})^k
\quad\hbox{and}\quad
\{k\} = q^k+q^{-k}
\quad\hbox{and}\quad
[k] = \frac{q^k-q^{-k}}{q-q^{-1}},$$
all form bases of the ring $\CC[(q+q^{-1})]$.  The transition matrix $B$ between
the $[k]$ and the $\{k\}$ is triangular (with 1s on the diagonal)
and the transition matrix $C$ between the $n^k$ and the $\{k\}$
is also triangular (the non zero entries are binomial coefficients). Hence, the transition matrix
$BC^{-1}$ between $[k]$ and $n^k$ has integer entries and so $[k]$ is, in fact, a polynomial
in $n$ with integer coefficients. 
\end{remark}

\subsection{Affine Temperley-Lieb algebras}

The \emph{affine Temperley-Lieb algebra} $T_k^a$ is the diagram algebra generated by
$$
e_0 = {\beginpicture
\setcoordinatesystem units <0.35cm,0.2cm>         % sets scale
\setplotarea x from -.5 to 6, y from -2.5 to 3.5    % sets plot size up
\linethickness=0.5pt                          % sets line thickness
\put{$\bullet$} at 1 -1 \put{$\bullet$} at 1 2
\put{$\bullet$} at 2 -1 \put{$\bullet$} at 2 2
\put{$\bullet$} at 3 -1 \put{$\bullet$} at 3 2
\put{$\bullet$} at 4 -1 \put{$\bullet$} at 4 2
\put{$\bullet$} at 5 -1 \put{$\bullet$} at 5 2
%\put{$\scriptstyle{i}$} at 2.9 3.5 
\put{$\bullet$} at 6 -1 \put{$\bullet$} at 6 2
% Pole
\plot 0 3.5 0 -2.5 /
\plot .25 3.5 .25 -2.5 /
\ellipticalarc axes ratio 1:1 360 degrees from 0 3.5 center 
at .124 3.6
\put{$*$} at  .124 3.6  
\ellipticalarc axes ratio 1:1 220 degrees from 0 -2.5 center 
at .124  -2.6 
% edges
\plot 2 -1 2 2  /
\plot 3 -1 3 2  /
\plot 4 -1 4 2  /
\plot 5 -1 5 2  /
\setquadratic
\plot 1 2 -.5 1.5 -.2 .95 /
\plot .5 .9 1 .8 1.8 .9 / 
\plot 2.2 .9 2.4 .9 2.8 1.0 /
\plot 3.2 1.0 3.4 1.0 3.8 1.1 /
\plot 4.2 1.1 4.4 1.1 4.8 1.2 /
\plot 5.2 1.3 5.4 1.4 6 2 /
%\plot .5 .9  4 1.0 6 2 /
%\plot .5 .1  4 0.0 6 -1 /
\plot 1 -1 -.5 -.5 -.2 0 /
\plot .5 .1 1 .2 1.8 .1 /
\plot 2.2 .1 2.4 .1 2.8 .1 /
\plot 3.2 .1 3.4 .1 3.8 0 /
\plot 4.2 0 4.4 -.1 4.8 -.3 /
\plot 5.2 -.3 5.4 -.4 6 -1 /
\endpicture}, \quad
e_i = {\beginpicture
\setcoordinatesystem units <0.35cm,0.2cm>         % sets scale
\setplotarea x from -.5 to 6, y from -2.5 to 3.5    % sets plot size up
\linethickness=0.5pt                          % sets line thickness
\put{$\bullet$} at 1 -1 \put{$\bullet$} at 1 2
\put{$\bullet$} at 2 -1 \put{$\bullet$} at 2 2
\put{$\bullet$} at 3 -1 \put{$\bullet$} at 3 2
\put{$\bullet$} at 4 -1 \put{$\bullet$} at 4 2
\put{$\bullet$} at 5 -1 \put{$\bullet$} at 5 2
\put{$\scriptstyle{i}$} at 2.9 3.5 
\put{$\bullet$} at 6 -1 \put{$\bullet$} at 6 2
% Pole
\plot 0 3.5 0 -2.5 /
\plot .25 3.5 .25 -2.5 /
\ellipticalarc axes ratio 1:1 360 degrees from 0 3.5 center 
at .124 3.6
\put{$*$} at  .124 3.6  
\ellipticalarc axes ratio 1:1 220 degrees from 0 -2.5 center 
at .124  -2.6 
% edges
\plot 1 -1 1 2  /
\plot 2 -1 2 2  /
\plot 5 -1 5 2  /
\plot 6 -1 6 2  /
\setquadratic
\plot 3 -1 3.5 0 4 -1 /
\plot 3 2 3.5 1.0 4 2 /
\endpicture} \ (1\le i\le k-1), \quad\hbox{and}\quad
\tau = {\beginpicture
\setcoordinatesystem units <0.35cm,0.2cm>         % sets scale
\setplotarea x from -.5 to 6, y from -2.5 to 3.5    % sets plot size up
\linethickness=0.5pt                          % sets line thickness
\put{$\bullet$} at 1 -1 \put{$\bullet$} at 1 2
\put{$\bullet$} at 2 -1 \put{$\bullet$} at 2 2
\put{$\bullet$} at 3 -1 \put{$\bullet$} at 3 2
\put{$\bullet$} at 4 -1 \put{$\bullet$} at 4 2
\put{$\bullet$} at 5 -1 \put{$\bullet$} at 5 2
%\put{$\scriptstyle{i}$} at 2.9 3.5 
\put{$\bullet$} at 6 -1 \put{$\bullet$} at 6 2
% Pole
\plot 0 3.5 0 -2.5 /
\plot .25 3.5 .25 -2.5 /
\ellipticalarc axes ratio 1:1 360 degrees from 0 3.5 center 
at .124 3.6
\put{$*$} at  .124 3.6  
\ellipticalarc axes ratio 1:1 220 degrees from 0 -2.5 center 
at .124  -2.6 
% edges
\plot 1 -1 2 2  /
\plot 2 -1 3 2  /
\plot 3 -1 4 2  /
\plot 4 -1 5 2  /
\plot 5 -1 6 2 /
\setquadratic
\plot 1 2 -.5 1.5 -.2 .5  /
\plot .5 .2 .6 .1 1.1 .1 /
\plot 1.6 .1 1.8 .1 2.0 .1 /
\plot 2.6 .1 2.8 .1 3.2 .1 /
\plot 3.5 .1 3.7 .1 4.1 0 /
\plot 4.5 0 4.6 0 5.0 -.1 /
\plot 5.5 -.3 5.6 -.4 6 -1 /
\endpicture}.
$$
The generators of $T_k^a$ satisfy $e_i^2 = ne_i$,  
$e_ie_{i\pm1}e_i=e_i$, $\tau e_i\tau^{-1} = e_{i+1}$
(where the indices are taken mod $k$) and
\begin{equation}\label{extraTLrel}
\tau^2e_{k-1} = 
 {\beginpicture
\setcoordinatesystem units <0.35cm,0.2cm>         % sets scale
\setplotarea x from -.5 to 6, y from -4.5 to 5.5    % sets plot size up
\linethickness=0.5pt                          % sets line thickness
\put{$\bullet$} at 1 -4 \put{$\bullet$} at 1 5
\put{$\bullet$} at 2 -4 \put{$\bullet$} at 2 5
\put{$\bullet$} at 3 -4 \put{$\bullet$} at 3 5
\put{$\bullet$} at 4 -4 \put{$\bullet$} at 4 5
\put{$\bullet$} at 5 -4 \put{$\bullet$} at 5 5
\put{$\bullet$} at 6 -4 \put{$\bullet$} at 6 5
\put{$\bullet$} at 1 -1 \put{$\bullet$} at 1 2
\put{$\bullet$} at 2 -1 \put{$\bullet$} at 2 2
\put{$\bullet$} at 3 -1 \put{$\bullet$} at 3 2
\put{$\bullet$} at 4 -1 \put{$\bullet$} at 4 2
\put{$\bullet$} at 5 -1 \put{$\bullet$} at 5 2
\put{$\bullet$} at 6 -1 \put{$\bullet$} at 6 2
% Pole
\plot 0 5.5 0 -4.5 /
\plot .25 5.5 .25 -4.5 /
\ellipticalarc axes ratio 1:1 360 degrees from 0 5.5 center 
at .124 5.6
\put{$*$} at  .124 5.6  
\ellipticalarc axes ratio 1:1 220 degrees from 0 -4.5 center 
at .124  -4.6 
% edges
\plot 1 -1 2 2  /
\plot 2 -1 3 2  /
\plot 3 -1 4 2  /
\plot 4 -1 5 2  /
\plot 5 -1 6 2 /
\plot 1 2 2 5  /
\plot 2 2 3 5  /
\plot 3 2 4 5  /
\plot 4 2 5 5  /
\plot 5 2 6 5 /
\plot 1 -1 1 -4  /
\plot 2 -1 2 -4  /
\plot 3 -1 3 -4  /
\plot 4 -1 4 -4  /
\setquadratic
\plot 1 2 -.5 1.5 -.2 .5  /
\plot .5 .2 .6 .1 1.1 .1 /
\plot 1.6 .1 1.8 .1 2.0 .1 /
\plot 2.6 .1 2.8 .1 3.2 .1 /
\plot 3.5 .1 3.7 .1 4.1 0 /
\plot 4.5 0 4.6 0 5.0 -.1 /
\plot 5.5 -.3 5.6 -.4 6 -1 /
\plot 1      5 -.5 4.5 -.2 3.5  /
\plot .5    3.2 .6 3.1 1.1 3.1 /
\plot 1.6 3.1 1.8 3.1 2.0 3.1 /
\plot 2.6 3.1 2.8 3.1 3.2 3.1 /
\plot 3.5 3.1 3.7 3.1 4.1 3 /
\plot 4.5  3 4.6 3 5.0 2.9  /
\plot 5.5  2.7 5.6 2.6 6 2 /
\setquadratic
\plot 5 -4 5.5 -3 6 -4 /
\plot 5 -1 5.5 -2 6 -1 /
\endpicture}
= 
{\beginpicture
\setcoordinatesystem units <0.35cm,0.2cm>         % sets scale
\setplotarea x from -.5 to 6, y from -2.5 to 3.5    % sets plot size up
\linethickness=0.5pt                          % sets line thickness
\put{$\bullet$} at 1 -1 \put{$\bullet$} at 1 2
\put{$\bullet$} at 2 -1 \put{$\bullet$} at 2 2
\put{$\bullet$} at 3 -1 \put{$\bullet$} at 3 2
\put{$\bullet$} at 4 -1 \put{$\bullet$} at 4 2
\put{$\bullet$} at 5 -1 \put{$\bullet$} at 5 2
\put{$\bullet$} at 6 -1 \put{$\bullet$} at 6 2
% Pole
\plot 0 3.5 0 -2.5 /
\plot .25 3.5 .25 -2.5 /
\ellipticalarc axes ratio 1:1 360 degrees from 0 3.5 center 
at .124 3.6
\put{$*$} at  .124 3.6  
\ellipticalarc axes ratio 1:1 220 degrees from 0 -2.5 center 
at .124  -2.6 
% edges
\plot 1 -1 3 2  /
\plot 2 -1 4 2  /
\plot 3 -1 5 2  /
\plot 4 -1 6 2  /
\setquadratic
\plot 5 -1 5.5 0 6 -1 /
\plot 1 2 1.5 1.0 2 2 /
\endpicture}
= 
 {\beginpicture
\setcoordinatesystem units <0.35cm,0.2cm>         % sets scale
\setplotarea x from -.5 to 6, y from -4.5 to 6.5    % sets plot size up
\linethickness=0.5pt                          % sets line thickness
\put{$\bullet$} at 1 6 \put{$\bullet$} at 2 6 \put{$\bullet$} at 3 6 \put{$\bullet$} at 4 6 \put{$\bullet$} at 5 6 \put{$\bullet$} at 6 6 
\put{$\bullet$} at 1 4 \put{$\bullet$} at 2 4 \put{$\bullet$} at 3 4 \put{$\bullet$} at 4 4 \put{$\bullet$} at 5 4 \put{$\bullet$} at 6 4 
\put{$\bullet$} at 1 2 \put{$\bullet$} at 2 2 \put{$\bullet$} at 3 2 \put{$\bullet$} at 4 2 \put{$\bullet$} at 5 2 \put{$\bullet$} at 6 2 
\put{$\bullet$} at 1 0 \put{$\bullet$} at 2 0 \put{$\bullet$} at 3 0 \put{$\bullet$} at 4 0 \put{$\bullet$} at 5 0 \put{$\bullet$} at 6 0 
\put{$\bullet$} at 1 -2 \put{$\bullet$} at 2 -2 \put{$\bullet$} at 3 -2 \put{$\bullet$} at 4 -2 \put{$\bullet$} at 5 -2 \put{$\bullet$} at 6 -2 
\put{$\bullet$} at 1 -4 \put{$\bullet$} at 2 -4 \put{$\bullet$} at 3 -4 \put{$\bullet$} at 4 -4 \put{$\bullet$} at 5 -4 \put{$\bullet$} at 6 -4 
% Pole
\plot 0 6.5 0 -4.5 /
\plot .25 6.5 .25 -4.5 /
\ellipticalarc axes ratio 1:1 360 degrees from 0 6.5 center 
at .124 6.6
\put{$*$} at  .124 6.6  
\ellipticalarc axes ratio 1:1 220 degrees from 0 -4.5 center 
at .124  -4.6 
% edges
\plot 1 4 1 -4 /
\plot 2 2 2 -4 /
\plot 3 0 3 -4 /
\plot 4 -2 4 -4 /
\plot 6 -2 6 6 /
\plot 5 0 5 6 /
\plot 4 2 4 6 /
\plot 3 4 3 6 /
\setquadratic
\plot 1  4 1.5 4.5 2 4 /
\plot 1  6 1.5 5.5 2 6  /
\plot 2  2 2.5 2.5 3 2 /
\plot 2  4 2.5 3.5 3 4  /
\plot 3 0 3.5 0.5  4 0 /
\plot 3  2 3.5 1.5 4   2  /
\plot 4 -2 4.5 -1.5  5 -2 /
\plot 4  0 4.5 -0.5 5   0  /
\plot 5 -4 5.5 -3.5  6 -4 /
\plot 5 -2 5.5 -2.5 6 -2  /
\endpicture} = e_1e_2\cdots e_{k-1}
\end{equation}
(see \cite[4.15(iv)]{GL4}).  
In $T_k^a$, we let $X^{\varepsilon_1} = T_1^{-1} T_2^{-1} \cdots T_{k-1}^{-1} \tau^{-1}$ (see Remark \ref{AlternateBk}).

Graham and Lehrer \cite[\S 4.3]{GL4} define four slightly different 
affine Temperley-Lieb algebras, the diagram algebra $T_k^a$ and the algebras defined
as follows:
$$
\begin{array}{ll}
\hbox{Type $GL_k$:} 
&\hbox{$\widehat{\TL}_k^a$ is  $\tilde H_k$ with the relation \eqref{TLKernel},} \\
\hbox{Type $SL_k$:} 
&\hbox{$\TL_k^a$ is $\tilde H_Q$with the relation \eqref{TLKernel},} \\
\hbox{Type $PGL_k$:} 
&\hbox{$\widetilde{\TL}_k^a$ is $\tilde H_P$ with the relation \eqref{TLKernel}.}
\end{array}
$$
For each invertible element $\alpha$ in the base ring there is a surjective homomorphism
\begin{equation}\label{AffineToAffineTL}
\begin{matrix}
\tilde H_k &\longrightarrow
&\widehat{\TL}_k^a &\longrightarrow &T_k^a \\
\tau &\longmapsto &\tau  &\longmapsto &\alpha\tau\\
T_i &\longmapsto &q-e_i &\longmapsto &q-e_i \\
\end{matrix}
\end{equation}
and every irreducible representation of $\widehat{\TL}_k^a$ factors through one of these
homomorphisms (see \cite[Prop. 4.14(v)]{GL4}).  In Proposition \ref{SchurWeylAffineTL} we shall see that these homomorphisms arise naturally in the Schur-Weyl duality setting.

\subsection{A commuting family in the affine Temperley-Lieb algebra}

View the elements $X^{-\varepsilon_i}$ in the affine Temperley-Lieb algebra 
$\widehat{\TL}_k^a$ via the surjective
algebra homomorphism of \eqref{AffineToAffineTL}.   
Define 
\begin{equation}
(q-q^{-1}) m_1  =  q^{-1} X^{-\varepsilon_1}
\quad\hbox{and}\quad
(q-q^{-1})m_i =  q^{i-2}(X^{-\varepsilon_i} - q^{-2}X^{-\varepsilon_{i-1}}), 
\label{midef}
\end{equation}
for $i=2,3,\ldots, k$.  Since 
$X^{-\varepsilon_i}X^{-\varepsilon_j}=X^{-\varepsilon_j}X^{-\varepsilon_i}$
for all $1\le i,j\le k$, and
the $m_i$ are linear combinations of the $X^{-\varepsilon_i}$, 
$$m_im_j = m_jm_i\ \ \hbox{in $\widehat{\TL}_k^a$,}
\qquad\hbox{for all $1\le i,j\le k$.}
$$ 

\begin{prop}\label{eqvforms}  For $1 \le i \le k$,
\begin{equation*}
\begin{array}{ll}
(a) &X^{-\varepsilon_i} = q^{-(i-2)}(q-q^{-1})(m_i+q^{-1}m_{i-1}+q^{-2}m_{i-2}+\cdots
+q^{-(i-1)}m_1),\hskip1truein \\ \\
(b) &X^{-\varepsilon_1}+\cdots + X^{-\varepsilon_i}
= q^{-(i-2)}(q-q^{-1})(m_i+[2]m_{i-1}+\cdots + [i]m_1).
\end{array}
\end{equation*}
\end{prop}
\begin{proof}
Rewrite \eqref{midef} as 
$$
X^{-\varepsilon_i} = q^{-(i-2)}(q-q^{-1})m_i+q^{-2}X^{-\varepsilon_{i-1}}
$$
and use induction,
\begin{align*}
X^{-\varepsilon_i} 
&= q^{-(i-2)}(q-q^{-1})m_i +q^{-2}\big(q^{-(i-1-2)}(q-q^{-1})
(m_{i-1}+q^{-1}m_{i-2}+\cdots+q^{-(i-2)}m_1)\big),
\end{align*}
to obtain the formula for $X^{-\varepsilon_i}$ in (a). 
Summing the formula in (a) over $i$ gives
\begin{align*}
\sum_{j=1}^i X^{-\varepsilon_j}
&= \sum_{j=1}^i \left(q^{-(j-2)}(q-q^{-1}) \sum_{\ell=0}^{j-2} q^{-\ell}m_{j-\ell}\right) 
= q^{-(i-2)}(q-q^{-1}) \sum_{j=1}^i \sum_{\ell=0}^{j-1}
q^{i-j-\ell}m_{j-\ell} 
\end{align*}
and, thus, formula (c) follows from
\begin{align*}
\sum_{j=1}^i \sum_{\ell=0}^{j-1}
q^{i-j-\ell}m_{j-\ell} 
&= \sum_{j=1}^i \sum_{r=1}^j q^{i-j-(j-r)}m_r
=\sum_{r=1}^i \sum_{j=r}^i q^{i+r-2j}m_r = \sum_{r=1}^i  [i-r+1]m_r.
\end{align*}
\end{proof}

The following Lemma is a transfer of the recursion $X^{\varepsilon_i}
=T_{i-1}X^{\varepsilon_{i-1}}T_{i-1}$ to the $m_i$.
The following are the base cases of Lemma \ref{recursion}.  
\begin{align*}
m_1 &= \frac{q^{-1}}{q-q^{-1}}  X^{-\varepsilon_1}  \qquad\hbox{and}\qquad
m_2  = \frac{x}{q-q^{-1}} e_1 - \left( e_1 m_1 + m_1 e_1\right)
%m_3 &= \frac{q^{1}x}{q-q^{-1}}e_2 - \left( e_2 m_2 + m_2 e_2\right) - [2] m_1 e_2  \\
%m_4 &= \frac{q^{2}x}{q-q^{-1}} e_3 - (e_3 m_3 + m_3 e_3)  - ([3]-[1]) m_1 e_3 - [2] m_2 e_3 \\
%m_5 &= \frac{q^{3}x}{q-q^{-1}} e_4 - (e_4 m_4 + m_4 e_4)  - ([4]-[2]) m_1 e_4 - ([3]-[1]) m_2 e_4 - [2] m_3 e_4. \\
\end{align*}

\begin{lemma}\label{recursion}
Let $x$ be the constant defined by the equation
$e_1X^{-\varepsilon_1}e_1 = xe_1$.  For $2 \le i \le k$, 
$$
m_i = \frac{q^{i-2}x}{q-q^{-1}} e_{i-1} - (e_{i-1}m_{i-1}+m_{i-1}e_{i-1})
-\sum_{\ell=1}^{i-2} ([i-\ell]-[i-\ell-2])m_\ell e_{i-1}.
$$
\end{lemma}
\begin{proof}
From \eqref{BraidMurphy} and  \eqref{TiFromEi} we have $X^{-\varepsilon_i} = ( q^{-1}-e_{i-1})X^{-\varepsilon_{i-1}}(q^{-1}-e_{i-1}).$ Substituting this into the definition of $m_i$ gives 
\begin{align*}
(q-q^{-1})m_i 
&= q^{i-2}(X^{-\varepsilon_i}-q^{-2}X^{-\varepsilon_{i-1}}) 
= q^{i-2}(q^{-1}-e_{i-1})X^{-\varepsilon_{i-1}}(q^{-1}-e_{i-1}) - q^{i-4}X^{-\varepsilon_{i-1}} \\
&= q^{i-2}e_{i-1}X^{- \varepsilon_{i-1}}e_{i-1}
-q^{i-3}(e_{i-1}X^{-\varepsilon_{i-1}}+X^{-\varepsilon_{i-1}}e_{i-1}). 
\end{align*}
Use Proposition \ref{eqvforms} (a) to substitute for $X^{-\varepsilon_{i-1}}$,
\begin{align*}
(q-q^{-1})m_i 
&= (q-q^{-1})q^{-(m+i-3)}
(q^{i-2}e_{i-1}m_{i-1}e_{i-1}-q^{i-3}(e_{i-1}m_{i-1}+m_{i-1}e_{i-1})) \\
&\qquad
+(q-q^{-1})q^{-(i-3)}(q^{-1}m_{i-2}+\cdots+q^{-(i-2)}m_1)(q^{i-2}e_{i-1}^2-2q^{i-3}e_{i-1}) \\
&=(q-q^{-1})\big(
qe_{i-1}m_{i-1}e_{i-1}
-(e_{i-1}m_{i-1}+m_{i-1}e_{i-1})\big) \\
&\qquad
+(q-q^{-1})(m_{i-2}+\cdots+q^{-(i-3)}m_1)(q+q^{-1}-2q^{-1})e_{i-1}  \\
&= (q-q^{-1})\left(
\begin{array}{l}qe_{i-1}m_{i-1}e_{i-1}
-(e_{i-1}m_{i-1}+m_{i-1}e_{i-1}) \\
\qquad
+(q-q^{-1})(m_{i-2}+\cdots+q^{-(i-3)}m_1)e_{i-1}
\end{array}\right),
\end{align*}
which gives
\begin{equation}\label{mitoes}
\begin{array}{rl}
m_i &=~ qe_{i-1}m_{i-1}e_{i-1}
 -(e_{i-1}m_{i-1}+m_{i-1}e_{i-1})  \\
&\qquad +(q-q^{-1})(m_{i-2}+q^{-1}m_{i-3}+q^{-2}m_{i-4}+\cdots+q^{-(i-3)}m_1)e_{i-1}.
\end{array}
\end{equation}
Using induction, substitute for the first $m_{i-1}$ in this equation to get
\begin{align*}
m_i &=- (e_{i-1}m_{i-1}+m_{i-1}e_{i-1})  + (q-q^{-1})\sum_{\ell=1}^{i-2} q^{-(i-2-\ell)}m_\ell e_{i-1} \\
&\qquad + q\left(
\frac{ q^{i-3}x}{q-q^{-1}} e_{i-1} 
- 2m_{i-2}e_{i-1}  - \sum_{\ell=1}^{i-3} ([i-\ell-1]-[i-\ell-3])m_{\ell} e_{i-1} \right) \\
&= \frac{q^{i-2}x}{q-q^{-1} }  e_{i-1}  - (e_{i-1}m_{i-1}+m_{i-1}e_{i-1})   -\sum_{\ell=2}^{i-2} ([i-\ell]-[i-\ell-2])m_\ell e_{i-1}.
\end{align*}
\end{proof}

%\begin{remark}  To work with the (nonaffine) Temperley-Lieb algebra (via \eqref{AffineToTL})
%let $X^{-\varepsilon_1}=1$.  Then $x = q+q^{-1}$ and 
%$$(q-q^{-1})m_2 = q^{2-2}(X^{-\varepsilon_2}-q^{-2}X^{-\varepsilon_1})
%= (q+q^{-1})e_1 - 2q^{-1}e_1 = (q-q^{-1})e_1.$$
%\end{remark}

%%%%%%%%%%%%%%%%%%%%%%%%%%%%%%%%%%%%%%%%%%%%%
%%%%%%%%%%%%%%%%%%%%%%%%%%%%%%%%%%%%%%%%%%%%%
%%%%%%%%%%%%%%%%%%%%%%%%%%%%%%%%%%%%%%%%%%%%%
%%%%%%%%%%%%%%%%%%%%%%%%%%%%%%%%%%%%%%%%%%%%%

\subsection{Diagram Representation of Murphy Elements}

Label the vertices from left to right in the top row of a diagram $d \in T_k$ with $1, 2, \ldots, k$,  
and label the corresponding vertices in the bottom row with $1', 2', \ldots, k'$.
The \emph{cycle type} of a diagram $d\in T_k$ is the set partition $\tau(d)$ of
$\{1,2,\ldots, k\}$ obtained from $d$ by setting $1=1',2=2',\ldots, k=k'$.  If $\tau(d)$ is a set partition 
of the form $\{\{1,2,\ldots,\gamma_1\},\{\gamma_1+1,\mu_1+2, \ldots,\gamma_1+\gamma+2\},
\ldots, \{\gamma_1+\cdots+\gamma_{\ell-1}+1,\ldots,k\}\}$, where $(\gamma_1,\ldots, \gamma_\ell)$ is
a composition of $k$, then we simplify notation by writing $\tau(d) = (\gamma_1,\ldots,\gamma_\ell)$.
For example
$$d={\beginpicture
\setcoordinatesystem units <0.5cm,0.25cm> % sets scale
\setplotarea x from 1 to 12, y from -1 to 2   % sets plot size up
\linethickness=0.5pt                        % sets line thickness
\put{$\bullet$} at 1 -1 \put{$\bullet$} at 1 2
\put{$\bullet$} at 2 -1 \put{$\bullet$} at 2 2
\put{$\bullet$} at 3 -1 \put{$\bullet$} at 3 2
\put{$\bullet$} at 4 -1 \put{$\bullet$} at 4 2
\put{$\bullet$} at 5 -1 \put{$\bullet$} at 5 2
\put{$\bullet$} at 6 -1 \put{$\bullet$} at 6 2
\put{$\bullet$} at 7 -1 \put{$\bullet$} at 7 2
\put{$\bullet$} at 8 -1 \put{$\bullet$} at 8 2
\put{$\bullet$} at 9 -1 \put{$\bullet$} at 9 2
\put{$\bullet$} at 10 -1 \put{$\bullet$} at 10 2
\put{$\bullet$} at 11 -1 \put{$\bullet$} at 11 2
\put{$\bullet$} at 12 -1 \put{$\bullet$} at 12 2
\plot 1 -1 5 2 /
\plot 8 -1 6 2 /
\setquadratic
\plot 1 2   1.5 1.5   2 2 /
\plot 3 2   3.5 1.5   4 2 /
\plot 2 -1   2.5 -.5   3 -1 /
\plot 4 -1   4.5 -.5   5 -1 /
\plot 6 -1   6.5 -.5   7 -1 /
\plot 9 -1   9.5 -.5   10 -1 /
\plot 11 -1   11.5 -.5   12 -1 /
\plot 7 2   7.5 1.5   8 2 /
\plot 9 2   10.5 1   12 2 /
\plot 10 2   10.5 1.5   11 2 /
\endpicture}
\qquad\hbox{has}\qquad
\tau(d) = (5,3,4).$$
There are diagrams whose cycle type cannot be written as a composition (for example
$
$$d={\beginpicture
\setcoordinatesystem units <0.3cm,0.18cm> % sets scale
\setplotarea x from 1 to 4, y from -1 to 2   % sets plot size up
\linethickness=0.5pt                        % sets line thickness
\put{$\bullet$} at 1 -1 \put{$\bullet$} at 1 2
\put{$\bullet$} at 2 -1 \put{$\bullet$} at 2 2
\put{$\bullet$} at 3 -1 \put{$\bullet$} at 3 2
\put{$\bullet$} at 4 -1 \put{$\bullet$} at 4 2
\setquadratic
\plot 1 2   2.5 1.0   4 2 /
\plot 2 2   2.5 1.5   3 2 /
\plot 1 -1   2.5 0 4 -1 /
\plot 2 -1  2.5 -.5   3 -1 /
\endpicture}$ has cycle type $\{ \{1,4\}, \{2,3\}\}$) but all of the diagrams needed here have cycle types that are compositions. 

 If $\gamma=(\gamma_1,\ldots,\gamma_\ell)$ is a composition of $k$ define
\begin{equation}\label{TLMurphyDiagram}
d_\gamma = \sum_{\tau(d)=\gamma} d
%\qquad\hbox{in}\ \ T_k^a
\end{equation}  
as the sum of the Temperley-Lieb diagrams on $k$ dots with cycle type $\gamma$.
Define $d_\gamma^\ast$ be the sum of diagrams obtained from the summands of $d_\gamma$ 
by wrapping the first edge in each row  around the pole, with the orientation coming from
$X^{-\varepsilon_1}$ as shown in the  examples below.  When the first edge in the top row connects to the first vertex in the bottom row  only one new diagram is produced, otherwise there are two.  For example, in $\widehat{\TL}_4^a$,
$$
\begin{array}{lll}
d_{31} =  {\beginpicture
\setcoordinatesystem units <0.2cm,0.15cm>         % sets scale
\setplotarea x from -2 to 5, y from -3 to 5    % sets plot size up
\linethickness=0.5pt                          % sets line thickness
\put{$\bullet$} at 1 -1 \put{$\bullet$} at 1 2
\put{$\bullet$} at 2 -1 \put{$\bullet$} at 2 2
\put{$\bullet$} at 3 -1 \put{$\bullet$} at 3 2
\put{$\bullet$} at 4 -1 \put{$\bullet$} at 4 2
\plot 0 -2.5  0 3.5 /
\plot -.5 -2.5  -.5 3.5 /
\setquadratic
\plot 0 3.5 -.25 3.7 -.5 3.5 /
\plot 0 3.5 -.25 3.3 -.5 3.5 /
\plot 0 -2.5 -.25 -2.7 -.5 -2.5 /
\setlinear
\plot 1 2 3 -1 /
\plot 4 -1 4 2 /
\setquadratic
\plot 2 2 2.5 1.5 3 2 /
\plot 1 -1 1.5 -.5 2 -1 /
\endpicture}
+
 {\beginpicture
\setcoordinatesystem units <0.2cm,0.15cm>         % sets scale
\setplotarea x from -2 to 5, y from -3 to 5    % sets plot size up
\linethickness=0.5pt                          % sets line thickness
\put{$\bullet$} at 1 -1 \put{$\bullet$} at 1 2
\put{$\bullet$} at 2 -1 \put{$\bullet$} at 2 2
\put{$\bullet$} at 3 -1 \put{$\bullet$} at 3 2
\put{$\bullet$} at 4 -1 \put{$\bullet$} at 4 2
\plot 0 -2.5  0 3.5 /
\plot -.5 -2.5  -.5 3.5 /
\setquadratic
\plot 0 3.5 -.25 3.7 -.5 3.5 /
\plot 0 3.5 -.25 3.3 -.5 3.5 /
\plot 0 -2.5 -.25 -2.7 -.5 -2.5 /
\setlinear
\plot 1 -1 3 2 /
\plot 4 -1 4 2 /
\setquadratic
\plot 1 2 1.5 1.5 2 2 /
\plot 2 -1 2.5 -.5 3 -1 /
\endpicture}
&\quad&
d_{31}^\ast = {\beginpicture
\setcoordinatesystem units <0.2cm,0.15cm>         % sets scale
\setplotarea x from -2 to 5, y from -3 to 5    % sets plot size up
\linethickness=0.5pt                          % sets line thickness
\put{$\bullet$} at 1 -1 \put{$\bullet$} at 1 2
\put{$\bullet$} at 2 -1 \put{$\bullet$} at 2 2
\put{$\bullet$} at 3 -1 \put{$\bullet$} at 3 2
\put{$\bullet$} at 4 -1 \put{$\bullet$} at 4 2
\plot 0 -2.5  0 3.5 /
\plot -.5 -2.5  -.5 3.5 /
\setquadratic
\plot 0 3.5 -.25 3.7 -.5 3.5 /
\plot 0 3.5 -.25 3.3 -.5 3.5 /
\plot 0 -2.5 -.25 -2.7 -.5 -2.5 /
\setlinear
\plot 3 -1 1 2 /
\plot 4 -1 4 2 /
\setquadratic
\plot 2 2 2.5 1.5 3 2 /
\plot 2 -1 1.5  0  .2 1  -.2 1.1  -.7 1 -.9 .5  -.7 0  /
%\plot -.7 1   -.9 .5     -.5  0  /
%\plot -.5 0 0 0 .2 0 /
\plot .2 0 .75 -.3 1 -1 /
\endpicture}
+ {\beginpicture
\setcoordinatesystem units <0.2cm,0.15cm>         % sets scale
\setplotarea x from -2 to 5, y from -3 to 5    % sets plot size up
\linethickness=0.5pt                          % sets line thickness
\put{$\bullet$} at 1 -1 \put{$\bullet$} at 1 2
\put{$\bullet$} at 2 -1 \put{$\bullet$} at 2 2
\put{$\bullet$} at 3 -1 \put{$\bullet$} at 3 2
\put{$\bullet$} at 4 -1 \put{$\bullet$} at 4 2
\plot 0 -2.5  0 3.5 /
\plot -.5 -2.5  -.5 3.5 /
\setquadratic
\plot 0 3.5 -.25 3.7 -.5 3.5 /
\plot 0 3.5 -.25 3.3 -.5 3.5 /
\plot 0 -2.5 -.25 -2.7 -.5 -2.5 /
\setlinear
\plot 4 -1 4 2 /
\setquadratic
\plot 2 2 2.5 1.5 3 2 /
\plot 1 -1 1.5 -.5 2 -1 /
\plot 1 2   .75 1.5    -.7   1 -.9 .5   -.7  0 /
% \plot -.5 0 0 0 .2 .1 /
\plot .2 0 2 0 3 -1 /
\endpicture} 
 + {\beginpicture
\setcoordinatesystem units <0.2cm,0.15cm>         % sets scale
\setplotarea x from -2 to 5, y from -3 to 5    % sets plot size up
\linethickness=0.5pt                          % sets line thickness
\put{$\bullet$} at 1 -1 \put{$\bullet$} at 1 2
\put{$\bullet$} at 2 -1 \put{$\bullet$} at 2 2
\put{$\bullet$} at 3 -1 \put{$\bullet$} at 3 2
\put{$\bullet$} at 4 -1 \put{$\bullet$} at 4 2
\plot 0 -2.5  0 3.5 /
\plot -.5 -2.5  -.5 3.5 /
\setquadratic
\plot 0 3.5 -.25 3.7 -.5 3.5 /
\plot 0 3.5 -.25 3.3 -.5 3.5 /
\plot 0 -2.5 -.25 -2.7 -.5 -2.5 /
\setlinear
\plot 4 -1 4 2 /
\setquadratic
\plot 1 2 1.5 1.5 2 2 /
\plot 3 2  2  1  -.7 1   -.9 .5     -.7  0  /
%\plot -.5 0 0 0 .2 0 /
\plot .2 0 .75 -.3 1 -1 /
\plot 2 -1 2.5 -.5 3 -1 /
\endpicture}
+ {\beginpicture
\setcoordinatesystem units <0.2cm,0.15cm>         % sets scale
\setplotarea x from -2 to 5, y from -3 to 5    % sets plot size up
\linethickness=0.5pt                          % sets line thickness
\put{$\bullet$} at 1 -1 \put{$\bullet$} at 1 2
\put{$\bullet$} at 2 -1 \put{$\bullet$} at 2 2
\put{$\bullet$} at 3 -1 \put{$\bullet$} at 3 2
\put{$\bullet$} at 4 -1 \put{$\bullet$} at 4 2
\plot 0 -2.5  0 3.5 /
\plot -.5 -2.5  -.5 3.5 /
\setquadratic
\plot 0 3.5 -.25 3.7 -.5 3.5 /
\plot 0 3.5 -.25 3.3 -.5 3.5 /
\plot 0 -2.5 -.25 -2.7 -.5 -2.5 /
\setlinear
\plot 1 -1 3 2 /
\plot 4 -1 4 2 /
\setquadratic
\plot 1 2 .75 1.4   -.7 1   -.9 .5     -.7  0  /
%\plot -.5 0 0 0 .2 .1 /
\plot .2 0 1 .3 2 2 /
\plot 2 -1 2.5 -.5 3 -1 /
\endpicture} \\
d_{13} = {\beginpicture
\setcoordinatesystem units <0.2cm,0.15cm>         % sets scale
\setplotarea x from -2 to 5, y from -3 to 5    % sets plot size up
\linethickness=0.5pt                          % sets line thickness
\put{$\bullet$} at 1 -1 \put{$\bullet$} at 1 2
\put{$\bullet$} at 2 -1 \put{$\bullet$} at 2 2
\put{$\bullet$} at 3 -1 \put{$\bullet$} at 3 2
\put{$\bullet$} at 4 -1 \put{$\bullet$} at 4 2
\plot 0 -2.5  0 3.5 /
\plot -.5 -2.5  -.5 3.5 /
\setquadratic
\plot 0 3.5 -.25 3.7 -.5 3.5 /
\plot 0 3.5 -.25 3.3 -.5 3.5 /
\plot 0 -2.5 -.25 -2.7 -.5 -2.5 /
\setlinear
\plot 1 -1 1 2 /
\plot 4 -1 2 2 /
\setquadratic
\plot 3 -1 2.5 -.5 2 -1 /
\plot 3 2 3.5 1.5 4 2 /
\endpicture} 
+
{\beginpicture
\setcoordinatesystem units <0.2cm,0.15cm>         % sets scale
\setplotarea x from -2 to 5, y from -3 to 5    % sets plot size up
\linethickness=0.5pt                          % sets line thickness
\put{$\bullet$} at 1 -1 \put{$\bullet$} at 1 2
\put{$\bullet$} at 2 -1 \put{$\bullet$} at 2 2
\put{$\bullet$} at 3 -1 \put{$\bullet$} at 3 2
\put{$\bullet$} at 4 -1 \put{$\bullet$} at 4 2
\plot 0 -2.5  0 3.5 /
\plot -.5 -2.5  -.5 3.5 /
\setquadratic
\plot 0 3.5 -.25 3.7 -.5 3.5 /
\plot 0 3.5 -.25 3.3 -.5 3.5 /
\plot 0 -2.5 -.25 -2.7 -.5 -2.5 /
\setlinear
\plot 1 -1 1 2 /
\plot 2 -1 4 2 /
\setquadratic
\plot 3 2 2.5 1.5 2 2 /
\plot 3 -1 3.5 -.5 4 -1 /
\endpicture}
& &
d_{13}^\ast = {\beginpicture
\setcoordinatesystem units <0.2cm,0.15cm>         % sets scale
\setplotarea x from -2 to 5, y from -3 to 5    % sets plot size up
\linethickness=0.5pt                          % sets line thickness
\put{$\bullet$} at 1 -1 \put{$\bullet$} at 1 2
\put{$\bullet$} at 2 -1 \put{$\bullet$} at 2 2
\put{$\bullet$} at 3 -1 \put{$\bullet$} at 3 2
\put{$\bullet$} at 4 -1 \put{$\bullet$} at 4 2
\plot 0 -2.5  0 3.5 /
\plot -.5 -2.5  -.5 3.5 /
\setquadratic
\plot 0 3.5 -.25 3.7 -.5 3.5 /
\plot 0 3.5 -.25 3.3 -.5 3.5 /
\plot 0 -2.5 -.25 -2.7 -.5 -2.5 /
\setlinear
\plot 4 -1 2 2 /
\setquadratic
\plot 1 2 .75 1.5  -.7 1   -.9 .5     -.7  0  /
%\plot -.5 0 0 0 .2 0 /
\plot .2 0 .75 -.3 1 -1 /
\setquadratic
\plot 3 -1 2.5 -.5 2 -1 /
\plot 3 2 3.5 1.5 4 2 /
\endpicture} 
+
{\beginpicture
\setcoordinatesystem units <0.2cm,0.15cm>         % sets scale
\setplotarea x from -2 to 5, y from -3 to 5    % sets plot size up
\linethickness=0.5pt                          % sets line thickness
\put{$\bullet$} at 1 -1 \put{$\bullet$} at 1 2
\put{$\bullet$} at 2 -1 \put{$\bullet$} at 2 2
\put{$\bullet$} at 3 -1 \put{$\bullet$} at 3 2
\put{$\bullet$} at 4 -1 \put{$\bullet$} at 4 2
\plot 0 -2.5  0 3.5 /
\plot -.5 -2.5  -.5 3.5 /
\setquadratic
\plot 0 3.5 -.25 3.7 -.5 3.5 /
\plot 0 3.5 -.25 3.3 -.5 3.5 /
\plot 0 -2.5 -.25 -2.7 -.5 -2.5 /
\setlinear
\plot 2 -1 4 2 /
\setquadratic
\plot 1 2 .75 1.5  -.7 1   -.9 .5     -.7  0  /
%\plot -.5 0 0 0 .2 0 /
\plot .2 0 .75 -.3 1 -1 /
\setquadratic
\plot 3 2 2.5 1.5 2 2 /
\plot 3 -1 3.5 -.5 4 -1 /
\endpicture} 
\\
d_{22}\,\, = {\beginpicture
\setcoordinatesystem units <0.2cm,0.15cm>         % sets scale
\setplotarea x from -2 to 5, y from -3 to 5    % sets plot size up
\linethickness=0.5pt                          % sets line thickness
\put{$\bullet$} at 1 -1 \put{$\bullet$} at 1 2
\put{$\bullet$} at 2 -1 \put{$\bullet$} at 2 2
\put{$\bullet$} at 3 -1 \put{$\bullet$} at 3 2
\put{$\bullet$} at 4 -1 \put{$\bullet$} at 4 2
\plot 0 -2.5  0 3.5 /
\plot -.5 -2.5  -.5 3.5 /
\setquadratic
\plot 0 3.5 -.25 3.7 -.5 3.5 /
\plot 0 3.5 -.25 3.3 -.5 3.5 /
\plot 0 -2.5 -.25 -2.7 -.5 -2.5 /
\setlinear
\setquadratic
\plot 1 2 1.5 1.5 2 2 /
\plot 1 -1 1.5 -.5 2 -1 /
\plot 3 2  3.5 1.5 4 2 /
\plot 3 -1 3.5 -.5 4 -1 /
\endpicture}
&\qquad&
d_{22}^\ast\,\,  = 
{\beginpicture
\setcoordinatesystem units <0.2cm,0.15cm>         % sets scale
\setplotarea x from -2 to 5, y from -3 to 5    % sets plot size up
\linethickness=0.5pt                          % sets line thickness
\put{$\bullet$} at 1 -1 \put{$\bullet$} at 1 2
\put{$\bullet$} at 2 -1 \put{$\bullet$} at 2 2
\put{$\bullet$} at 3 -1 \put{$\bullet$} at 3 2
\put{$\bullet$} at 4 -1 \put{$\bullet$} at 4 2
\plot 0 -2.5  0 3.5 /
\plot -.5 -2.5  -.5 3.5 /
\setquadratic
\plot 0 3.5 -.25 3.7 -.5 3.5 /
\plot 0 3.5 -.25 3.3 -.5 3.5 /
\plot 0 -2.5 -.25 -2.7 -.5 -2.5 /
\setlinear
\setquadratic
\plot 1 2 1.5 1.5 2 2 /
\plot 3 2  3.5 1.5 4 2 /
\plot 2 -1 1.5  0  .2 1  -.2 1.1  -.7 1 -.9 .5  -.7 0  /
%\plot -.7 1   -.9 .5     -.5  0  /
%\plot -.5 0 0 0 .2 0 /
\plot .2 0 .75 -.3 1 -1 /
\endpicture}
+{\beginpicture
\setcoordinatesystem units <0.2cm,0.15cm>         % sets scale
\setplotarea x from -2 to 5, y from -3 to 5    % sets plot size up
\linethickness=0.5pt                          % sets line thickness
\put{$\bullet$} at 1 -1 \put{$\bullet$} at 1 2
\put{$\bullet$} at 2 -1 \put{$\bullet$} at 2 2
\put{$\bullet$} at 3 -1 \put{$\bullet$} at 3 2
\put{$\bullet$} at 4 -1 \put{$\bullet$} at 4 2
\plot 0 -2.5  0 3.5 /
\plot -.5 -2.5  -.5 3.5 /
\setquadratic
\plot 0 3.5 -.25 3.7 -.5 3.5 /
\plot 0 3.5 -.25 3.3 -.5 3.5 /
\plot 0 -2.5 -.25 -2.7 -.5 -2.5 /
\setlinear
\setquadratic
\plot 1 -1 1.5 -.5 2 -1 /
\plot 1 2 .75 1.5  -.7 1   -.9 .5     -.7  0  /
%\plot -.5 0 0 0 .2 .1 /
\plot .2 0 1 .3 2 2 /
\plot 3 2  3.5 1.5 4 2 /
\plot 3 -1 3.5 -.5 4 -1 /
\endpicture} 
\end{array}
$$
View $d_\gamma$ 
and $d_\gamma^*$ as elements of $\widehat{TL}_k^a$ by setting
$$d_\gamma = d_{\gamma1^{k-i}},
\qquad\hbox{if  $\gamma$ is a composition of $i$ with $i < k$.}
$$
With this notation, expanding the first few $m_i$ in terms of diagrams gives
\begin{align*}
(q-q^{-1})m_1 &= q^{-1} d_1^\ast,
\qquad
(q-q^{-1})m_2  = xd_2-q^{-1}d_2^*, \\
(q-q^{-1})m_3 &=  qx d_{1,2} - q^{-1} [2]  d_{1,2}^\ast - x d_3 + q^{-1} d_3^\ast, \\
(q-q^{-1})m_4 &= q^2 x d_{1^2,2} - q^{-1}([3]-[1]) d_{1^2,2}^\ast
- x [2]  d_{2,2} + q^{-1} [2] d_{2,2}^\ast \\
&\quad - qx d_{1,3} + q^{-1}[2]  d_{1,3}^\ast + x d_4 - q^{-1} d_4^\ast, \\
(q-q^{-1})m_5 &= 
q^3 x d_{1^3,2} - q^{-1} ([4]-[2]) d_{1^3,2}^\ast  
- q^2 x d_{1^2,3} + q^{-1} ([3]-[1]) d_{1^2,3}^\ast \\
&\quad +q x d_{1,4} - q^{-1} [2]  d_{1,4}^\ast 
- qx [2] d_{1,2,2} + q^{-1} [2]^2 d_{1,2,2}^\ast
+ qx [2] d_{2,3} - q^{-1} [2] d_{2,3}^\ast \\
&\quad- x ([3]-[1]) d_{2,1,2} + q^{-1} ([3]-[1]) d_{2,1,2}^\ast 
 + x [2] d_{3,2} - q^{-1} [2] d_{3,2}^\ast
- x d_{5} + q^{-1} d_{5}^\ast,
\end{align*}
where, as in Lemma \ref{recursion},  $x$ is the constant defined by the equation
$e_1X^{-\varepsilon_1}e_1 = xe_1.$

\begin{thm} \label{ExpansionInDiagrams}  Let $x$ be the constant defined by the equation
$e_1X^{-\varepsilon_1}e_1 = xe_1$.   Then $(q-q^{-1})m_1 = q^{-1} d_1$, 
$(q-q^{-1})m_2 = x d_2 - q^{-1} d_2^\ast$ and, for $i\ge 2$,
$$m_i = \sum_{\mathrm{compositions}\ \gamma} (m_i)_\gamma d_\gamma + 
(m_i)_{\gamma}^* d_\gamma^*,$$
where the sum is over all compositions $\gamma= 1^{b_1}r_11^{b_2}r_2\cdots 1^{b_\ell}r_\ell$ of $i$ with $r_\ell > 1$,
and
\begin{align*}
(m_i)_\gamma &= 
(-1)^{|\gamma|-\ell(\gamma)-1} \frac{q^{b_1}x}{q-q^{-1}} 
\displaystyle{\prod_{b_j\ge 0, j>1}} ([b_j+2]-[b_j]), \quad\hbox{and} \\
(m_i)_{\gamma}^* &= 
(-1)^{|\gamma|-\ell(\gamma)} \frac{q^{-1}}{q-q^{-1}} 
([b_1+1]-[b_1-1]) \displaystyle{\prod_{b_j\ge 0, j>1}} ([b_j+2]-[b_j]),
\end{align*}
with $\ell(\gamma) = \ell+b_1+\cdots+b_\ell$. 
\end{thm}

\begin{proof}  From our computations above,  $m_1 = A d_1^\ast$ and 
$m_2 = B d_2 - A d_2^\ast$, where
$$A = \frac{q^{-1}}{q-q^{-1}} \qquad\hbox{and}\qquad B = \frac{x}{q-q^{-1}}.$$
Let $m_1 = A d_1^\ast$. For $i > 2$ the recursion in Lemma \ref{recursion} gives
\begin{align*}
m_i &= q^{i-2} B e_{i-1}  - (e_{i-1}m_{i-1}+m_{i-1}e_{i-1})  - \sum_{\ell=1}^{i-2} ([i-\ell]-[i-\ell-2])m_\ell e_{i-1} \\
&= q^{i-2} B d_{1^{i-2},2} - ([i-1]-[i-3]) A d_{1^{i-2},2}^\ast
 - \big( (m_{i-1})_{\gamma'r}d_{\gamma',r+1}+(m_{i-1})_{\gamma'r}^* d_{\gamma',r+1}^*\big) \\
&\qquad +\sum_{\ell=2}^{i-2} - ([i-\ell]-[i-\ell-2])
\big( (m_\ell)_{\gamma'}d_{\gamma'1^{i-2-\ell}2}+(m_\ell)_{\gamma'}^*d_{\gamma'1^{i-2-\ell}2}^*\big).
\end{align*}
So if $d$ has cycle type $\gamma = 1^{b_1}r_11^{b_2}r_2\cdots 1^{b_\ell}r_\ell$ with $r_\ell >0$, then 
\begin{enumerate}
\item[($a$)] Each part of size $r$ ($r>1$) contributes $(-1)^{r-1}$ to the coefficient. Thus, there is a total contribution of $(-1)^{|\gamma| - \ell(\gamma)}$ from these parts.
\item[($b$)] Each inner $1^b$ ($b\ge 0$) contributes a factor of $[b+2]-[b]$ to the coefficient.
\item[($c$)] The first $1^b$ ($b>0$) contributes a $- q^{b} B$ in a nonstarred class,
\item[($c'$)] The first $1^b$ ($b=0$) contributes a $-B$ in a nonstarred class, which is the same as case ($c$) with $b=0$.
\item[($d$)] The first $1^b$ ($b>0$) contributes a $([b+1]-[b-1])A$ in a starred class.
\item[($d'$)] The first $1^b$ ($b=0$) contributes an $A$ in a starred class, which is the same as case ($d$) with $b=0$ assuming $[-1]=0$.
\end{enumerate}
\end{proof}

\begin{remark}\label{integrality}
To view $m_1,\ldots, m_k$ in  the (nonaffine) Temperley-Lieb algebra $\TL_k(n)$
(via \eqref{AffineToTL})
let $X^{-\varepsilon_1}=1$ so that $x = q+q^{-1}$. 
If $b_1 >1$ then  $d_\gamma^\ast = d_\gamma$ and 
if $b_1 = 0$ then $d_\gamma^\ast = 2 d_\gamma$.  In both cases
the coefficients in Theorem  \ref{ExpansionInDiagrams} specialize to
\begin{align*}
(m_i)_\gamma + (m_i)_\gamma^\ast 
%&= (-1)^{|\gamma|-\ell(\gamma)-1}\bigg( \displaystyle{\prod_{b_j\ge 0, j>1}} ([b_j+2]-[b_j]) \bigg) 
%\left( q^{b_1} B - A([b_1+1]-[b_1-1]) \right) \\
&= (-1)^{|\gamma|-\ell(\gamma)-1}
[b_1 + 1] 
\displaystyle{\prod_{b_j\ge 0, j>1}} ([b_j+2]-[b_j]) 
\end{align*}
and
$$
m_i = \sum_{\gamma} \left((m_i)_\gamma + (m_i)_\gamma^\ast\right) d_\gamma ,
%\qquad\hbox{where}\quad
%(m_i)_\gamma = 
%(-1)^{|\gamma|-\ell(\gamma)-1}[b_1+1]\displaystyle{\prod_{b_j\ge 0, j>1}} ([b_j+2]-[b_j])
$$
where the sum is over compositions $\gamma= 1^{b_1}r_11^{b_2}r_2\cdots 1^{b_\ell}r_\ell$ of $i$ with $r_\ell > 1$.
The first few examples are
\begin{align*}
m_1 &= \frac{q^{-1}}{q-q^{-1}} = \frac{q^{-1}}{q-q^{-1}} d_1, \qquad
m_2 = e_2= d_2, \qquad
m_3  = [2]d_{12}-d_3, \\
 m_4 &= [3] d_{1^2,2} - [2] d_{2,2} - [2] d_{1,3} + d_4, \\
  m_5 &= [4] d_{1^3,2} - [3] d_{1^2,3} + [2] d_{1,4} -[2]^2 d_{1,2,2} + [2] d_{2,3} - ([3]-[1]) d_{2,1,2} + [2] d_{3,2} - d_5.
 \end{align*}
\end{remark}

\end{section}

\begin{section}{Schur functors}

\subsection{$R$-matrices and quantum Casimir Elements}

Let $U_h\fg$ be the Drinfeld-Jimbo quantum group corresponding
to a finite dimensional complex semisimple Lie algebra $\fg$.  We shall use the notations
and conventions for $U_h\fg$ as in \cite{LR} and \cite{OR}.  There
is an invertible element ${\cal R}=\sum a_i\otimes b_i$ in (a suitable
completion of) $U_h\fg\otimes U_h\fg$ such that, for two $U_h\fg$
modules $M$ and $N$, the map
$$
\beginpicture
\setcoordinatesystem units <1cm,.5cm>         % sets scale
\setplotarea x from -8 to 2, y from 0 to 1.5    % sets plot size up
\put{$
\begin{matrix}
\check R_{MN}\colon &M\otimes N &\longrightarrow &N\otimes M\\
&m\otimes n &\longmapsto &\displaystyle{
\sum b_in\otimes a_i m }
\end{matrix}
$} at -5 1
\put{$M\otimes N$} at 0.5 2.4
\put{$N\otimes M$} at 0.5 0.1
\put{$\bullet$} at  0.1 1.9      %   Top dots
\put{$\bullet$} at  0.9 1.9      %
\put{$\bullet$} at  0.1 .6          %  Bottom dots
\put{$\bullet$} at  0.9 .6          %
% Vertical edges
%\plot -1 2  -1 .5 /
%\plot  0 2   0 .5 /
%\plot  1 2   1 .5 /
\setquadratic
\plot  0.1 .6  .15 .9  .4 1.15 /
\plot  .6 1.35  .85 1.6  0.9 1.9 /
\plot 0.1 1.9  .15 1.6  .5 1.25  .85 .9  0.9 .6 /
\endpicture
$$
is a $U_h\fg$ module isomorphism. 
In order to be consistent with the graphical calculus these operators should be
written \emph{on the right}.  The element ${\cal R}$ satisfies ``quasitriangularity relations''
(see \cite[(2.1-2.3)]{LR}) which imply that, for $U_h\fg$ modules 
$M,N,P$ and a $U_h\fg$ module isomorphism $\tau_M\colon M\to M$,
\begin{align*}
\beginpicture
\setcoordinatesystem units <1cm,.5cm>         % sets scale
\setplotarea x from 0 to 2, y from -1 to 2    % sets plot size up
\put{$M\otimes N$} at 0.5 2.4
\put{$N\otimes M$} at 0.5 -1.2
\put{$\tau_M$} at 1.2 -0.1
\put{$\bullet$} at  0.1 1.9      %   Top dots
\put{$\bullet$} at  0.9 1.9      %
\put{$\bullet$} at  0.1 -0.7          %  Bottom dots
\put{$\bullet$} at  0.9 -0.7          %
% Vertical edges
%\plot -1 2  -1 .5 /
%\plot  0 2   0 .5 /
\plot  0.9 .6   0.9 -0.7 /
\plot  0.1 .6   0.1 -0.7 /
\setquadratic
\plot  0.1 .6  .15 .9  .4 1.15 /
\plot  .6 1.35  .85 1.6  0.9 1.9 /
\plot 0.1 1.9  .15 1.6  .5 1.25  .85 .9  0.9 .6 /
\endpicture
~~&=~~
\beginpicture
\setcoordinatesystem units <1cm,.5cm>         % sets scale
\setplotarea x from 0 to 2, y from 0 to 2    % sets plot size up
\put{$M\otimes N$} at 0.5 2.4
\put{$N\otimes M$} at 0.5 -1.2
\put{$\tau_M$} at -0.2 1.3
\put{$\bullet$} at  0.1 1.9      %   Top dots
\put{$\bullet$} at  0.9 1.9      %
\put{$\bullet$} at  0.1 -0.7          %  Bottom dots
\put{$\bullet$} at  0.9 -0.7          %
% Vertical edges
%\plot -1 2  -1 .5 /
%\plot  0 2   0 .5 /
\plot  0.9 .6   0.9 1.9 /
\plot  0.1 .6   0.1 1.9 /
\setquadratic
\plot  0.1 -0.7  .15 -0.4  .4 -0.15 /
\plot  .6 0.05  .85 .3  0.9 .6 /
\plot 0.1 .6  .15 0.3  .5 -0.05  .85 -0.4  0.9 -0.7 /
\endpicture
\\
\check R_{MN}(\id_N\otimes\tau_M) &= (\tau_M\otimes \id_N)\check R_{MN}, 
\end{align*}
\begin{equation*}
\beginpicture
\setcoordinatesystem units <1cm,.5cm>         % sets scale
\setplotarea x from -2 to 2, y from 0 to 1.5    % sets plot size up
\put{$M\otimes\, (N\otimes P)$} at -2.1 1.8
\put{$(N\otimes P)\ \otimes M$} at -2.15 -0.6
\put{$\check R_{M,N\otimes P}
= 
(\check R_{MN}\otimes \id_P)(\id_N\otimes \check R_{MP})$} at -.7 -2.5
\put{$=$} at -.7 0.6
\put{$\bullet$} at  -1.65 1.2      %   Top dots
\put{$\bullet$} at  -2.6 1.2      %
\put{$\bullet$} at  -1.65 0          %  Bottom dots
\put{$\bullet$} at  -2.6  0          %
\put{$M\otimes N\otimes P$} at 0.9 2.5
\put{$N\otimes P\otimes M$} at 0.9 -1.4
\put{$\bullet$} at  0.15 1.9      %   Top dots
\put{$\bullet$} at  0.9 1.9      %
\put{$\bullet$} at  1.65 1.9      %
\put{$\bullet$} at  0.15 -0.7          %  Bottom dots
\put{$\bullet$} at  0.9 -0.7          %
\put{$\bullet$} at  1.65 -0.7          %
% Vertical edges
%\plot -1 2  -1 .5 /
%\plot  0 2   0 .5 /
\plot  1.65 1.9   1.65 .6 /
\plot  0.15 .6   0.15 -0.7 /
\setquadratic
\plot  0.15 .6  .2 .9  .45 1.15 /
\plot  .65 1.35  .85 1.6  0.9 1.9 /
\plot 0.15 1.9  .2 1.6  .55 1.25  .85 .9  0.9 .6 /
\plot  0.9 -0.7  0.95 -0.4  1.15 -0.15 /
\plot  1.35 .05  1.6 .3  1.65 .6 /
\plot 0.9 .6  0.95 .3  1.25 -0.05  1.6 -0.4  1.65 -0.7 /
\plot  -2.6 0  -2.55 .3  -2.2 0.55 /
\plot  -2 .75  -1.7 1  -1.65 1.3 /
\plot -2.6 1.3  -2.55 1  -2.1 .65  -1.7 .3  -1.65 0 /
\endpicture
\qquad\qquad
\beginpicture
\setcoordinatesystem units <1cm,.5cm>         % sets scale
\setplotarea x from -2 to 2, y from 0 to 1.5    % sets plot size up
\put{$(M\otimes N)\,\otimes P$} at -2.2 1.8
\put{$P\otimes (M\otimes N)$} at -2.05 -0.6
\put{$\check R_{M\otimes N,P}
=
(\id_M\otimes \check R_{NP})(\check R_{MP}\otimes \id_N),$} at -.7 -2.5
\put{$=$} at -.7 0.6
\put{$\bullet$} at  -1.65 1.2      %   Top dots
\put{$\bullet$} at  -2.6 1.2      %
\put{$\bullet$} at  -1.65 0          %  Bottom dots
\put{$\bullet$} at  -2.6  0          %
\put{$M\otimes N\otimes P$} at 0.9 2.5
\put{$P\otimes M\otimes N$} at 0.9 -1.4
\put{$\bullet$} at  0.15 1.9      %   Top dots
\put{$\bullet$} at  0.9 1.9      %
\put{$\bullet$} at  1.65 1.9      %
\put{$\bullet$} at  0.15 -0.7          %  Bottom dots
\put{$\bullet$} at  0.9 -0.7          %
\put{$\bullet$} at  1.65 -0.7          %
% Vertical edges
%\plot -1 2  -1 .5 /
%\plot  0 2   0 .5 /
\plot  1.65 0.6   1.65 -0.7 /
\plot  0.15 1.9   0.15 0.6 /
\setquadratic
\plot  0.15 -0.7  .2 -0.4  .45 -0.15 /
\plot  .65 .05  .85 0.3  0.9 0.6 /
\plot 0.15 0.6  .2 0.3  .55 -0.05  .85 -0.4  0.9 -0.7 /
\plot  0.9 0.6  0.95 0.9  1.15 1.15 /
\plot  1.35 1.35  1.6 1.6  1.65 1.9 /
\plot 0.9 1.9  0.95 1.6  1.25 1.25  1.6 0.9  1.65 0.6 /
\plot  -2.6 0  -2.55 .3  -2.2 0.55 /
\plot  -2 .75  -1.7 1  -1.65 1.3 /
\plot -2.6 1.3  -2.55 1  -2.1 .65  -1.7 .3  -1.65 0 /
\endpicture
\end{equation*}
which, together, imply the braid relation
\begin{align*}
\beginpicture
\setcoordinatesystem units <1cm,.5cm>         % sets scale
\setplotarea x from 0 to 2, y from -3 to 2.6    % sets plot size up
\put{$M\otimes N\otimes P$} at 0.9 2.3
\put{$P\otimes N\otimes M$} at 0.9 -2.5
\put{$\bullet$} at  0.15 1.9      %   Top dots
\put{$\bullet$} at  0.9 1.9      %
\put{$\bullet$} at  1.65 1.9      %
\put{$\bullet$} at  0.15 -2          %  Bottom dots
\put{$\bullet$} at  0.9 -2          %
\put{$\bullet$} at  1.65 -2          %
% Vertical edges
%\plot -1 2  -1 .5 /
%\plot  0 2   0 .5 /
\plot  1.65 1.9   1.65 .6 /
\plot  0.15 .6   0.15 -0.7 /
\plot  1.65 -0.7   1.65 -2 /
\setquadratic
\plot  0.15 .6  .2 .9  .45 1.15 /
\plot  .65 1.35  .85 1.6  0.9 1.9 /
\plot 0.15 1.9  .2 1.6  .55 1.25  .85 .9  0.9 .6 /
\plot  0.9 -0.7  0.95 -0.4  1.15 -0.15 /
\plot  1.35 .05  1.6 .3  1.65 .6 /
\plot 0.9 .6  0.95 .3  1.25 -0.05  1.6 -0.4  1.65 -0.7 /
\plot  0.15 -2  .2 -1.7  .45 -1.45 /
\plot  .65 -1.25  .85 -1  0.9 -0.7 /
\plot 0.15 -0.7  .2 -1  .55 -1.35  .85 -1.7  0.9 -2 /
\endpicture
~~&=~~
\beginpicture
\setcoordinatesystem units <1cm,.5cm>         % sets scale
\setplotarea x from 0 to 2, y from -3 to 2.6    % sets plot size up
\put{$M\otimes N\otimes P$} at 0.9 2.3
\put{$P\otimes N\otimes M$} at 0.9 -2.5
\put{$\bullet$} at  0.15 1.9      %   Top dots
\put{$\bullet$} at  0.9 1.9      %
\put{$\bullet$} at  1.65 1.9      %
\put{$\bullet$} at  0.15 -2          %  Bottom dots
\put{$\bullet$} at  0.9 -2          %
\put{$\bullet$} at  1.65 -2          %
% Vertical edges
%\plot -1 2  -1 .5 /
%\plot  0 2   0 .5 /
\plot  0.15 1.9   0.15 .6 /
\plot  1.65 .6   1.65 -0.7 /
\plot  0.15 -0.7   0.15 -2 /
\setquadratic
\plot  0.9 0.6  0.95 0.9  1.15 1.15 /
\plot  1.35 1.35  1.6 1.6  1.65 1.9 /
\plot 0.9 1.9  0.95 1.6  1.25 1.25  1.6 0.9  1.65 0.6 /
\plot  0.15 -0.7  .2 -0.4  .45 -0.15 /
\plot  .65 .05  .85 .3  0.9 .6 /
\plot 0.15 .6  .2 .3  .55 -0.05  .85 -0.4  0.9 -0.7 /
\plot  0.9 -2  0.95 -1.7  1.15 -1.45 /
\plot  1.35 -1.25  1.6 -1  1.65 -0.7 /
\plot 0.9 -0.7  0.95 -1  1.25 -1.35  1.6 -1.7  1.65 -2 /
\endpicture
\\
(\check R_{MN}\otimes \id_P)
(\id_N\otimes \check R_{MP})
(\check R_{NP}\otimes \id_M)
&=
(\id_M\otimes \check R_{NP})
(\check R_{MP}\otimes \id_N)
(\id_P\otimes \check R_{MN}).
\end{align*}

Let $\rho$ be such that $\langle \rho, \alpha_i\rangle = 1$ for all simple roots
$\alpha_i$.  As explained in \cite[(2.14)]{LR} and \cite{Dr1}, there is a
\emph{quantum Casimir element} $e^{-h\rho}u$ in the center of $U_h\fg$
and, for a $U_h\fg$ module $M$ we define a $U_h\fg$ module isomorphism
$$
\beginpicture
\setcoordinatesystem units <1cm,.6cm>         % sets scale
\setplotarea x from -5 to 1, y from 0 to 1.5    % sets plot size up
\put{$
\begin{matrix}
C_M\colon &M &\longrightarrow &M \\
&m &\longmapsto &(e^{-h\rho}u)m 
\end{matrix}
$} at -4 1.2
\put{$M$} at 0.7 2.2
\put{$M$} at 0.7 0.3
\put{$C_M$} at 1.2 1.3
\put{$\bullet$} at  0.7 1.8      %   Top dots
\put{$\bullet$} at  0.7 .7          %  Bottom dots
% Vertical edges
\plot  0.7 1.8   0.7 .7 /
%\setquadratic
%\plot  0 .5  .05 .8  .4 1.15 /
%\plot  .6 1.35  .95 1.7  1 2 /
%\plot 0 2  .05 1.7  .5 1.25  .95 .8  1 .5 /
\endpicture
$$
and the elements $C_M$ satisfy
\begin{equation}\label{casimir}
C_{M\otimes N} =
(\check R_{MN}\check R_{NM})^{-1}
(C_M\otimes C_N),
\qquad\hbox{and}\qquad
C_M = q^{-\langle \lambda,\lambda+2\rho\rangle} \id_M
\end{equation}
if $M$ is a $U_h\fg$ module generated by a highest weight vector 
$v^+$ of weight $\lambda$ (see \cite[Prop. 2.14]{LR} or \cite[Prop. 3.2]{Dr1}).
Note that $\langle \lambda,\lambda+2\rho\rangle
=\langle \lambda+\rho,\lambda+\rho\rangle - \langle\rho,\rho\rangle$
are the eigenvalues of the classical Casimir operator \cite[7.8.5]{Dx}.
From the relation \eqref{casimir} it follows that if $M=L(\mu)$,
$N=L(\nu)$ are finite dimensional irreducible $U_h\fg$ modules
then $\check R_{MN}\check R_{NM}$ acts on the
$\lambda$ isotypic component 
$L(\lambda)^{\oplus c_{\mu\nu}^\lambda}$
of the decomposition
\begin{equation}\label{fulltwist}
L(\mu)\otimes L(\nu) = \bigoplus_\lambda L(\lambda)^{\oplus c_{\mu\nu}^\lambda}
\qquad\hbox{by the constant}\qquad
q^{\langle\lambda,\lambda+2\rho\rangle 
-\langle\mu,\mu+2\rho\rangle 
-\langle\nu,\nu+2\rho\rangle}.
\end{equation}

\subsection{The $\tilde \cB_k$ module $M\otimes V^{\otimes k}$}

Let $U_h\fg$ be a quantum group and let 
$M$ and $V$ be $U$-modules such that the operators
$\check R_{MV}$, $\check R_{VM}$ and $\check R_{VV}$ are
well defined.
Define $\check R_i$, $1\le i\le k-1$, and $\check R_0^2$
in $\End_{U}(M\otimes V^{\otimes k})$ by
$$\check R_i = \id_M\otimes \id_V^{\otimes (i-1)}
\otimes \check R_{VV}\otimes \id_V^{\otimes (k-i-1)}
\qquad\hbox{and}\qquad
\check R_0^2 = (\check R_{MV}\check R_{VM})\otimes \id_V^{\otimes (k-1)}.
$$
Then the braid relations
$$
\check R_i\check R_{i+1}\check R_i 
=
\beginpicture
\setcoordinatesystem units <.3cm,.3cm>         % sets scale
\setplotarea x from 0 to 2, y from -2 to 2    % sets plot size up
%\put{$M\otimes N\otimes P$} at 0.9 2.3
%\put{$P\otimes N\otimes M$} at 0.9 -2.5
%\put{$\cdot$} at  0.15 3.9      %   Top dots
%\put{$\cdot$} at  0.9 3.9      %
%\put{$\cdot$} at  1.65 3.9      %
%\put{$\cdot$} at  0.15  -4          %  Bottom dots
%\put{$\cdot$} at  0.9  -4          %
%\put{$\cdot$} at  1.65  -4          %
% Vertical edges
%\plot -1 2  -1 .5 /
%\plot  0 2   0 .5 /
%\plot  0.15  -2.7   0.15  -4 /
%\plot  1.65  3.9   1.65  2.6 /
\plot  1.65  2.6   1.65  1.3 /
\plot  0.15  1.3     0.15  -0.1 /
\plot  1.65  -0.1   1.65  -1.4 /
%\plot  1.65  -1.4   1.65  -2.7 /
\setsolid
\setquadratic
\plot  0.15 -1.4  .2 -1.1  .45 -0.85 /
\plot  .65 -0.7  .85 -0.3  0.9 0 /
\plot 0.15 0  .2 -0.3  .55 -.8  .85 -1.1  0.9 -1.4 /
\plot  0.9 0  0.95 0.3  1.15 0.55 /
\plot  1.35 0.75  1.6 1  1.65 1.3 /
\plot 0.9 1.3  0.95 1  1.25 .65  1.6 .3  1.65 0 /
%\plot  0.15 -2.7  .2 -2.4  .45  -2.15 /
%\plot  .65 -2.05  .85 -1.7  0.9 -1.3 /
%\plot 0.15 -1.4  .2 -1.7  .55 -2.05  .85 -2.4  0.9 -2.7 /
%
%\plot  0.9 -4  0.95 -3.7  1.15 -3.45 /
%\plot  1.35 -3.25  1.6 -3  1.65 -2.7 /
%\plot 0.9 -2.7  0.95 -3  1.25 -3.35  1.6 -3.7  1.65 -4 /
%\plot  0.15 2.6  .2 2.9  .45 3.15 /
%\plot  .65 3.35  .85 3.6  0.9 3.9 /
%\plot 0.15 3.9  .2 3.6  .55 3.25  .85 2.9  0.9 2.6 /
\plot  0.15 1.3  .2 1.6  .45 1.85 /
\plot  .65 2.05  .85 2.3  0.9 2.6 /
\plot 0.15 2.6  .2 2.3  .55 1.95  .85 1.6  0.9 1.3 /
\endpicture
=
\beginpicture
\setcoordinatesystem units <.3cm,.3cm>         % sets scale
\setplotarea x from 0 to 2, y from -2 to 2    % sets plot size up
%\put{$M\otimes N\otimes P$} at 0.9 2.3
%\put{$P\otimes N\otimes M$} at 0.9 -2.5
%\put{$\cdot$} at  0.15 3.9      %   Top dots
%\put{$\cdot$} at  0.9 3.9      %
%\put{$\cdot$} at  1.65 3.9      %
%\put{$\cdot$} at  0.15  -4          %  Bottom dots
%\put{$\cdot$} at  0.9  -4          %
%\put{$\cdot$} at  1.65  -4          %
% Vertical edges
%\plot -1 2  -1 .5 /
%\plot  0 2   0 .5 /
%\plot  1.65  3.9   1.65  2.6 /
\plot  0.15  2.6   0.15  1.3 /
\plot  1.65  1.3   1.65  0 /
\plot  0.15  0   0.15  -1.4 /
%\plot  1.65  -1.4   1.65  -2.7 /
%\plot  0.15  -2.7   0.15  -4 /
%\setdashes
%\plot  0  2.6   2  2.6 /
%\plot  0  -1.4   2  -1.4 /
\setsolid
\setquadratic
%\plot  0.15 2.6  .2 2.9  .45 3.15 /
%\plot  .65 3.35  .85 3.6  0.9 3.9 /
%\plot 0.15 3.9  .2 3.6  .55 3.25  .85 2.9  0.9 2.6 /
\plot  0.9 1.3  0.95 1.6  1.15  1.85 /
\plot  1.35 2.05  1.6 2.3  1.65 2.6 /
\plot 0.9 2.6  0.95 2.3  1.25 1.95  1.6 1.6  1.65 1.3 /
\plot  0.15 0  .2 .3  .45  0.55 /
\plot  .65 .75  .85 1  0.9 1.3 /
\plot 0.15 1.3  .2 1  .55 .65  .85 .3  0.9 0 /
\plot  0.9 -1.4  0.95 -1.1  1.15 -0.85 /
\plot  1.35 -0.7  1.6 -0.3  1.65 0 /
\plot 0.9 -0  0.95 -0.3  1.25 -.8  1.6 -1.1  1.65 -1.4 /
%\plot  0.15 -2.7  .2 -2.4  .45 -2.15 /
%\plot  .65 -1.95  .85 -1.7  0.9 -1.4 /
%\plot 0.15 -1.4  .2 -1.7  .55 -2.05  .85 -2.4  0.9 -2.7 /
%\plot  0.9 -4  0.95 -3.7  1.15 -3.45 /
%\plot  1.35 -3.25  1.6 -3  1.65 -2.7 /
%\plot 0.9 -2.7  0.95 -3  1.25 -3.35  1.6 -3.7  1.65 -4 /
\endpicture
= \check R_{i+1}\check R_i \check R_{i+1}$$
and
$$
\check R_0^2\check R_1\check R_0^2\check R_1 
=
\beginpicture
\setcoordinatesystem units <.3cm,.3cm>         % sets scale
\setplotarea x from 0 to 2, y from -2 to 2    % sets plot size up
%\put{$M\otimes N\otimes P$} at 0.9 2.3
%\put{$P\otimes N\otimes M$} at 0.9 -2.5
\put{$\cdot$} at  0.15 3.9      %   Top dots
\put{$\cdot$} at  0.9 3.9      %
\put{$\cdot$} at  1.65 3.9      %
\put{$\cdot$} at  0.15  -4          %  Bottom dots
\put{$\cdot$} at  0.9  -4          %
\put{$\cdot$} at  1.65  -4          %
% Vertical edges
%\plot -1 2  -1 .5 /
%\plot  0 2   0 .5 /
\plot  0.15  -2.7   0.15  -4 /
\plot  1.65  3.9   1.65  2.6 /
\plot  1.65  2.6   1.65  1.3 /
\plot  0.15  1.3     0.15  -0.1 /
\plot  1.65  -0.1   1.65  -1.4 /
\plot  1.65  -1.4   1.65  -2.7 /
\setsolid
\setquadratic
\plot  0.15 -1.4  .2 -1.1  .45 -0.85 /
\plot  .65 -0.7  .85 -0.3  0.9 0 /
\plot 0.15 0  .2 -0.3  .55 -.8  .85 -1.1  0.9 -1.4 /
\plot  0.9 0  0.95 0.3  1.15 0.55 /
\plot  1.35 0.75  1.6 1  1.65 1.3 /
\plot 0.9 1.3  0.95 1  1.25 .65  1.6 .3  1.65 0 /
\plot  0.15 -2.7  .2 -2.4  .45  -2.15 /
\plot  .65 -2.05  .85 -1.7  0.9 -1.3 /
\plot 0.15 -1.4  .2 -1.7  .55 -2.05  .85 -2.4  0.9 -2.7 /
\plot  0.9 -4  0.95 -3.7  1.15 -3.45 /
\plot  1.35 -3.25  1.6 -3  1.65 -2.7 /
\plot 0.9 -2.7  0.95 -3  1.25 -3.35  1.6 -3.7  1.65 -4 /
\plot  0.15 2.6  .2 2.9  .45 3.15 /
\plot  .65 3.35  .85 3.6  0.9 3.9 /
\plot 0.15 3.9  .2 3.6  .55 3.25  .85 2.9  0.9 2.6 /
\plot  0.15 1.3  .2 1.6  .45 1.85 /
\plot  .65 2.05  .85 2.3  0.9 2.6 /
\plot 0.15 2.6  .2 2.3  .55 1.95  .85 1.6  0.9 1.3 /
\endpicture
=
\beginpicture
\setcoordinatesystem units <.3cm,.3cm>         % sets scale
\setplotarea x from 0 to 2, y from -2 to 2    % sets plot size up
%\put{$M\otimes N\otimes P$} at 0.9 2.3
%\put{$P\otimes N\otimes M$} at 0.9 -2.5
\put{$\cdot$} at  0.15 3.9      %   Top dots
\put{$\cdot$} at  0.9 3.9      %
\put{$\cdot$} at  1.65 3.9      %
\put{$\cdot$} at  0.15  -4          %  Bottom dots
\put{$\cdot$} at  0.9  -4          %
\put{$\cdot$} at  1.65  -4          %
% Vertical edges
%\plot -1 2  -1 .5 /
%\plot  0 2   0 .5 /
\plot  1.65  3.9   1.65  2.6 /
\plot  0.15  2.6   0.15  1.3 /
\plot  1.65  1.3   1.65  0 /
\plot  0.15  0   0.15  -1.4 /
\plot  1.65  -1.4   1.65  -2.7 /
\plot  0.15  -2.7   0.15  -4 /
\setdashes
\plot  0  2.6   2  2.6 /
\plot  0  -1.4   2  -1.4 /
\setsolid
\setquadratic
\plot  0.15 2.6  .2 2.9  .45 3.15 /
\plot  .65 3.35  .85 3.6  0.9 3.9 /
\plot 0.15 3.9  .2 3.6  .55 3.25  .85 2.9  0.9 2.6 /
\plot  0.9 1.3  0.95 1.6  1.15  1.85 /
\plot  1.35 2.05  1.6 2.3  1.65 2.6 /
\plot 0.9 2.6  0.95 2.3  1.25 1.95  1.6 1.6  1.65 1.3 /
\plot  0.15 0  .2 .3  .45  0.55 /
\plot  .65 .75  .85 1  0.9 1.3 /
\plot 0.15 1.3  .2 1  .55 .65  .85 .3  0.9 0 /
\plot  0.9 -1.4  0.95 -1.1  1.15 -0.85 /
\plot  1.35 -0.7  1.6 -0.3  1.65 0 /
\plot 0.9 -0  0.95 -0.3  1.25 -.8  1.6 -1.1  1.65 -1.4 /
\plot  0.15 -2.7  .2 -2.4  .45 -2.15 /
\plot  .65 -1.95  .85 -1.7  0.9 -1.4 /
\plot 0.15 -1.4  .2 -1.7  .55 -2.05  .85 -2.4  0.9 -2.7 /
\plot  0.9 -4  0.95 -3.7  1.15 -3.45 /
\plot  1.35 -3.25  1.6 -3  1.65 -2.7 /
\plot 0.9 -2.7  0.95 -3  1.25 -3.35  1.6 -3.7  1.65 -4 /
\endpicture
=
\beginpicture
\setcoordinatesystem units <.3cm,.3cm>         % sets scale
\setplotarea x from 0 to 2, y from -2 to 2    % sets plot size up
%\put{$M\otimes N\otimes P$} at 0.9 2.3
%\put{$P\otimes N\otimes M$} at 0.9 -2.5
\put{$\cdot$} at  0.15 3.9      %   Top dots
\put{$\cdot$} at  0.9 3.9      %
\put{$\cdot$} at  1.65 3.9      %
\put{$\cdot$} at  0.15  -4          %  Bottom dots
\put{$\cdot$} at  0.9  -4          %
\put{$\cdot$} at  1.65  -4          %
% Vertical edges
%\plot -1 2  -1 .5 /
%\plot  0 2   0 .5 /
\plot  0.15  3.9   0.15  2.6 /
\plot  1.65  2.6   1.65  1.3 /
\plot  0.15  1.3   0.15  0 /
\plot  1.65  0   1.65  -1.4 /
\plot  0.15  -1.4   0.15  -2.7 /
\plot  1.65  -2.7   1.65  -4 /
\setdashes
\plot  0  0   2  0 /
\setsolid
\setquadratic
\plot  0.15 -1.4  .2 -1.1  .45 -0.85 /
\plot  .65 -0.7  .85 -0.3  0.9 0 /
\plot 0.15 0  .2 -0.3  .55 -.8  .85 -1.1  0.9 -1.4 /
\plot  0.9 -2.7  0.95 -2.4  1.15  -2.15 /
\plot  1.35 -1.95  1.6 -1.7  1.65 -1.4 /
\plot 0.9 -1.4  0.95 -1.7  1.25 -2.05  1.6 -2.4  1.65 -2.7 /
\plot  0.15 -4  .2 -3.7  .45  -3.45 /
\plot  .65 -3.25  .85 -3  0.9 -2.7 /
\plot 0.15 -2.7  .2 -3  .55 -3.35  .85 -3.7  0.9 -4 /
\plot  0.9 2.6  0.95 2.9  1.15 3.15 /
\plot  1.35 3.3  1.6 3.7  1.65 4 /
\plot 0.9 4  0.95 3.7  1.25 3.2  1.6 2.9  1.65 2.6 /
\plot  0.15 1.3  .2 1.6  .45 1.85 /
\plot  .65 2.05  .85 2.3  0.9 2.6 /
\plot 0.15 2.6  .2 2.3  .55 1.95  .85 1.6  0.9 1.3 /
\plot  0.9 0  0.95 0.3  1.15 0.55 /
\plot  1.35 0.75  1.6 1  1.65 1.3 /
\plot 0.9 1.3  0.95 1  1.25 .65  1.6 .3  1.65 0 /
\endpicture
=
\beginpicture
\setcoordinatesystem units <.3cm,.3cm>         % sets scale
\setplotarea x from 0 to 2, y from -2 to 2    % sets plot size up
%\put{$M\otimes N\otimes P$} at 0.9 2.3
%\put{$P\otimes N\otimes M$} at 0.9 -2.5
\put{$\cdot$} at  0.15 3.9      %   Top dots
\put{$\cdot$} at  0.9 3.9      %
\put{$\cdot$} at  1.65 3.9      %
\put{$\cdot$} at  0.15  -4          %  Bottom dots
\put{$\cdot$} at  0.9  -4          %
\put{$\cdot$} at  1.65  -4          %
% Vertical edges
%\plot -1 2  -1 .5 /
%\plot  0 2   0 .5 /
\plot  0.15  3.9   0.15  2.6 /
\plot  1.65  2.6   1.65  1.3 /
\plot  1.65  1.3   1.65  0 /
\plot  0.15  0     0.15  -1.4 /
\plot  1.65  -1.4   1.65  -2.7 /
\plot  1.65  -2.7   1.65  -4 /
\setdashes
\plot  0  1.3   2  1.3 /
\plot  0  -2.7   2  -2.7 /
\setsolid
\setquadratic
\plot  0.15 -2.7  .2 -2.4  .45 -2.15 /
\plot  .65 -1.95  .85 -1.7  0.9 -1.4 /
\plot 0.15 -1.4  .2 -1.7  .55 -2.05  .85 -2.4  0.9 -2.7 /
\plot  0.9 -1.4  0.95 -1.1  1.15  -0.85 /
\plot  1.35 -0.7  1.6 -0.3  1.65 0 /
\plot 0.9 0  0.95 -0.3  1.25 -0.8  1.6 -1.1  1.65 -1.4 /
\plot  0.15 -4  .2 -3.7  .45  -3.45 /
\plot  .65 -3.25  .85 -3  0.9 -2.7 /
\plot 0.15 -2.7  .2 -3  .55 -3.35  .85 -3.7  0.9 -4 /
\plot  0.9 2.6  0.95 2.9  1.15 3.15 /
\plot  1.35 3.3  1.6 3.7  1.65 4 /
\plot 0.9 4  0.95 3.7  1.25 3.2  1.6 2.9  1.65 2.6 /
\plot  0.15 1.3  .2 1.6  .45 1.85 /
\plot  .65 2.05  .85 2.3  0.9 2.6 /
\plot 0.15 2.6  .2 2.3  .55 1.95  .85 1.6  0.9 1.3 /
\plot  0.15 0  .2 0.3  .45 0.55 /
\plot  .65 .75  .85 1  0.9 1.3 /
\plot 0.15 1.3  .2 1  .55 .65  .85 .3  0.9 0 /
\endpicture
= \check R_1\check R_0^2\check R_1\check R_0^2. 
$$
imply that there is a well defined map
\begin{equation}\label{BraidRep}
\begin{matrix}
\Phi\colon &\tilde {\cal B}_k &\longrightarrow
&\End_{U}(M\otimes V^{\otimes k}) \\
&T_i &\longmapsto &\check R_i, &\qquad &1\le i\le k-1,\\
&X^{\varepsilon_1} &\longmapsto &\check R_0^2,
\end{matrix}
\end{equation}
which makes $M\otimes V^{\otimes k}$ into a right $\tilde {\cal B}_k$
module.  By \eqref{casimir} and the fact that
\begin{equation}\label{XiasRmatrix}
\Phi(X^{\varepsilon_i}) = 
\check R_{M\otimes V^{\otimes(i-1)},V}\check R_{V, M\otimes V^{\otimes {(i-1)}}}
= 
\beginpicture
\setcoordinatesystem units <.5cm,1cm>         % sets scale
\setplotarea x from -5.5 to 4.5, y from -1.25 to 1.25    % sets plot size up
\put{${}^i$} at 1 1.2 
\put{$\bullet$} at -3 0.75      %
\put{$\bullet$} at -2 0.75      %
\put{$\bullet$} at -1 0.75      %
\put{$\bullet$} at  0 0.75      %   Top dots
\put{$\bullet$} at  1 0.75      %
\put{$\bullet$} at  2 0.75      %
\put{$\bullet$} at  3 0.75      %   
\put{$\bullet$} at -3 -0.75          %
\put{$\bullet$} at -2 -0.75          %
\put{$\bullet$} at -1 -0.75          %
\put{$\bullet$} at  0 -0.75          %  Bottom dots
\put{$\bullet$} at  1 -0.75          %
\put{$\bullet$} at  2 -0.75          %
\put{$\bullet$} at  3 -0.75          %
\plot -5 0 4 0 /
%Flagpole
\setquadratic
\plot -4.5 1.25  -4.4 .9 -4 .6  / 
\plot  -4 .6 -3.4 .25 -3.25 0 /
\plot -4.25 1.25  -4.15 .9 -3.75 .6  / 
\plot  -3.75 .6 -3.15 .25 -3.0 0 /
\plot -4.5 -1.25  -4.4 -.9 -4 -.6  / 
\plot -4.25 -1.25  -4.15 -.9 -3.75 -.6  / 
%\plot  -4 -.6 -3.4 -.25 -3.25 0 /
%\plot  -3.75 -.6 -3.15 -.25 -3.0 0 /
\plot -3.25 0 -3.30 -.09  -3.40 -.17 /
\plot -3.00 0 -3.05 -.09 -3.12 -.17 /
\plot  -4.00   -.6 -3.70 -.5 -3.40 -.3 /
\plot  -3.75   -.6 -3.45  -.5 -3.15 -.3 /

\ellipticalarc axes ratio 1:1 360 degrees from -4.5 1.25 center 
at -4.375 1.25
\put{$*$} at -4.375 1.25  
\ellipticalarc axes ratio 1:1 180 degrees from -4.5 -1.25 center 
at -4.375 -1.25 
% Vertical edges
\setquadratic
\plot -3 0.75  -2.9  .55  -2.5 .375 /
\plot -2.5 .375  -2.1  .25  -2 0 /
\plot -2    0.75     -1.9  .55  -1.5 .375 /
\plot -1.5 0.375  -1.1  .25  -1 0 /
\plot -1    0.75     -0.9  .55  -0.5 .375 /
\plot -0.5 0.375  -0.1  .25  0  0 /
\plot 0    0.75     .1  .55  .5 .375 /
\plot .5  0.375   .9  .25  1  0 /
\plot -3 -0.75  -2.9  -.55  -2.5 -.380 /
\plot -2.3  -.26  -2.1  -.20  -2 0 /
\plot -2    -0.75     -1.9  -.55  -1.7 -.45 /
\plot -1.4 -0.28  -1.1  -.20  -1 0 /
\plot -1    -0.75   -0.9  -.6  -0.7 -.48 /
\plot -0.4 -0.32  -0.1  -.2  0  0 /
\plot 0    -0.75     .1  -.65  .2 -.55 /
\plot .55  -0.38   .9  -.25  1  0 /
\setlinear
\plot 2 .75  2 -.75 /
\plot 3 .75 3 -.75 /
%curve around pole
\setquadratic
\plot  1 .75 .9 .6  .4 .47 /
%\plot 1 .75  .5  .5 -1 .375 / 
%\plot -1 .375  -3.25  .25  -4 0 /
\plot -3.5 .22 -3.9  .15  -4 0 /
\plot -1 -.375  -3.25  -.25  -4 0 /
\plot 1 -.75  .5  -.5 -1 -.375 / 
\setlinear
\plot .1 .45  -.2 .4 /
\plot -.8 .38  -1.1 .37 /
\plot -1.7 .36 -2.1 .35 /
\plot -2.6 .33  -3 .29 /
\endpicture.
\end{equation}
the eigenvalues of $\Phi(X^{\varepsilon_i})$ are related to the eigenvalues of the Casimir.  The
\emph{Schur functors} are the functors
\begin{equation}\label{Schurfunctor}
\begin{matrix}
F^\lambda_{V} \colon &\{ \hbox{$U$-modules}\} &\longrightarrow &\{ \hbox{$\tilde\cB_k$-modules} \} \\
&M &\longmapsto &\Hom_U(M(\lambda), M\otimes V^{\otimes k})
\end{matrix}
\end{equation}
where $\Hom_U(M(\lambda), M\otimes V^{\otimes k})$ is the vector space
of highest weight vectors of weight $\lambda$ in $M\otimes V^{\otimes k}$.

\subsection{The quantum group $U_h\fgl_n$}

Although the Lie algebra $\fgl_n$ is reductive, not semisimple, all of the general
setup of Sections 3.1 and 3.2 can be applied without change.
The \emph{simple roots} are 
$\alpha_i = \varepsilon_i-\varepsilon_{i+1}$,  $1\le i\le n-1$,
and
\begin{equation}
\rho = (n-1)\varepsilon_1+(n-2)\varepsilon_2+\cdots+\varepsilon_{n-1}.
\end{equation}
The \emph{dominant integral weights} of $\fgl_n$ are
$$
\lambda = \lambda_1\varepsilon_1+\cdots+\lambda_n\varepsilon_n, 
\qquad\hbox{where}\quad
\lambda_1\ge \lambda_2\ge \cdots\ge \lambda_n, \quad\hbox{and}\quad
\lambda_1,\ldots, \lambda_n\in \ZZ.
$$
and these index the simple finite dimensional $U_h\fgl_n$-modules $L(\lambda)$.
A \emph{partition with $\le n$ rows} is a dominant integral weight with $\lambda_n\ge 0$. 
If $\lambda_n<0$ and $\Delta$ denotes the 1-dimensional ``determinant'' representation
of $U_h\fgl_n$ then (see \cite[\S 15.5]{FH})
\begin{equation}\label{detrep}
L(\lambda) \cong \Delta^{\lambda_n}\otimes L(\lambda+(-\lambda_n,\ldots, -\lambda_n))
\qquad\hbox{with $\lambda + (-\lambda_n,\ldots, -\lambda_n)$ a partition.}
\end{equation}

Identify each partition $\lambda$ with the configuration of boxes which has 
$\lambda_i$ boxes in row $i$.  For example,
\begin{equation}\label{ptnex}
\lambda = 
\beginpicture
\setcoordinatesystem units <0.25cm,0.25cm>         % sets scale
\setplotarea x from 0 to 4, y from -3 to 3    % sets plot size up
\linethickness=0.5pt                          % sets line thickness
\putrule from 0 3 to 5 3          %
\putrule from 0 2 to 5 2          %  draws horizontal lines
\putrule from 0 1 to 5 1          %
\putrule from 0 0 to 3 0          %
\putrule from 0 -1 to 3 -1          %
\putrule from 0 -2 to 1 -2          %
\putrule from 0 -3 to 1 -3          %
\putrule from 0 -3 to 0 3        %
\putrule from 1 -3 to 1 3        %
\putrule from 2 -1 to 2 3        %
\putrule from 3 -1 to 3 3        %  draws vertical lines
\putrule from 4 1 to 4 3        %
\putrule from 5 1 to 5 3        %
\endpicture
= 
5\varepsilon_1+5\varepsilon_2
+3\varepsilon_3+3\varepsilon_4+\varepsilon_5+\varepsilon_6.
\end{equation}
If $\mu$ and $\lambda$ are partitions with $\mu\subseteq \lambda$ (as collections of boxes)
then the \emph{skew shape} $\lambda/\mu$ is the collection of boxes of $\lambda$ that are
not in $\mu$.  For example, if $\lambda$ is as in \eqref{ptnex} and
$$\mu = \beginpicture
\setcoordinatesystem units <0.25cm,0.25cm>         % sets scale
\setplotarea x from 0 to 4, y from -3 to 3    % sets plot size up
\linethickness=0.5pt                          % sets line thickness
\putrule from 0 3 to 2 3          %
\putrule from 0 2 to 2 2          %  draws horizontal lines
\putrule from 0 1 to 2 1          %
\putrule from 0 0 to 2 0          %
\putrule from 0 -1 to 2 -1          %
%\putrule from 0 -2 to 1 -2          %
%\putrule from 0 -3 to 1 -3          %
\putrule from 0 -1 to 0 3        %
\putrule from 1 -1 to 1 3        %
\putrule from 2 -1 to 2 3        %
%\putrule from 3 -1 to 3 3        %  draws vertical lines
%\putrule from 4 1 to 4 3        %
%\putrule from 5 1 to 5 3        %
\endpicture
\qquad\hbox{then}\qquad
\lambda/\mu = 
\beginpicture
\setcoordinatesystem units <0.25cm,0.25cm>         % sets scale
\setplotarea x from 0 to 4, y from -3 to 3    % sets plot size up
\linethickness=0.5pt                          % sets line thickness
\putrule from 2 3 to 5 3          %
\putrule from 2 2 to 5 2          %  draws horizontal lines
\putrule from 2 1 to 5 1          %
\putrule from 2 0 to 3 0          %
\putrule from 0 -1 to 1 -1          %
\putrule from 2 -1 to 3 -1          %
\putrule from 0 -2 to 1 -2          %
\putrule from 0 -3 to 1 -3          %
\putrule from 0 -3 to 0 -1        %
\putrule from 1 -3 to 1 -1        %
\putrule from 2 -1 to 2 3        %
\putrule from 3 -1 to 3 3        %  draws vertical lines
\putrule from 4 1 to 4 3        %
\putrule from 5 1 to 5 3        %
\endpicture
.$$
If $b$ is a box in position $(i,j)$ of $\lambda$ 
the {\it content} of $b$ is
\begin{equation}\label{content}
c(b)=j-i = \hbox{the diagonal number of $b$,}
\qquad\hbox{so that}\qquad
\beginpicture
\setcoordinatesystem units <0.5cm,0.5cm>         % sets scale
\setplotarea x from 0 to 4, y from -3 to 3    % sets plot size up
\linethickness=0.5pt                          % sets line thickness
\putrule from 0 3 to 5 3          %
\putrule from 0 2 to 5 2          %  draws horizontal lines
\putrule from 0 1 to 5 1          %
\putrule from 0 0 to 3 0          %
\putrule from 0 -1 to 3 -1          %
\putrule from 0 -2 to 1 -2          %
\putrule from 0 -3 to 1 -3          %
\putrule from 0 -3 to 0 3        %
\putrule from 1 -3 to 1 3        %
\putrule from 2 -1 to 2 3        %
\putrule from 3 -1 to 3 3        %  draws vertical lines
\putrule from 4 1 to 4 3        %
\putrule from 5 1 to 5 3        %
\put{$\scriptstyle{0}$} at .5 2.5 
\put{$\scriptstyle{1}$} at 1.5 2.5 
\put{$\scriptstyle{2}$} at 2.5 2.5 
\put{$\scriptstyle{3}$} at 3.5 2.5 
\put{$\scriptstyle{4}$} at 4.5 2.5 
\put{$\scriptstyle{-1}$} at .5 1.5 
\put{$\scriptstyle{0}$} at 1.5 1.5 
\put{$\scriptstyle{1}$} at 2.5 1.5 
\put{$\scriptstyle{2}$} at 3.5 1.5 
\put{$\scriptstyle{3}$} at 4.5 1.5 
\put{$\scriptstyle{-2}$} at .5 .5 
\put{$\scriptstyle{-1}$} at 1.5 .5 
\put{$\scriptstyle{0}$} at 2.5 .5 
\put{$\scriptstyle{-3}$} at .5 -.5 
\put{$\scriptstyle{-2}$} at 1.5 -.5 
\put{$\scriptstyle{-1}$} at 2.5 -.5 
\put{$\scriptstyle{-4}$} at .5 -1.5 
\put{$\scriptstyle{-5}$} at .5 -2.5 
\endpicture
\end{equation}
are the contents of the boxes for the partition in \eqref{ptnex}.

If $\nu$ is a partition and
\begin{equation}\label{decomp}
V = L(\square)
\qquad\hbox{then}\qquad
L(\nu)  \otimes V  = \bigoplus_{\lambda\in \nu^+} L(\lambda),
\end{equation}
where the sum is over all partitions $\lambda$ with $\le n$ rows that are obtained by adding a box
to $\nu$  \cite[I App. A (8.4) and I (5.16)]{Mac},
Hence, the $U_h\fgl_n$-module
decompositions of  
\begin{equation}\label{gl-bimodule}
L(\mu)\otimes V^{\otimes k}
= \bigoplus_\lambda L(\lambda)\otimes \tilde H_k^{\lambda/\mu},
\qquad  k\in \ZZ_{\ge 0},
\end{equation}
are encoded by the graph $\hat H^{/\mu}$ with 
\begin{equation}
\begin{array}{ll}
\hbox{vertices on level $k$:}\  &\{\hbox{skew shapes $\lambda/\mu$ with $k$ boxes}\} \\
\hbox{edges:} & \hbox{$\lambda/\mu \longrightarrow \gamma/\mu$,
if $\gamma$ is obtained from $\lambda$ by adding a box} \\
\hbox{labels on edges:} &\hbox{content of the added box.}
\end{array}
\end{equation}
For example if $\mu = (3,3,3,2) = \beginpicture
\setcoordinatesystem units <0.13cm,0.13cm>   
\setplotarea x from 55 to 59.5, y from -3 to 3 
\plot 56 2.5     59 2.5    59 -0.5      58 -0.5      58 -1.5   56 -1.5    56 2.5 /
\endpicture$, then the first 4 rows of  $\hat H^{/\mu}$ are
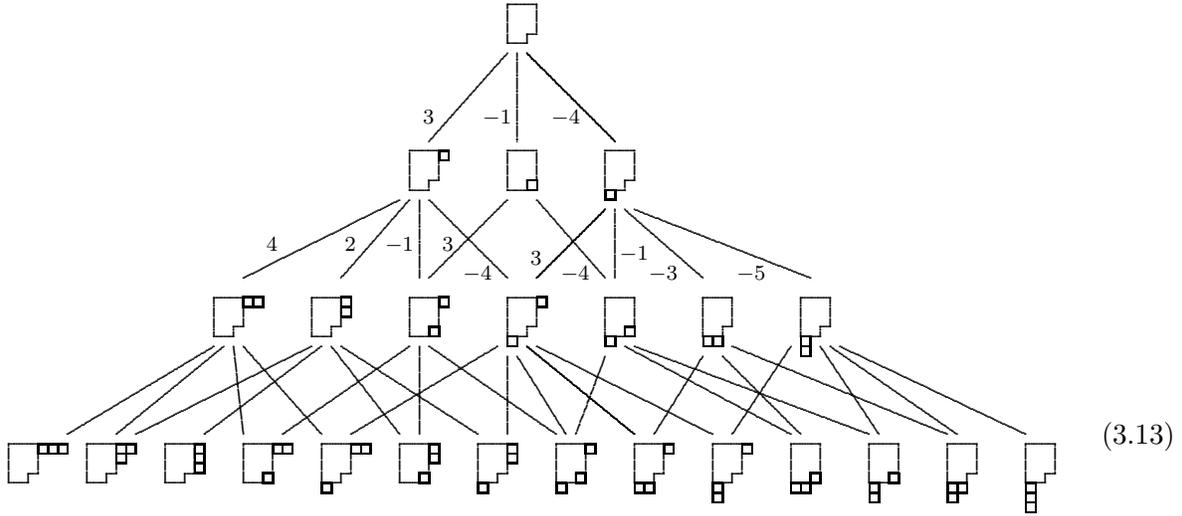
\begin{equation}\label{HkBratteli}
 {\beginpicture
\setcoordinatesystem units <0.13cm,0.13cm>         
\setplotarea x from -8 to 100, y from 0 to 45
\linethickness=0.5pt   
%%%%%%%%%%%%%%%%%%%%%%%%%%%%%%%%%%%%%%%%%%
\plot -8 0    -5 0      -5 -3      -6 -3    -6  -4   -8 -4   -8 0 /
\putrectangle corners at -5 0 and -4 -1
\putrectangle corners at -4 0 and -3 -1
\putrectangle corners at -3 0 and -2 -1

\plot 0 0       3 0      3 -3      2 -3      2 -4   0 -4    0 0 /
\putrectangle corners at 3 0 and 4 -1
\putrectangle corners at 4 0 and 5 -1
\putrectangle corners at 3 -1 and 4 -2

\plot 8 0     11 0    11 -3      10 -3      10 -4   8 -4    8 0 /
\putrectangle corners at 11 0  and 12 -1
\putrectangle corners at 11 -1  and 12 -2
\putrectangle corners at 11 -2 and 12 -3

\plot 16 0     19 0    19 -3      18 -3      18 -4   16 -4    16 0 /
\putrectangle corners at 19 0 and 20 -1
\putrectangle corners at 20 0 and 21 -1
\putrectangle corners at 18 -3 and 19 -4

\plot 24 0     27 0    27 -3      26 -3      26 -4   24 -4    24 0 /
\putrectangle corners at 27 0 and 28 -1
\putrectangle corners at 28 0 and 29 -1
\putrectangle corners at 24 -4 and 25 -5

\plot 32 0     35 0    35 -3      34 -3      34 -4   32 -4    32 0 /
\putrectangle corners at 35 0 and 36 -1
\putrectangle corners at 35 -1 and 36 -2
\putrectangle corners at 34 -3 and 35 -4

\plot 40 0     43 0    43 -3      42 -3      42 -4   40 -4    40 0 /
\putrectangle corners at 43 0 and 44 -1
\putrectangle corners at 43 -1 and 44 -2
\putrectangle corners at 40 -4 and 41 -5

\plot 48 0     51 0    51 -3      50 -3      50 -4   48 -4    48 0 /
\putrectangle corners at 51 0 and 52 -1
\putrectangle corners at 50 -3 and 51 -4
\putrectangle corners at 48 -4 and 49 -5

\plot 56 0     59 0    59 -3      58 -3      58 -4   56 -4    56 0 /
\putrectangle corners at 59 0 and 60 -1
\putrectangle corners at 57 -4 and 58 -5
\putrectangle corners at 56 -4 and 57 -5

\plot 64 0     67 0    67 -3      66 -3      66 -4   64 -4    64 0 /
\putrectangle corners at 67 0 and 68 -1
\putrectangle corners at 64 -4 and 65 -5
\putrectangle corners at 64 -5 and 65 -6

\plot 72 0     75 0    75 -3      74 -3      74 -4   72 -4    72 0 /
\putrectangle corners at 74 -3 and 75 -4
\putrectangle corners at 72 -4 and 73 -5
\putrectangle corners at 73 -4 and 74 -5

\plot 80 0     83 0    83 -3      82 -3      82 -4   80 -4    80 0 /
\putrectangle corners at 82 -3 and 83 -4
\putrectangle corners at 80 -4 and 81 -5
\putrectangle corners at 80 -5 and 81 -6

\plot 88 0     91 0    91 -3      90 -3      90 -4   88 -4    88 0 /
\putrectangle corners at 88 -4 and 89 -5
\putrectangle corners at 89 -4 and 90 -5
\putrectangle corners at 88 -5 and 89 -6

\plot 96 0     99 0    99 -3      98 -3      98 -4   96 -4    96 0 /
\putrectangle corners at 96 -4 and 97 -5
\putrectangle corners at 96 -5 and 97 -6
\putrectangle corners at 96 -6 and 97 -7

\plot 72 10 66 1 /
\plot 75 10 81 1 /
\plot 76 10 89 1 /
\plot 77 10 96 1 /

\plot 63 9 58 1 /
\plot 65 9 73 1 /
\plot 66 10 88 1 /

\plot 53 9 50 1 /
\plot 55 10 72 1 /
\plot 56 10 80 1 /

\plot 42 10 27 1 /
\plot 43 9 43 1 /
\plot 44 9 49 1 /
\plot 45 10 56 1 /
\plot 46 10 64 1 /

\plot 33 10 20 1 /
\plot 34 10 34 1 /
\plot 35 10 48 1 /

\plot 23 10 5 1 /
\plot 24 10 12 1 /
\plot 25 10 32 1 /
\plot 26 10 40 1 /

\plot  16 10 24 1 /
\plot  15 10 16 1 /
\plot  14 10 3 1 /
\plot  13 10 -2 1 /
%%%%%%%%%%%%%%%%%%%%%%%%%%%%%%%%%%%%%%%%%%
\plot 73 15  76 15  76 12    75 12      75  11  73 11  73 15 /
\putrectangle corners at 73 11 and 74 10
\putrectangle corners at 73 10 and 74 9
\plot 63 15  66 15  66 12    65 12      65  11  63 11  63 15 /
\putrectangle corners at 63 11 and 64 10
\putrectangle corners at 64 11 and 65 10
\plot 53 15  56 15  56 12    55 12      55  11  53 11  53 15 /
\putrectangle corners at 55 12 and 56 11
\putrectangle corners at 53 11 and 54 10
\plot 43 15  46 15  46 12    45 12      45  11  43 11  43 15 /
\putrectangle corners at 46 15 and 47 14
\putrectangle corners at 43 11 and 44 10
\plot 33 15  36 15  36 12    35 12      35  11  33 11  33 15 /
\putrectangle corners at 36 15 and 37 14
\putrectangle corners at 35 12 and 36 11
\plot 23 15  26 15  26 12    25 12      25  11  23 11  23 15 /
\putrectangle corners at 26 15 and 27 14
\putrectangle corners at 26 14 and 27 13
\plot 13 15  16 15  16 12    15 12      15  11  13 11  13 15 /
\putrectangle corners at 16 15 and 17 14 
\putrectangle corners at 17 15 and 18 14 

\plot 16 17 32 25 /
\put{${}^4$} at 19 20
\plot 26 17 33 25 /
\put{${}^2$} at 27 20
\plot 34 17 34 25 /
\put{${}^{-1}$} at 32 20
\plot 43 17 35 25 /
\put{${}^3$} at 37 20
\plot 43 25 35 17 /
\put{${}^{-4}$} at 40 17
\plot 46 25 53 17 /
\put{${}^3$} at 46 18.5
\plot 53 24 46 17 /
\put{${}^{-4}$} at 50 17
\plot 54 24 54 17 /
\put{${}^{-1}$} at 56 19
\plot 55 24 63 17 /
\put{${}^{-3}$} at 59 17
\plot 56 24 74 17 /
\put{${}^{-5}$} at 68 17
%%%%%%%%%%%%%%%%%%%%%%%%%%%%%%%%%%%%%%%%%%

\plot 53 30  56 30  56 27    55 27      55  26  53 26  53 30 /
\putrectangle corners at 53 26 and 54 25
\plot 43 30  46 30  46 27    45 27      45  26  43 26  43 30 /
\putrectangle corners at 45 27 and 46 26
\plot 33 30  36 30  36 27    35 27      35  26  33 26  33 30 /
\putrectangle corners at 36 30 and 37 29

\plot 43 40 35 31 /
\put{${}^3$} at 35 33
\plot 44 40 44 31 /
\put{${}^{-1}$} at 42 33
\plot 45 40 54 31 /
\put{${}^{-4}$} at 49 33
%%%%%%%%%%%%%%%%%%%%%%%%%%%%%%%%%%%%%%%%%%
\plot 43 45  46 45  46 42    45 42      45  41  43 41  43 45 /
%%%
%%%darken the lines
\linethickness=4pt
\plot 45 40 54 31 /
\plot 53 24 46 17 /
\plot 45 10 56 1 /

\endpicture}
\end{equation}

The following result is well known (see \cite{Ji} or \cite[(4.4)]{LR}).

\begin{prop}
If  $U = U_h\fgl_n$ and 
$V = L(\varepsilon_1)=L(
\beginpicture
\setcoordinatesystem units <0.13cm,0.13cm>   
\setplotarea x from 0 to 1, y from 0 to 1 
\putrectangle corners at 0 0 and 1 1
\endpicture)$
is the n-dimensional ``standard''
representation of $\fgl_n$ then the map $\Phi$ of \eqref{BraidRep} factors through the 
surjective homomorphism \eqref{AffineHeckeAlgebra} to give a representation of the affine Hecke algebra. 
\end{prop}

For a skew shape $\lambda/\mu$ with $k$ boxes
identify paths from $\mu$ to $\lambda/\mu$ in $\hat H^{/\mu}$
with \emph{standard tableaux}
of shape $\lambda/\mu$ by filling the boxes, successively, with $1,2,\ldots, k$ as they appear.
In the example graph $\hat H^{/\mu}$ above
$$
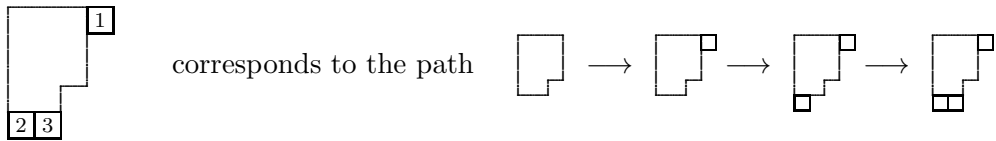
\beginpicture
\setcoordinatesystem units <0.35cm,0.35cm>         
\setplotarea x from 55 to 60, y from -3 to 3
\linethickness=0.5pt   
\plot 56 2.5     59 2.5    59 -0.5      58 -0.5      58 -1.5   56 -1.5    56 2.5 /
\putrectangle corners at 59 2.5 and 60 1.5
\putrectangle corners at 57 -1.5 and 58 -2.5
\putrectangle corners at 56 -1.5 and 57 -2.5
\put{$\scriptstyle{1}$} at 59.5 2
\put{$\scriptstyle{2}$} at 56.5 -2
\put{$\scriptstyle{3}$} at 57.5 -2
\endpicture
\qquad\hbox{corresponds to the path} \quad
\beginpicture
\setcoordinatesystem units <0.2cm,0.2cm>   
\setplotarea x from 55 to 60, y from -3 to 3 
\plot 56 2.5     59 2.5    59 -0.5      58 -0.5      58 -1.5   56 -1.5    56 2.5 /
%\plot 61 0 63 0 /
\endpicture
\longrightarrow
\beginpicture
\setcoordinatesystem units <0.2cm,0.2cm>   
\setplotarea x from 55 to 60, y from -3 to 3 
\plot 56 2.5     59 2.5    59 -0.5      58 -0.5      58 -1.5   56 -1.5    56 2.5 /
\putrectangle corners at 59 2.5 and 60 1.5
%\plot 61 0 63 0 /
\endpicture
\longrightarrow
\beginpicture
\setcoordinatesystem units <0.2cm,0.2cm>   
\setplotarea x from 55 to 60, y from -3 to 3 
\plot 56 2.5     59 2.5    59 -0.5      58 -0.5      58 -1.5   56 -1.5    56 2.5 /
\putrectangle corners at 59 2.5 and 60 1.5
\putrectangle corners at 56 -1.5 and 57 -2.5
%\plot 61 0 63 0 /
\endpicture
\longrightarrow
\beginpicture
\setcoordinatesystem units <0.2cm,0.2cm>   
\setplotarea x from 55 to 60, y from -3 to 3 
\plot 56 2.5     59 2.5    59 -0.5      58 -0.5      58 -1.5   56 -1.5    56 2.5 /
\putrectangle corners at 59 2.5 and 60 1.5
\putrectangle corners at 57 -1.5 and 58 -2.5
\putrectangle corners at 56 -1.5 and 57 -2.5
\endpicture
$$

\subsection{The quantum group $U_h\fgl_2$}

In the case when $n=2$, $U=U_h\fgl_2$ and the partitions which
appear in \eqref{gl-bimodule} and in the graph $\hat H^{/\mu}$ all have $\le 2$ rows.
For example if $\mu = (42) =
{\beginpicture   \setcoordinatesystem units <0.2cm,.2cm>       \setplotarea x from 0 to 4, y from -1 to 1   
\linethickness=0.5pt                      \plot 0 1.5    4 1.5    4 0.5    2 0.5    2 -.5    0 -.5  0 1.5  /
\endpicture}$
then the first few rows of $\hat H^{/\mu}$ are
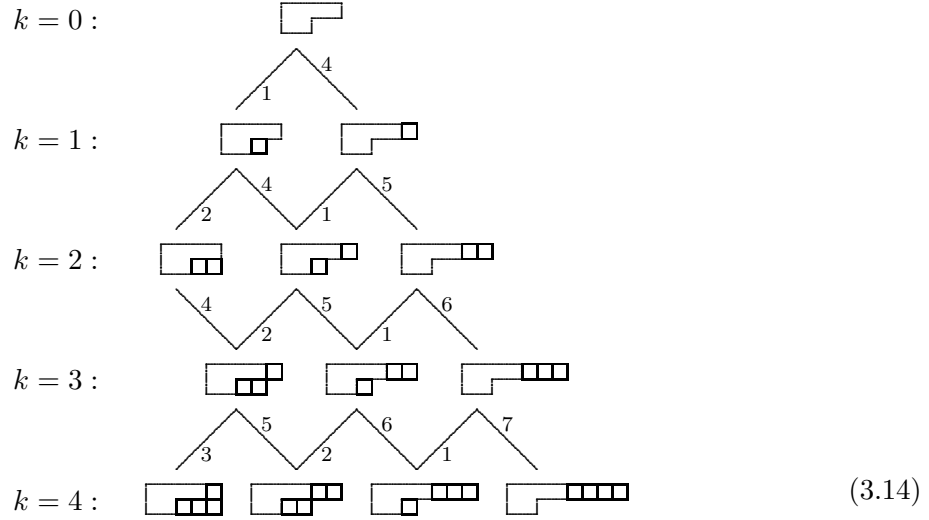
\begin{equation}\label{gltwoBratteli}
{\beginpicture
\setcoordinatesystem units <0.2cm,.2cm>         % sets scale
\setplotarea x from -8 to 27, y from -3  to 34   % sets plot size up
\linethickness=0.5pt                          % sets line thickness
%%%%%%%%%% level 4 %%%%%%%%%%%%%%%%%%%%%%%%%%%%%%%%%%%%%%%%%%%%
\put{$k = 4:$} at  -6 0.0
\plot 0 1    4 1    4 0    2 0    2 -1    0 -1  0 1 /
\putrectangle corners at 4 0 and 5 1
\putrectangle corners at 2 -1 and 3 0
\putrectangle corners at 3 -1 and 4 0
\putrectangle corners at 4 -1 and 5 0
\plot 7 1    11 1    11 0    9 0    9 -1    7 -1  7 1 /
\putrectangle corners at 11 0 and 12 1
\putrectangle corners at 12 0 and 13 1
\putrectangle corners at 9 -1 and 10 0
\putrectangle corners at 10 -1 and 11 0
\plot 15 1    19 1    19 0    17 0    17 -1    15 -1  15 1 /
\putrectangle corners at 19 0 and 20 1
\putrectangle corners at 20 0 and 21 1
\putrectangle corners at 21 0 and 22 1
\putrectangle corners at 17 -1 and 18 0
\plot 24 1    28 1    28 0    26 0    26 -1    24 -1  24 1 /
\putrectangle corners at 28 0 and 29 1
\putrectangle corners at 29 0 and 30 1
\putrectangle corners at 30 0 and 31 1
\putrectangle corners at 31 0 and 32 1
\plot 2 2 6 6 /  \put{$\scriptstyle{3}$} at 4 3
\plot 10 2 6 6 /  \put{$\scriptstyle{5}$} at 8 5
\plot 10 2 14 6 / \put{$\scriptstyle{2}$} at 12  3
\plot 14 6 18 2 / \put{$\scriptstyle{6}$} at 16 5
\plot 18 2 22 6 / \put{$\scriptstyle{1}$} at 20 3
\plot  22 6  26 2  / \put{$\scriptstyle{7}$} at 24 5
%%%%%%%%%% level 3 %%%%%%%%%%%%%%%%%%%%%%%%%%%%%%%%%%%%%%%%%%%%
\put{$k = 3:$} at  -6 8.0
\plot 4 9    8 9    8 8    6 8    6 7    4 7  4 9 /
\putrectangle corners at 8 8 and 9 9
\putrectangle corners at 6 7 and 7 8
\putrectangle corners at 7 7 and 8 8

\plot 12 9    16 9    16 8    14 8    14 7    12 7  12 9 /
\putrectangle corners at 16 8 and 17 9
\putrectangle corners at 17 8 and 18 9
\putrectangle corners at 14 7 and 15 8

\plot 21 9    25 9    25 8    23 8    23 7    21 7  21 9 /
\putrectangle corners at 25 8 and 26 9
\putrectangle corners at 26 8 and 27 9
\putrectangle corners at 27 8 and 28 9
\plot 6 10 2 14 /  \put{$\scriptstyle{4}$} at 4 13
\plot 6 10 10 14 /  \put{$\scriptstyle{2}$} at 8 11
\plot 10 14 14 10 /  \put{$\scriptstyle{5}$} at 12 13
\plot 14 10 18 14 /  \put{$\scriptstyle{1}$} at 16 11
\plot 18 14 22 10 /  \put{$\scriptstyle{6}$} at 20 13
%%%%%%%%%% level 2 %%%%%%%%%%%%%%%%%%%%%%%%%%%%%%%%%%%%%%%%%%%%
\put{$k = 2:$} at  -6 16.0
\plot 1 17    5 17    5 16    3 16    3 15    1 15  1 17 /
\putrectangle corners at 3 15 and 4 16
\putrectangle corners at 4 15 and 5 16
\plot 9 17    13 17    13 16    11 16    11 15    9 15  9 17 /
\putrectangle corners at 13 16 and 14 17
\putrectangle corners at 11 15 and 12 16
\plot 17 17    21 17    21 16    19 16    19 15    17 15  17 17 /
\putrectangle corners at 21 16 and 22 17
\putrectangle corners at 22 16 and 23 17
\plot 2 18 6 22 /  \put{$\scriptstyle{2}$} at 4 19
\plot 6 22 10 18 /  \put{$\scriptstyle{4}$} at 8 21
\plot 10 18 14 22 /  \put{$\scriptstyle{1}$} at 12 19
\plot 14 22 18 18 / \put{$\scriptstyle{5}$} at 16 21
%%%%%%%%%% level 1 %%%%%%%%%%%%%%%%%%%%%%%%%%%%%%%%%%%%%%%%%%%%
\put{$k = 1:$} at  -6 24.0
\plot 5 25    9 25    9 24    7 24    7 23    5 23  5 25 /
\putrectangle corners at 7 23 and 8 24
\plot 13 25    17 25    17 24    15 24    15 23    13 23  13 25 /
\putrectangle corners at 17 24 and 18 25
\plot 6 26 10 30 /   \put{$\scriptstyle{1}$} at 8 27
\plot 10 30 14 26  /   \put{$\scriptstyle{4}$} at 12 29
%%%%%%%%%% level 1 %%%%%%%%%%%%%%%%%%%%%%%%%%%%%%%%%%%%%%%%%%%%
\put{$k = 0:$} at  -6 32.0
\plot 9 33    13 33    13 32    11 32    11 31    9 31  9 33 /
\endpicture}
\end{equation}
and this
is the graph which describes the decompositions in \eqref{gl-bimodule}.

\begin{prop}\label{SchurWeylAffineTL}
If  $U = U_h\fgl_2$, $M=L(\mu)$ where $\mu$ is a partition of $m$ with $\le 2$ rows, and 
$V = L(\varepsilon_1)=L(
\beginpicture
\setcoordinatesystem units <0.13cm,0.13cm>   
\setplotarea x from 0 to 1, y from 0 to 1 
\putrectangle corners at 0 0 and 1 1
\endpicture)$
is the 2-dimensional ``standard''
representation of $\fgl_2$ then the map $\Phi$ of \eqref{BraidRep} factors through the
surjective homomorphism of \eqref{AffineToAffineTL} with $\alpha^2=-q^{2m-1}$ to give a representation
of the affine Temperley-Lieb algebra $T_k^a$.
\end{prop}
\begin{proof}
The proof that the kernel of $\Phi$ contains the element \eqref{extraTLrel}
is exactly as in the proof of \cite[Thm. 6.1(c)]{OR}:
The element $e_1$ in $T_2^a$ acts on $V^{\otimes 2}$ as 
$(q+q^{-1})\cdot {\rm pr}$ where ${\rm pr}$ is the unique $U_h\fg$-invariant 
projection onto 
$L( 
\beginpicture
\setcoordinatesystem units <0.13cm,0.13cm>   
\setplotarea x from 0 to 1, y from 0 to 2 
\putrectangle corners at 0 0 and 1 1
\putrectangle corners at 0 1 and 1 2
\endpicture)$
in $V^{\otimes 2}$.  Using $e_1T_1 = -q^{-1}e_1$ and the pictorial equalities
$$
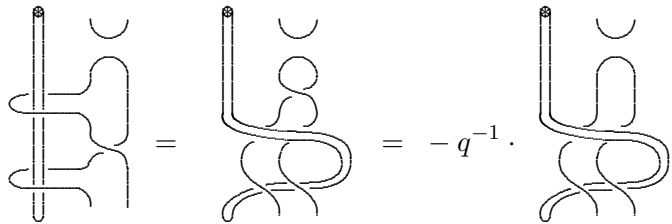
\beginpicture
\setcoordinatesystem units <.5cm,.5cm>         % sets scale
\setplotarea x from -4 to 1.5, y from -2 to 4    % sets plot size up
%Flagpole
\plot -1.5 3.7 -1.5 1.12 /
\plot -1.5 0.88 -1.5 -0.88 /
\plot -1.5 -1.12 -1.5 -1.75 /
\plot -1.25 3.7 -1.25 1.12 /
\plot -1.25 0.88 -1.25 -0.88 /
\plot -1.25 -1.12 -1.25 -1.75 /
\ellipticalarc axes ratio 1:1 360 degrees from -1.5 3.7 center 
at -1.375 3.7
\put{$*$} at -1.375 3.7  
\ellipticalarc axes ratio 1:1 180 degrees from -1.5 -1.75 center 
at -1.375 -1.75 
% Vertical edges
\plot  1 2   1 0.5 /
\plot  1 -0.5   1 -1.5 /
\setlinear
%first curve around pole
\plot -0.3 1.5  -1.1 1.5 /
\ellipticalarc axes ratio 2:1 180 degrees from -1.65 1.5  center 
at -1.65 1.25 
\plot -1.65 1  -0.3 1 /
%third curve around pole
\plot -0.3 -0.5  -1.1 -0.5 /
\ellipticalarc axes ratio 2:1 180 degrees from -1.65 -.5  center 
at -1.65 -.75 
\plot -1.65 -1  -0.3 -1 /
\setquadratic
\plot  -0.3 1  -0.05 0.8  -0 0.5 /
\plot   0 0.5   0.15 0.2   0.7 0 /
\plot   1 -0.5   0.9 -0.15   0.7 0 /
\plot  -0.3 1.5  -0.05 1.7  -0 2 /
\plot   -0.3 -0.5     -.1 -.425  -0.05 -0.325 /
\plot   -0.05 -0.325   0.15 -0.1   0.35 -0.05 /
\plot   0.65 0.15   0.9 0.25   1 0.5 /
\plot  -0.3 -1  -0.05 -1.2  -0 -1.5 /
%conditional expectation
\ellipticalarc axes ratio 1:1 180 degrees from 1 2 center 
at 0.5 2
\ellipticalarc axes ratio 1:1 180 degrees from 0 3.5 center 
at 0.5 3.5
\endpicture
=
\beginpicture
\setcoordinatesystem units <.5cm,.5cm>         % sets scale
\setplotarea x from -2.5 to 2.5, y from -2 to 4    % sets plot size up
%Flagpole
\plot -1.5 3.7 -1.5 1 /
%\plot -1.5 0.88 -1.5 -0.88 /
%\plot -1.5 -1.12 -1.5 -1.75 /
\plot -1.25 3.7 -1.25 1 /
%\plot -1.25 0.88 -1.25 -0.88 /
%\plot -1.25 -1.12 -1.25 -1.75 /
\ellipticalarc axes ratio 1:1 360 degrees from -1.5 3.7 center 
at -1.375 3.7
\put{$*$} at -1.375 3.7  
\ellipticalarc axes ratio 1:1 180 degrees from -1.5 -1.75 center 
at -1.375 -1.75 
% Vertical edges
%\plot  1 2   1 0.5 /
%\plot  1 -0.5   1 -1.5 /
\setlinear
%%first curve around pole
%\plot -0.3 1.5  -1.1 1.5 /
%\ellipticalarc axes ratio 2:1 180 degrees from -1.65 1.5  center 
%at -1.65 1.25 
%\plot -1.65 1  -0.3 1 /
%%third curve around pole
%\plot -0.3 -0.5  -1.1 -0.5 /
%\ellipticalarc axes ratio 2:1 180 degrees from -1.65 -.5  center 
%at -1.65 -.75 
%\plot -1.65 -1  -0.3 -1 /
\setquadratic
%\plot  -0.3 1  -0.05 0.8  -0 0.5 /
%positive crossing
\plot   0 2   0.15 1.7   0.7 1.5 /
\plot   1 1   0.9 1.35   0.7 1.5 /
\plot   0 1   0.1 1.3   0.35 1.45 /
\plot   0.65 1.65   0.9 1.75   1 2 /
%pole crossing
\plot   -1.5 1   -1.15 0.55   0.215 0.295 /
\plot   1.65 -0.2    1.25 0.2   0.215 0.295 /
\plot   -1.25 1   -0.9 0.7   0.205 0.515 /
\plot    1.9 -0.2    1.5 0.35   0.205 0.515 /
%single strings
\plot   0.65 0.65   0.9 0.75   1 1 /
\plot   -0.35 0.65  -0.1 0.75   0 1 /
\plot   -1 -0.2   -0.9 0.1   -0.65 0.25 /
\plot   -0 -0.2    0.1 0.1    0.25 0.2 /
%single strings back
\plot   0 -0.2   0.1 -0.6   0.5 -0.95 /
\plot   1 -1.7   0.9 -1.3   0.5 -0.95 /
\plot   -1 -0.2   -0.9 -0.6   -0.5 -0.95 /
\plot   0 -1.7   -0.1 -1.3   -0.5 -0.95 /
%pole back
\plot  1.9 -0.2    1.65 -0.8    0.8 -1 /
\plot  -1.25 -1.75  -1.05 -1.35  -0.45 -1.1 /   
\plot  -0.35 -0.9    -0.15 -0.875   0.2  -0.85 /
\plot  -0.2 -1.05    0.05 -1.05   0.4  -1.05 /
\plot  1.65 -0.2    1.45 -0.7   0.55  -0.85 /
\plot  -1.5 -1.75  -1.35 -1.35  -0.7 -1 /   
%\plot   -0.3 -0.5     -.1 -.425  -0.05 -0.325 /
%\plot  -0.3 -1  -0.05 -1.2  -0 -1.5 /
%conditional expectation
\ellipticalarc axes ratio 1:1 180 degrees from 1 2 center 
at 0.5 2
\ellipticalarc axes ratio 1:1 180 degrees from 0 3.5 center 
at 0.5 3.5
\endpicture
= ~-q^{-1}\cdot
\beginpicture
\setcoordinatesystem units <.5cm,.5cm>         % sets scale
\setplotarea x from -2 to 1.5, y from -2 to 4    % sets plot size up
%Flagpole
\plot -1.5 3.7 -1.5 1 /
%\plot -1.5 0.88 -1.5 -0.88 /
%\plot -1.5 -1.12 -1.5 -1.75 /
\plot -1.25 3.7 -1.25 1 /
%\plot -1.25 0.88 -1.25 -0.88 /
%\plot -1.25 -1.12 -1.25 -1.75 /
\ellipticalarc axes ratio 1:1 360 degrees from -1.5 3.7 center 
at -1.375 3.7
\put{$*$} at -1.375 3.7  
\ellipticalarc axes ratio 1:1 180 degrees from -1.5 -1.75 center 
at -1.375 -1.75 
% Vertical edges
\plot  0 2   0 1 /
\plot  1 2   1 1 /
\setlinear
\setquadratic
%\plot  -0.3 1  -0.05 0.8  -0 0.5 /
%positive crossing
%\plot   0 2   0.15 1.7   0.7 1.5 /
%\plot   1 1   0.9 1.35   0.7 1.5 /
%\plot   0 1   0.1 1.3   0.35 1.45 /
%\plot   0.65 1.65   0.9 1.75   1 2 /
%pole crossing
\plot   -1.5 1   -1.15 0.55   0.215 0.295 /
\plot   1.65 -0.2    1.25 0.2   0.215 0.295 /
\plot   -1.25 1   -0.9 0.7   0.205 0.515 /
\plot    1.9 -0.2    1.5 0.35   0.205 0.515 /
%single strings
\plot   0.65 0.65   0.9 0.75   1 1 /
\plot   -0.35 0.65  -0.1 0.75   0 1 /
\plot   -1 -0.2   -0.9 0.1   -0.65 0.25 /
\plot   -0 -0.2    0.1 0.1    0.25 0.2 /
%single strings back
\plot   0 -0.2   0.1 -0.6   0.5 -0.95 /
\plot   1 -1.7   0.9 -1.3   0.5 -0.95 /
\plot   -1 -0.2   -0.9 -0.6   -0.5 -0.95 /
\plot   0 -1.7   -0.1 -1.3   -0.5 -0.95 /
%pole back
\plot  1.9 -0.2    1.65 -0.8    0.8 -1 /
\plot  -1.25 -1.75  -1.05 -1.35  -0.45 -1.1 /   
\plot  -0.35 -0.9    -0.15 -0.875   0.2  -0.85 /
\plot  -0.2 -1.05    0.05 -1.05   0.4  -1.05 /
\plot  1.65 -0.2    1.45 -0.7   0.55  -0.85 /
\plot  -1.5 -1.75  -1.35 -1.35  -0.7 -1 /   
%\plot   -0.3 -0.5     -.1 -.425  -0.05 -0.325 /
%\plot  -0.3 -1  -0.05 -1.2  -0 -1.5 /
%conditional expectation
\ellipticalarc axes ratio 1:1 180 degrees from 1 2 center 
at 0.5 2
\ellipticalarc axes ratio 1:1 180 degrees from 0 3.5 center 
at 0.5 3.5
\endpicture
$$
it follows that $\Phi_2(e_1X^{\varepsilon_1}T_1X^{\varepsilon_1})$
acts as $-(q+q^{-1})q^{-1}\cdot \check R_{L(
\beginpicture
\setcoordinatesystem units <0.13cm,0.13cm>   
\setplotarea x from 0 to 1, y from 0 to 2 
\putrectangle corners at 0 0 and 1 1
\putrectangle corners at 0 1 and 1 2
\endpicture), L(\mu)}\check R_{L(\mu),L(
\beginpicture
\setcoordinatesystem units <0.13cm,0.13cm>   
\setplotarea x from 0 to 1, y from 0 to 2 
\putrectangle corners at 0 0 and 1 1
\putrectangle corners at 0 1 and 1 2
\endpicture)}(\id_{L(\mu)}\otimes {\rm pr})$.
By \eqref{casimir}, this is equal to
\begin{align*}
-q^{-1}(C_{L(\mu)}\otimes C_{L(
\beginpicture
\setcoordinatesystem units <0.13cm,0.13cm>   
\setplotarea x from 0 to 1, y from 0 to 2 
\putrectangle corners at 0 0 and 1 1
\putrectangle corners at 0 1 and 1 2
\endpicture)})&C_{L(\mu)\otimes L(
\beginpicture
\setcoordinatesystem units <0.13cm,0.13cm>   
\setplotarea x from 0 to 1, y from 0 to 2 
\putrectangle corners at 0 0 and 1 1
\putrectangle corners at 0 1 and 1 2
\endpicture)}^{-1}\Phi_2(\id_{L(\mu)}\otimes e_1) \\
&=-q^{-1}q^{-\langle \mu,\mu+2\rho\rangle}
q^{-\langle \varepsilon_1+\varepsilon_2,\varepsilon_1+\varepsilon_2+2\rho\rangle}
C_{L(\mu+\varepsilon_1+\varepsilon_2)}^{-1}
\Phi_2(\id_{L(\mu)}\otimes e_1).
\end{align*}
and the coefficient 
$-q^{-1}q^{-\langle \mu,\mu+2\rho\rangle}
q^{-\langle \varepsilon_1+\varepsilon_2,\varepsilon_1+\varepsilon_2+2\rho\rangle}
C_{L(\mu+\varepsilon_1+\varepsilon_2)}^{-1}$
simplifies to
$$-q^{-1}q^{-\langle \mu,\mu+2\rho\rangle}
q^{-\langle \varepsilon_1+\varepsilon_2,\varepsilon_1+\varepsilon_2+2\rho\rangle}
q^{\langle \mu+\varepsilon_1+\varepsilon_2,\mu+\varepsilon_1+\varepsilon_2+2\rho\rangle}
=-q^{-1}q^{2(\mu_1+\mu_2)} = -q^{2m-1},
$$
where $m = \mu_1+\mu_2 = |\mu|$.
\end{proof}

\subsection{The quantum group $U_h\fsl_2$}

The restriction of an irreducible representation $L(\lambda)$ of $U_h\fgl_n$ to $U_h\fsl_n$ 
is irreducible and all irreducible representations of $U_h\fsl_n$ are obtained in this fashion.
Since the ``determinant'' representation is trivial as an $U_h\fsl_n$ module
it follows from \eqref{detrep}
that the irreducible representations $L_{\fsl_n}(\lambda)$ of $U_h\fsl_n$ are indexed by partitions 
$\lambda = (\lambda_1,\ldots, \lambda_n)$ with $\lambda_n=0$.  
Hence, the graph which describes the $U_h\fsl_2$-module decompositions of  
\begin{equation}\label{sl-bimodule}
L(\mu)\otimes V^{\otimes k}
= \bigoplus_\lambda L(\lambda)\otimes \tilde T_k^{\lambda/\mu},
\qquad  k\in \ZZ_{\ge 0}.
\end{equation}
is exactly the same as the graph for $U_h\fgl_2$ except with all columns of length $2$ removed
from the partitions.  
More precisely, the decompositions
are encoded by the graph $\hat T^{/\mu}$ with 
\begin{equation}
\begin{array}{ll}
\hbox{vertices on level $k$:}\  &\{\mu_1-\mu_2+k, \mu_1-\mu_2+k-2, \ldots, \mu_1-\mu_2-k\}
\cap \ZZ_{\ge 0} \\
\hbox{edges:} & \ell \longrightarrow \ell\pm 1.
 \\
\end{array}
\end{equation}
For example if $m=7$ and $\mu_1-\mu_2 = 3$ then the first few rows of $\hat T^{/\mu}$ are
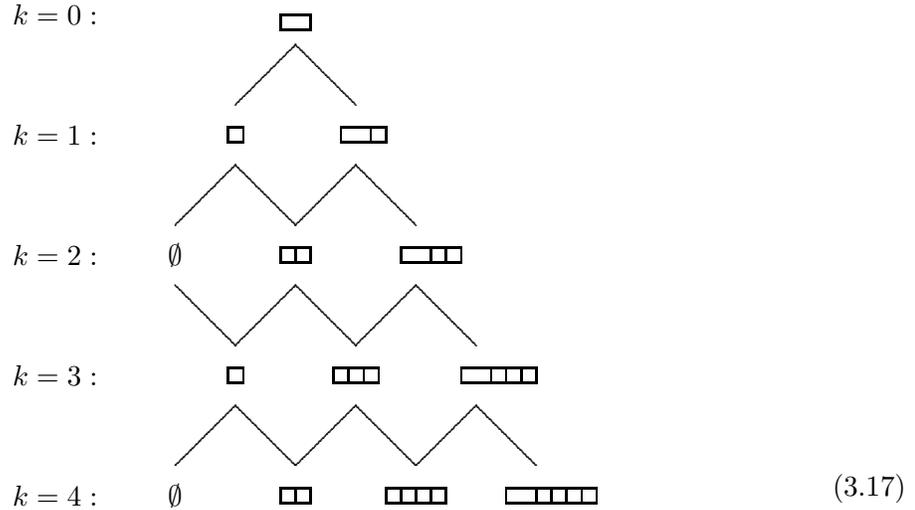
\begin{equation}\label{TkBratteli}
{\beginpicture
\setcoordinatesystem units <0.2cm,.2cm>         % sets scale
\setplotarea x from -8 to 27, y from -3  to 34   % sets plot size up
\linethickness=0.5pt                          % sets line thickness
%%%%%%%%%% level 4 %%%%%%%%%%%%%%%%%%%%%%%%%%%%%%%%%%%%%%%%%%%%
\put{$k = 4:$} at  -6 0.0
\put{$\emptyset$} at 2 0
\putrectangle corners at 9 -0.5 and 10 0.5
\putrectangle corners at 10 -0.5 and 11 0.5
\putrectangle corners at 16 -0.5 and 17 0.5
\putrectangle corners at 17 -0.5 and 18 0.5
\putrectangle corners at 18 -0.5 and 19 0.5
\putrectangle corners at 19 -0.5 and 20 0.5

\putrectangle corners at 24 -0.5 and 26 0.5
\putrectangle corners at 26 -0.5 and 27 0.5
\putrectangle corners at 27 -0.5 and 28 0.5
\putrectangle corners at 28 -0.5 and 29 0.5
\putrectangle corners at 29 -0.5 and 30 0.5

\plot 2 2 6 6 / 
% \put{$\scriptstyle{3}$} at 4 3
\plot 10 2 6 6 /  
%\put{$\scriptstyle{5}$} at 8 5
\plot 10 2 14 6 / 
%\put{$\scriptstyle{2}$} at 12  3
\plot 14 6 18 2 / 
%\put{$\scriptstyle{6}$} at 16 5
\plot 18 2 22 6 / 
%\put{$\scriptstyle{1}$} at 20 3
\plot  22 6  26 2  / 
%\put{$\scriptstyle{7}$} at 24 5
%%%%%%%%%% level 3 %%%%%%%%%%%%%%%%%%%%%%%%%%%%%%%%%%%%%%%%%%%%
\put{$k = 3:$} at  -6 8.0
\putrectangle corners at 5.5 7.5 and 6.5 8.5

\putrectangle corners at 12.5 7.5 and 13.5 8.5 
\putrectangle corners at 13.5 7.5 and 14.5 8.5 
\putrectangle corners at 14.5 7.5 and 15.5 8.5 

\putrectangle corners at 21 7.5 and 23 8.5
\putrectangle corners at 23 7.5 and 24 8.5
\putrectangle corners at 24 7.5 and 25 8.5
\putrectangle corners at 25 7.5 and 26 8.5

\plot 6 10 2 14 /  
%\put{$\scriptstyle{4}$} at 4 13
\plot 6 10 10 14 /  
%\put{$\scriptstyle{2}$} at 8 11
\plot 10 14 14 10 / 
%\put{$\scriptstyle{5}$} at 12 13
\plot 14 10 18 14 /  
%\put{$\scriptstyle{1}$} at 16 11
\plot 18 14 22 10 /  
%\put{$\scriptstyle{6}$} at 20 13
%%%%%%%%%% level 2 %%%%%%%%%%%%%%%%%%%%%%%%%%%%%%%%%%%%%%%%%%%%
\put{$k = 2:$} at  -6 16.0
\put{$\emptyset$} at 2 16.0
\putrectangle corners at 9 15.5 and 10 16.5
\putrectangle corners at 10 15.5 and 11 16.5
\putrectangle corners at 17 15.5 and 19 16.5
\putrectangle corners at 19 15.5 and 20 16.5
\putrectangle corners at 20 15.5 and 21 16.5

\plot 2 18 6 22 /  
%\put{$\scriptstyle{2}$} at 4 19
\plot 6 22 10 18 /  
%\put{$\scriptstyle{4}$} at 8 21
\plot 10 18 14 22 /  
%\put{$\scriptstyle{1}$} at 12 19
\plot 14 22 18 18 / 
%\put{$\scriptstyle{5}$} at 16 21
%%%%%%%%%% level 1 %%%%%%%%%%%%%%%%%%%%%%%%%%%%%%%%%%%%%%%%%%%%
\put{$k = 1:$} at  -6 24.0
\putrectangle corners at 5.5 23.5 and 6.5 24.5 
\putrectangle corners at 13 23.5 and 15 24.5
\putrectangle corners at 15 23.5 and 16 24.5
\plot 6 26 10 30 /  
%\put{$\scriptstyle{1}$} at 8 27
\plot 10 30 14 26  /   
%\put{$\scriptstyle{4}$} at 12 29
%%%%%%%%%% level 1 %%%%%%%%%%%%%%%%%%%%%%%%%%%%%%%%%%%%%%%%%%%%
\put{$k = 0:$} at  -6 32.0
\putrectangle corners at 9 31 and 11 32
\endpicture}
\end{equation}
Paths in \eqref{TkBratteli}  correspond to paths in \eqref{gltwoBratteli}
which correspond to standard tableaux $T$ of shape $\lambda/\mu$.
%: The path $p$ 
%adds a box at step $i$ if the number $i$ is in the first row of standard tableau $T$,
%and removes a box at step $i$ if $i$ is in the second row of $T$.

\end{section}

\begin{section}{Eigenvalues}%%%%%%%%%%%%%%%%%%%%%%%%%%

\subsection{Eigenvalues of the $X^{\varepsilon_i}$ in the affine Hecke algebra}

Recall, from \eqref{AffineHeckeAlgebra}, that the \emph{affine Hecke algebra} $\tilde H_k$ 
is the quotient of the group algebra of the affine braid group $\CC \tilde B_k$ by the relations
\begin{equation} %\label{AffineHeckeAlgebra}
T_{i}^{2}=(q-q^{-1})T_{i}+1.
\end{equation}
As observed in Proposition \ref{SchurWeylAffineTL} the map $\Phi$ in \eqref{BraidRep} makes
the module $L(\mu) \otimes V^{\otimes k}$ in \eqref{gl-bimodule} into an $\tilde H_k$ module.
Thus the vector spaces $\tilde H_k^{\lambda/\mu}$ in \eqref{gl-bimodule} are
the $\tilde H_k$-modules given by
$$\tilde H_k^{\lambda/\mu} = F^\lambda_{V}(L(\mu)),
\quad\hbox{where $F^\lambda_{V}$ are the Schur functors of \eqref{Schurfunctor}.}
$$
The following theorem is well known (see, for example, \cite{Ch}).

\begin{thm} \label{Thm:AffineMurphy} 
\begin{enumerate}
\item[(a)]  The $X^{\varepsilon_i}, 1 \le i \le k,$ mutually commute in the affine Hecke algebra $\tilde H_k$.
%$X^{\varepsilon_i}X^{\varepsilon_j}=X^{\varepsilon_j}X^{\varepsilon_i}$ for all $1\le i,j\le n$.
\item[(b)]  The eigenvalues of  $X^{\varepsilon_i}$ are given by the graph $\hat H^{/\mu}$ of 
\eqref{HkBratteli}
in the sense that if
\begin{align*}
\hat H^{/\mu}_k &= \{ \hbox{skew shapes $\lambda/\mu$ with $k$ boxes}\}
\qquad\hbox{and} \\
\hat H_k^{\lambda/\mu} &=
\{ \hbox{standard tableaux $T$ of shape $\lambda/\mu$}\}
\end{align*}
for $\lambda/\mu\in \hat H_k^{/\mu}$,
then
$$
\hbox{$\hat H_k^{/\mu}$ is an index set for the simple $\tilde H_k$ modules 
$\tilde H_k^{\lambda/\mu}$ appearing in $L(\mu)\otimes V^{\otimes k}$,}
$$
and
$$\tilde H_k^{\lambda/\mu}
\qquad\hbox{has a basis}\qquad \{v_T\ |\ T\in \hat H_k^{\lambda/\mu}\}
\qquad\hbox{with}\qquad
X^{\varepsilon_i} v_T = q^{2c(T(i))} v_T,$$ 
where $c(T(i))$ is the content of box $i$ of $T$.
\item[(c)] $\kappa
=X^{\varepsilon_1}\cdots X^{\varepsilon_k}$ is a central element of $\tilde H_k$
and 
$$\hbox{
$\kappa$ acts on 
$\tilde H_k^{\lambda/\mu}$ by the constant}\qquad
q^{2\sum_{b\in \lambda/\mu} c(b)}.
$$
\end{enumerate}
\end{thm}

\begin{proof} (a) is a restatement of \eqref{XiCommuteInAffine}.
(b) 
Since the $\tilde H_k$ action and the $U_h\fgl_n$ action commute on 
$L(\mu)\otimes V^{\otimes k}$ it follows that the decomposition in 
\eqref{gl-bimodule} is a decomposition as $(U_h\fgl_n, \tilde H_k)$ bimodules,
where the $\tilde H_k^{\lambda/\mu}$ are some $\tilde H_k$-modules.
Comparing the $L(\lambda)$ components on each side of
\begin{align*}
\bigoplus_\lambda L(\lambda)\otimes \tilde H_\ell^{\lambda/\mu}
&\cong L(\mu)\otimes V^{\otimes \ell} = L(\mu)\otimes V^{\otimes(\ell-1)}\otimes V 
\cong\big(\bigoplus_\nu L(\nu)\otimes\tilde  H_{\ell-1}^{\nu/\mu}\big)\otimes V \\
&\cong \bigoplus_\lambda \bigoplus_{\lambda/\nu=\square} 
L(\nu) \otimes \tilde H_{\ell-1}^{\nu/\mu}  
\cong \bigoplus_\lambda 
\Big(L(\lambda) \otimes 
\big(\bigoplus_{\lambda/\nu=\square} \tilde H_{\ell-1}^{\nu/\mu}\big)\Big)
\end{align*}
gives
\begin{equation}\label{Hdecomp}
\tilde H_\ell^{\lambda/\mu}\cong \bigoplus_{\lambda/\nu=\square} \tilde H_{\ell-1}^{\nu/\mu},
\end{equation}
for any $\ell\in \ZZ_{\ge 0}$ and skew shape $\lambda/\mu$ with $\ell$ boxes.
Iterate \eqref{Hdecomp} (with $\ell = k, k-1, \ldots$) to produce a decomposition
$$\tilde H_k^{\lambda/\mu} = \bigoplus_{T\in \hat H_k^{\lambda/\mu}} \tilde H_1^T,$$
where the summands $\tilde H_1^T$ are 1-dimensional vector spaces.  This determines
a basis (unique up to multiplication of the basis vectors by constants) 
$\{v_T\ |\ T\in \hat H_k^{\lambda/\mu}\}$  of $\tilde H_k^{\lambda/\mu}$ which respects the 
decompositions in \eqref{Hdecomp} for $1\le \ell \le k$.

Combining \eqref{casimir}, \eqref{fulltwist} and \eqref{XiasRmatrix} gives that 
$X^{\varepsilon_i}$ acts on the $L(\lambda)$ component of the decomposition
\eqref{decomp}
by the constant
$$
q^{\langle \lambda, \lambda + 2 \rho \rangle - \langle \nu, \nu + 2 \rho \rangle  
- \langle \varepsilon_1, \varepsilon_1 + 2 \rho \rangle} = q^{2c(\lambda/\nu)}
$$
since if $\lambda = \nu+\varepsilon_j$, so that $\lambda$ is the same as $\nu$ except with an
additional box in row $j$, then
$\nu\subseteq\lambda$,  $\lambda/\nu = \square$ and
\begin{align*}
\langle\lambda,\lambda+2\rho\rangle&-\langle\nu,\nu+2\rho\rangle
-\langle \varepsilon_1, \varepsilon_1+2\rho\rangle \\
&=\langle \nu+\varepsilon_j,\nu+\varepsilon_j+2\rho\rangle
-\langle\nu,\nu+2\rho\rangle -(1+2(n-1)) \\
&=2\nu_j+\langle \varepsilon_j, \varepsilon_j+2\rho\rangle - 2n+1 
=2\nu_j+(1+2(n-j)) - 2n+1 \\
&=2(\nu_j+1)-2j 
= 2c(\lambda/\nu).
\end{align*}
Hence,
$$X^{\varepsilon_i}v_T = q^{2c(T(i))}v_T,
\qquad\hbox{for $1\le i\le k$,}
$$
where $T(i)$ is the box containing $i$ in $T$.

The remainder of the proof, including the simplicity of the $\tilde H_k$-modules
$\tilde H_k^{\lambda/\mu}$, is accomplished as in \cite[Thm. 4.1]{R}.

(c) The element $X^{\varepsilon_1}\cdots X^{\varepsilon_k}$ is central in $\tilde B_k$
(it is a full twist) and hence its image is central in $\tilde H_k$.  The constant describing
its action on $\tilde H_k^{\lambda/\mu}$ follows from the formula 
$X^{\varepsilon_i}v_T = q^{2c(T(i))}v_T$.
\end{proof}

\subsection{Eigenvalues of the $m_i$ in $T_k^a$}

Let $m_1, m_2, \ldots, m_k$ be the commuting family in the affine Temperley-Lieb algebra
as defined in \eqref{midef}.  We will use the results of Theorem \ref{Thm:AffineMurphy} 
to determine the eigenvalues of the $m_i$ in the (generically) irreducible representations.

\begin{thm}  
\begin{enumerate}
\item[(a)]  The elements $m_i,1 \le i \le k$, mutually commute in $T_k^a$.
\item[(b)]  The eigenvalues of the elements $m_i$ are given by the graph $\hat T^{/\mu}$ of
\eqref{TkBratteli} in the sense that if the set of vertices on level $k$ is
\begin{enumerate}
\item[] $\hat T^{/\mu}_k = \{ \mu_1-\mu_2+k, \mu_1-\mu_2+k-2, \ldots, \mu_1-\mu_2-k\}
\cap \ZZ_{\ge 0}$, and
\item[] $\hat T_k^{\lambda/\mu} =
\{ \hbox{paths $p=(\mu=p^{(0)} \to p^{(1)} \to \cdots\to p^{(k)} = \lambda/\mu)$ 
to $\lambda/\mu$ in $\hat T^{/\mu}$}\}$,
\end{enumerate}
for $\lambda/\mu\in \hat T_k^{/\mu}$ then
$$
\hbox{$\hat T_k^{/\mu}$ is an index set for the simple $T_k^a$ modules $T_k^{\lambda/\mu}$ 
appearing in $L(\mu)\otimes V^{\otimes k}$,}
$$
and
$$T_k^{\lambda/\mu}\qquad\hbox{has a basis}\qquad \{v_p\ |\ p\in \hat T_k^{\lambda/\mu}\}$$
with
$$m_i v_p = \begin{cases}
\pm[p^{(i-1)}+1] v_p, &\hbox{if $p^{(i-1)}\pm 1 = p^{(i-2)} = p^{(i)}$,} \\
%\hbox{if \quad $p =  \quad
%\beginpicture
%\setcoordinatesystem units <0.175cm,0.15cm>         % sets scale
%\setplotarea x from 0 to 8, y from -10 to 10   % sets plot size up
%\linethickness=0.5pt                          % sets line thickness
%\put{$p^{(i-1)}$} at 0 0 %
%\put{$\scriptstyle{[\ell]}$}[l] at 3 0 %
%\plot 4.5 4.5  0.5 2 /
%\plot 4.5 -4.5  0.5 -2 /
%\put{$p^{(i)}$} at 5 -6 %
%\put{$p^{(i-2)}$} at 5.5 6.5 %
%\setdots
%\plot 5 10 5 7 /
%\plot 5 -10 5 -7 /
%\endpicture 
%$} \\
%-[\ell] v_p, &\hbox{if $p^{(i)}=p^{(i-2)}=p^{(i-1)}-1$,} \\
%\hbox{if \quad $p =  \quad
%\beginpicture
%\setcoordinatesystem units <0.175cm,0.15cm>         % sets scale
%\setplotarea x from -8 to 0, y from -10 to 10   % sets plot size up
%\linethickness=0.5pt                          % sets line thickness
%\put{$p^{(i-1)}$} at 0 0 %
%\put{$\scriptstyle{-[\ell]}$}[r] at -3 0 %
%\plot -4.5 -4.5  -0.5 -2 /
%\plot -4.5 4.5  -0.5 2 /
%\put{$p^{(i)}$} at -5 -6 %
%\put{$p^{(i-2)}$} at -5.5 6.5 %
%\setdots
%\plot -5 10 -5 7 /
%\plot -5 -10 -5 -7 /
%\endpicture
%$} \\ 
0, &\hbox{otherwise.}
%\hbox{if \quad $p=\quad \beginpicture
%\setcoordinatesystem units <0.175cm,0.15cm>         % sets scale
%\setplotarea x from -8 to 8, y from -10 to 10   % sets plot size up
%\linethickness=0.5pt                          % sets line thickness
%\put{$p^{(i-1)}$} at 0 0 %
%\plot 5 5  0.5 2 /
%\plot -5 -5  -0.5 -2 /
%\put{$p^{(i)}$} at -5 -7 %
%\put{$p^{(i-2)}$} at 5 7 %
%\setdots
%\plot 5 10 5 7 /
%\plot -5 -10 -5 -7 /
%\endpicture
%$ \quad or \quad
%$p= \quad
%\beginpicture
%\setcoordinatesystem units <0.175cm,0.15cm>         % sets scale
%\setplotarea x from -8 to 8, y from -10 to 10  % sets plot size up
%\linethickness=0.5pt                          % sets line thickness
%\put{$p^{(i-1)}$} at 0 0 %
%\plot -5 5  -0.5 2 /
%\plot  5 -5  0.5 -2 /
%\setdots
%\put{$p^{(i)}$} at -5 7 %
%\put{$p^{(i-2)}$} at 5 -7 %
%\plot 5 -10 5 -7 /
%\plot -5 10 -5 7 /
%\endpicture
%$}
\end{cases}
$$
where $p^{(i)}$ is the partition (a single part in this case) on level $i$ of the path $p$.
\item[(c)] $\kappa
=m_k+[2]m_{k-1}+\cdots+[k]m_1$ is a central element of $T_k^a$
and $\kappa$ acts on 
$T_k^{\lambda/\mu}$ by the constant
$$[k]\frac{q^{-(\mu_1+\mu_2)-(\mu_1-\mu_2)+1}}{q-q^{-1}}
+q^{-(\mu_1+\mu_2)}([\lambda_1-\lambda_2+2]+[\lambda_1-\lambda_2+4]+\cdots
%+[\mu_1-\mu_2+k-2]
+[\mu_1-\mu_2+k]).
$$
\end{enumerate}
\end{thm}

\begin{proof}
(a)  The elements $X^{\varepsilon_i}$ commute with one another in the affine Hecke algebra (see
(\ref{XiCommuteInAffine}) and the $m_j$ are by definition linear combinations of the 
$X^{\varepsilon_i}$ (see \ref{midef}), so they commute. 

(b)  Let $p$ be a path to $\lambda/\mu$ in $\hat T^{\mu}$ and let 
$T$ be the corresponding standard tableau on 2 rows.  
If $p^{(i)}=p^{(i-1)}-1=p^{(i-2)}-2$ or
If $p^{(i)}=p^{(i-1)}+1=p^{(i-2)}+2$ 
%If \quad $p=\quad \beginpicture
%\setcoordinatesystem units <0.175cm,0.15cm>         % sets scale
%\setplotarea x from -8 to 8, y from -10 to 10   % sets plot size up
%\linethickness=0.5pt                          % sets line thickness
%\put{$p^{(i-1)}$} at 0 0 %
%\plot 5 5  0.5 2 /
%\plot -5 -5  -0.5 -2 /
%\put{$p^{(i)}$} at -5 -7 %
%\put{$p^{(i-2)}$} at 5 7 %
%\setdots
%\plot 5 10 5 7 /
%\plot -5 -10 -5 -7 /
%\endpicture
%$ \quad or \quad
%$p= \quad
%\beginpicture
%\setcoordinatesystem units <0.175cm,0.15cm>         % sets scale
%\setplotarea x from -8 to 8, y from -10 to 10  % sets plot size up
%\linethickness=0.5pt                          % sets line thickness
%\put{$p^{(i-1)}$} at 0 0 %
%\plot -5 5  -0.5 2 /
%\plot  5 -5  0.5 -2 /
%\setdots
%\put{$p^{(i)}$} at -5 7 %
%\put{$p^{(i-2)}$} at 5 -7 %
%\plot 5 -10 5 -7 /
%\plot -5 10 -5 7 /
%\endpicture
%$
then $c(T(i-1)) =  c(T(i)) - 1$ and, 
from \eqref{midef} and Theorem \ref{Thm:AffineMurphy}(b), 
$$
m_i v_T =q^{i-2} \frac{q^{-2c(T(i))}-q^{-2}q^{-2c(T(i-1))}}{q-q^{-1}} v_T
= q^{i-2}\frac{q^{-2c(T(i))}-q^{-2}q^{-2c(T(i))+2}}{q-q^{-1}} 
v_T
= 0.
$$
If $p^{(i)}=p^{(i-2)}=p^{(i-1)}-1$
%If \quad $p =  \quad
%\beginpicture
%\setcoordinatesystem units <0.175cm,0.15cm>         % sets scale
%\setplotarea x from -8 to 0, y from -10 to 10   % sets plot size up
%\linethickness=0.5pt                          % sets line thickness
%\put{$p^{(i-1)}$} at 0 0 %
%%\put{$\scriptstyle{-[\ell]}$}[r] at -3 0 %
%\plot -4.5 -4.5  -0.5 -2 /
%\plot -4.5 4.5  -0.5 2 /
%\put{$p^{(i)}$} at -5 -6 %
%\put{$p^{(i-2)}$} at -5.5 6.5 %
%\setdots
%\plot -5 10 -5 7 /
%\plot -5 -10 -5 -7 /
%\endpicture
%$\quad 
with $T^{(i-1)} = (a,b)$
then $c(T(i))=a$ and $c(T(i-1))=b-2$ and
$$
m_i v_T 
= q^{i-2}\frac{q^{-2a}-q^{-2}q^{-2 b + 4}}{q-q^{-1}} v_T
= q^{i}\frac{q^{-(a+b+1)} \left( q^{-(a-b+1)} - q^{(a-b+1)}\right) }{q-q^{-1}} 
= -q^{-m}[a-b+1] v_T,
$$
where $m=|\mu|=a+b-i+1$.
If $p^{(i)}=p^{(i-2)}=p^{(i-1)}+1$
%If \quad $p =  \quad
%\beginpicture
%\setcoordinatesystem units <0.175cm,0.15cm>         % sets scale
%\setplotarea x from 0 to 8, y from -10 to 10   % sets plot size up
%\linethickness=0.5pt                          % sets line thickness
%\put{$p^{(i-1)}$} at 0 0 %
%%\put{$\scriptstyle{[\ell]}$}[l] at 3 0 %
%\plot 4.5 4.5  0.5 2 /
%\plot 4.5 -4.5  0.5 -2 /
%\put{$p^{(i)}$} at 5 -6 %
%\put{$p^{(i-2)}$} at 5.5 6.5 %
%\setdots
%\plot 5 10 5 7 /
%\plot 5 -10 5 -7 /
%\endpicture 
%$\quad 
with $T^{(i-1)} = (a,b)$
then $c(T(i-1))=a-1$ and $c(T(i))=b-1$ and
$$
m_i v_T 
=q^{i-2} \frac{q^{-2b+2}-q^{-2}q^{-2 a+2}}{q-q^{-1}} v_T
= q^{i}\frac{q^{-(a+b+1)} \left(q^{(a-b+1)} - q^{-(a-b+1)}\right) }{q-q^{-1}} 
= q^{-m}[a-b+1] v_T,
$$
where $m=|\mu|=a+b-i+1$.

(c)  Let $k = |\lambda/\mu|$. The identity  
$$q^{\lambda_1 + \lambda_2 -2}  \sum_{b \in \lambda/\mu} q^{-2 c(b) } \\
= \left( \sum_{i=\mu_2}^{\lambda_2-1} [\lambda_1 + \lambda_2 -2i] (q-q^{-1})  \right)  + [k] q^{\mu_2-\mu_1 + 1},  
$$
is best visible in an example:
With $\lambda = (10,6)$ and $\mu = (4,2)$,
\begin{align*}
&q^{16-2}\left(
\begin{array}{llllllllll}
\phantom{+}0&+0&+0&+0 &+q^{-8} &+q^{-10} &+q^{-12} &+q^{-14} &+q^{-16} &+q^{-18} \\
+0&+0 &+q^{-2} &+q^{-4} &+q^{-6} &+q^{-8}
\end{array}
\right) \\
&=
\begin{array}{llllllllll}
\phantom{+}0&+0&+0&+0 &+q^{6} &+q^{4} &+q^{2} &+q^{0} &+q^{-2} &+q^{-4} \\
+0&+0 &+q^{12} &+q^{10} &+q^{8} &+q^{6}
\end{array}
\\
&=
\left(
\begin{array}{llllllllll}
\phantom{+}0&+0&+0&+0 &+0 &+0 &+0 &+0 &+0 &+0 \\
+0&+0 &(q^{12}-q^{-12}) &+(q^{10}-q^{-10}) &+(q^{8}-q^{-8}) &+(q^{6}-q^{-6})
\end{array}
\right) \\
&\quad+\left(
\begin{array}{llllllllll}
\phantom{+}0&+0&+0&+0 &q^{-4} &+q^{-2} &+q^0 &+q^2 &+q^4 &+q^6 \\
+0&+0 &+q^{-12} &q^{-10} &+q^{-8} &+q^{-6}
\end{array}
\right) \\
&=\left(\sum_{i=2}^{6-1} q^{16-2i}-q^{-(16-2i)}\right)+[10]q^{4-2+1}.
\end{align*}
Then Proposition \ref{eqvforms} says 
$$
X^{-\varepsilon_1}+\cdots + X^{-\varepsilon_k}
= q^{-(k-2)}(q-q^{-1})(m_k+[2]m_{k-1}+\cdots + [k]m_1),
$$
and so $m_k+[2]m_{k-1}+\cdots + [k]m_1$ acts on $T_k^{\lambda/\mu}$ by the constant
\begin{align*}
(q-&q^{-1})^{-1}q^{k-2}\sum_{b \in \lambda/\mu} q^{-2 c(b) } =
(q-q^{-1})^{-1}q^{-(\mu_1+\mu_2) }q^{\lambda_1 + \lambda_2 -2} \sum_{b \in \lambda/\mu} q^{-2 c(b) } \\
&=(q-q^{-1})^{-1}q^{-(\mu_1+\mu_2) }
\left([k] q^{\mu_2-\mu_1 +1} 
+ \sum_{i=\mu_2}^{\lambda_2-1} [\lambda_1 + \lambda_2 -2i] (q-q^{-1})  \right) \\
&= [k]\frac{q^{-m-p^{(0)}+1}}{q-q^{-1}}
+\sum_{i=\mu_2}^{\lambda_2-1} q^{-m}[m+k-2i] \\
&= [k]\frac{q^{-m-p^{(0)}+1}}{q-q^{-1}}
+q^{-m}([p^{(k)}+2]+[p^{(k)}+4]+\cdots+[p^{(0)}+k-2]+[p^{(0)}+k]),
\end{align*}
since $\mu_1+\mu_2=m$, $\mu_1-\mu_2 = p^{(0)}$, $\lambda_1+\lambda_2=m+k$
and $\lambda_1-\lambda_2=p^{(k)}$.  

 \end{proof}

%\begin{remark}  Graham and Lehrer (Enseign. Math. 1998) have a beautiful formula that 
%expresses the minimal idempotent for the trivial representation in terms of Temperley-Lieb 
%diagrams. {\bf (I don't see how to get this idempotent
%from our Murphys.  Lehrer thinks that their formula only holds when q is a certain root of unity.)}
%Note that this element would be
%$$\prod_{i=3}^k (m_i-[i-1]) = (m_3-[2])(m_4-[3])\cdots.$$
%\end{remark}

\end{section}

\end{document}